\documentclass[11pt]{amsart}
\newcommand{\sfrac}[2]{{\textstyle\frac{#1}{#2}}}
\usepackage{amsmath,amssymb}
\usepackage{graphicx}  
\usepackage{color}

\numberwithin{equation}{section}

\newcommand{\Ex}{\mathbb{E}}
\renewcommand\Ex{\operatorname{\mathbb E}{}} 
 \renewcommand{\Pr}{{\mathbb{P}}}
 
 \newcommand{\FF}{\mathcal F}

 \newcommand{\eps}{\varepsilon}
 \newcommand{\bX}{\mathbf X}

 \newcommand{\bt}{\mathbf t}

 \newcommand\CTCS{\operatorname{CTCS}}
\newcommand\DTCS{\mathrm{DTCS}}
\newtheorem{Lemma}{Lemma}[section]
\newtheorem{Theorem}[Lemma]{Theorem}

\newtheorem{Corollary}[Lemma]{Corollary}

\theoremstyle{definition}
\newtheorem{Remark}[Lemma]{Remark}
\newtheorem{remark}[Lemma]{Remark}

\newenvironment{romenumerate}[1][-10pt]{
\addtolength{\leftmargini}{#1}\begin{enumerate}
 }{\end{enumerate}}

\newenvironment{alphenumerate}[1][-10pt]{
\addtolength{\leftmargini}{#1}\begin{enumerate}
 }{\end{enumerate}}

 \newcommand{\var}{\mathrm{var}}

\newcommand\nn{^{(n)}}
\newcommand\nni{^{(n+1)}}
\newcommand\nnn{^{[n]}}
\newcommand\bbN{\mathbb N}
\newcommand\bbR{\mathbb R}
\newcommand\bbC{\mathbb C}

\newcommand\setn{\set{1,\dots,n}}
\newcommand\set[1]{\ensuremath{\{#1\}}}
\newcommand\bigset[1]{\ensuremath{\bigl\{#1\bigr\}}}

\newcommand\bigpar[1]{\bigl(#1\bigr)}
\newcommand\Bigpar[1]{\Bigl(#1\Bigr)}

\newcommand\bigsqpar[1]{\bigl[#1\bigr]}

\newcommand\Bigsqpar[1]{\Bigl[#1\Bigr]}

\newcommand\bigabs[1]{\bigl\lvert#1\bigr\rvert}

\newcommand\KX{K}

\newcommand\XP{P}

\newcommand\ntoo{\ensuremath{{n\to\infty}}}

\newcommand\oi{\ensuremath{[0,1]}}

\newcommand\oio{\ensuremath{(0,1)}}
\newcommand\cXt{\XP_{t,1}}

\newcommand\gb{\beta}
\newcommand\gd{\delta}

\newcommand\gG{\Gamma}

\newcommand\gu{\upsilon}
\newcommand\gU{\Upsilon}

\newcommand\cL{{\mathcal L}}

\newcommand\cP{\mathcal P}

\newcommand\cR{{\mathcal R}}

\newcommand\cY{{\mathcal Y}}

\newcommand{\sumko}{\sum_{k=0}^\infty}

\newcommand{\suml}{\sum_{\ell=1}^\infty}

\newcommand\intoo{\int_0^\infty}
\newcommand\dd{\,\mathrm{d}}
\newcommand\ddx{\mathrm{d}}

\newcommand\eqd{\overset{\mathrm{d}}{=}}
\newcommand\intoi{\int_0^1}

\newcommand\qw{^{-1}}
\newcommand\Res{\operatorname{Res}}

\newcommand\nux{\nu_\Delta}
\newcommand\tnux{\widetilde{\nux}}

\newcommand\qww{^{-2}}
\newcommand\ii{\mathrm{i}}
\newcommand\xq{\widehat q}

\newcommand\Expo{\operatorname{Exp}}

\newcommand\one{\boldsymbol1}
\newcommand\onen{\one_{[n]}}
\newcommand\oneoo{\one_{[\infty]}}

\newcommand{\refT}[1]{Theorem~\ref{#1}}

\newcommand{\refC}[1]{Corollary~\ref{#1}}

\newcommand{\refL}[1]{Lemma~\ref{#1}}

\newcommand{\refR}[1]{Remark~\ref{#1}}

\newcommand{\refS}[1]{Section~\ref{#1}}

\newcommand{\refSS}[1]{Section~\ref{#1}}

\newcommand{\refF}[1]{Figure~\ref{#1}}
\newcommand{\refApp}[1]{Appendix~\ref{#1}}

\newcommand\marginal[1]{\marginpar[\raggedleft\tiny #1]{\raggedright\tiny#1}}

\newcommand\REM[1]{{\raggedright\texttt{[#1]}\par\marginal{XXX}}}

\newcommand\E{\Ex}

\newcommand\LL{\cL}
\newcommand\RR{\cR}
\newcommand\stopp{{\sf stop}}
\newcommand\bc{\mathbf c}
\newcommand\bs{\mathbf s}
\newcommand\bmu{\boldsymbol\mu}
\newcommand\bnu{\boldsymbol\nu}
\newcommand\xbnu{\widetilde{\boldsymbol\nu}}
\newcommand\brho{\boldsymbol\rho}
\newcommand\beps{\boldsymbol\epsilon^{(1)}}
\newcommand\cPm{\cP_{\mathfrak m}}

\newcommand\qup{q^\uparrow}

\begingroup
  \divide\count255 by 60
  \multiply\count255 by -60
  \advance\count255 by \time
  \ifnum \count255 < 10 \xdef\klockan{\the\count1.0\the\count255}
  \else\xdef\klockan{\the\count1.\the\count255}\fi
\endgroup

\begin{document}

\title[The Critical Beta-splitting Random Tree III]
{The Critical Beta-splitting Random Tree III: The exchangeable partition representation and the fringe tree} 

 \author{David J. Aldous}
\address{Department of Statistics,
 367 Evans Hall \#\  3860,
 U.C. Berkeley CA 94720}
\email{aldousdj@berkeley.edu}
\urladdr{www.stat.berkeley.edu/users/aldous.}

\author{Svante Janson}
\thanks{SJ supported by the Knut and Alice Wallenberg Foundation
and
the Swedish Research Council
}
\address{Department of Mathematics, Uppsala University, PO Box 480,
SE-751~06 Uppsala, Sweden}
\email{svante.janson@math.uu.se}
\newcommand\urladdrx[1]{{\urladdr{\def~{{\tiny$\sim$}}#1}}}
\urladdrx{www2.math.uu.se/~svante}

\date{December 11, 2024} 

 \begin{abstract}
 In the critical beta-splitting model of a random $n$-leaf rooted tree, clades are recursively split into sub-clades, and a clade of $m$ leaves is split into sub-clades 
 containing  $i$ and $m-i$ leaves  with probabilities $\propto 1/(i(m-i))$. 
 Study of structure theory and explicit quantitative aspects of the model is an active research topic. 
 It turns out that many results have several different proofs, and  detailed studies of analytic
 proofs are given in \cite{beta1} (via analysis of recursions) and \cite{beta4} (via Mellin transforms).
 This article describes two core probabilistic methods for studying $n \to \infty$ asymptotics of the basic finite-$n$-leaf models.

 (i) There is a canonical embedding into a continuous-time model, that is a random tree  $\CTCS(n)$ on $n$ leaves with real-valued edge lengths, and this model
 turns out to be more convenient to study. 
 The family $(\CTCS(n), n \ge 2)$ is consistent under a ``delete random leaf and prune" operation.
  That leads to an explicit inductive construction (the {\em growth algorithm}) of $(\CTCS(n), n \ge 2)$  as $n$ increases, 
  and then to a limit structure $\CTCS(\infty)$ which can be formalized via exchangeable partitions, 
  in some ways analogous to the Brownian continuum random tree.
  
  (ii) There is an explicit description of the limit {\em fringe distribution} relative to a random leaf,
   whose graphical representation  is essentially the format of the cladogram 
   representation of biological phylogenies.
 \end{abstract}

\maketitle
  

{\tt {\bf Keywords:}
Exchangeable partition, fringe tree, Markov chain, phylogenetic tree, random tree, subordinator.

MSC 60C05; 05C05, 60G09, 92B10}

\tableofcontents


\section{Introduction}

This article is part of a broad project \cite{beta2-arxiv,beta4,HDchain,beta1}
studying a certain random tree model.  The model is defined by recursively splitting a given set of leaves such that 
a set of $m$ leaves is split into subsets containing  $i$ and $m-i$ leaves  with probabilities proportional to 
$1/i(m-i)$; see \refS{sec:tree} for details.
The model arose \cite{me_clad} as a toy model for phylogenetic trees, designed to mimic the
uneven splits observed in real world examples  
(see  \refS{sec:clad2}). 
The model turns out to have a rich mathematical structure.
There are many questions one can ask and many different proof techniques can be exploited. 
Indeed several of the key results each have several quite different proofs, a fact which may be of pedagogical interest.

This article provides an introduction to the model, emphasizing connections with previous work,
and describes core probabilistic methods for studying $n \to \infty$ asymptotics of the basic finite-$n$-leaf models.
A detailed technical study of some aspects via the analysis of recursions is given in \cite{beta1},
and another detailed technical study of other aspects via Mellin transforms will be given in \cite{beta4}.
In parallel, a document \cite{beta2-arxiv} will be maintained, to summarize known results and provide more heuristics and open problems.

We do find it convenient to adopt the biological term \emph{clade} for the
set of leaves in a subtree, 
that is the elements in a subset somewhere in the splitting
process. There is an obvious correspondence between the clades and the
total $2n-1$ nodes of the binary tree, where leaves correspond to the $n$ clades of size 1
and internal nodes to the $n-1$ larger clades.

\subsection{Outline of results}
The most studied aspects of the model have been centered on the  CLT for leaf heights, proved by different methods in \cite{beta1,iksanovCLT,kolesnik}, with 
 further related ``height" results in \cite{beta4,beta1}. Basic such results are mentioned in Section \ref{sec:LH}, but this article is essentially independent of those results. 
 
 \noindent
 \begin{itemize}
 \item In Section \ref{sec:consistency} we describe the consistency property (Theorem \ref{T:consistent}) for $n$-leaf trees and the resulting representation of a 
 limit tree  $\CTCS(\infty)$ via an exchangeable random partition of $\bbN$.
 
 \item For finite $n$ the ``number of subclades along a path to a uniform random leaf" is a certain continuous-time Markov chain that we call the
 {\em harmonic descent} chain (Section \ref{sec:HD}).  
 The probability of visiting a given state (subclade size) $i$ has an explicit formula $a(i)$ in the $n \to \infty$ limit.
 This {\em occupation measure} Theorem \ref{T:alimit} (Section \ref{sec:OP}) has been proved originally in \cite{HDchain} and then \cite{iksanovHD}.
 
 \item This leads in Section \ref{sec:exch} to an exact description of the ``number of subclades along a path to a uniform random leaf on the infinite boundary" process within $\CTCS(\infty)$, 
 in terms of a certain subordinator (Theorem \ref{T:exact}).
 \item In Section \ref{sec:OPfringe} we observe that  the ``occupation measure" Theorem \ref{T:alimit} leads to an explicit description (Theorem \ref{Tqup}) 
 of the asymptotic {\em fringe tree},  many of whose properties have yet to be investigated.
 The fringe tree is essentially the way that real-world phylogenies are drawn as {\em cladograms}, and we illustrate a real example alongside a realization
of our model. 
\item In Section \ref{sec:surprise} we give a novel third proof of the  occupation measure theorem, as a first indication of the power of  Mellin transform methods.
\item In Section \ref{sec:CRT} we discuss analogies with the Brownian continuum random tree.
 \end{itemize}

 \section{The critical beta-splitting model of random trees}
  \label{sec:tree}

In this section we give the basic construction of the critical beta-splitting random
tree.  In fact we will give several different versions:
we define a discrete-time version $\DTCS(n)$ 
and a continuous-time version $\CTCS(n)$;
furthermore, in both cases we define ordered and unordered versions.
Moreover, for the unordered version of $\CTCS(n)$,
we define an explicit growth process 
$(\CTCS(n), n = 1,2,3,\ldots)$ that constructs $\CTCS(n)$ for all $n$ 
jointly in a natural way by adding leaves one by one.

\subsection{The ordered versions}\label{SSordered}
For $m \ge 2$, consider the probability distribution 
$(q(m,i),\ 1 \le i \le m-1)$
constructed to be proportional to $\frac{1}{i(m-i)}$.
 Explicitly 
\begin{equation}\label{01}
q(m,i)=\frac{m}{2h_{m-1}}\cdot\frac{1}{i(m-i)}
=\frac{1}{2h_{m-1}}\Bigpar{\frac{1}{i}+\frac{1}{m-i}},
\qquad 1\le i\le m-1,
\end{equation}
where  $h_{m-1}$ is the harmonic sum 
\begin{align}
h_{m-1}:=\sum_{i=1}^{m-1}\frac1i.   
\end{align}
Now fix $n \ge 2$.
Consider the process of constructing a random binary tree with $n$ leaves,
labelled $1,\dots,n$, by recursively
splitting the integer interval 
$[n] := \{1,2,\ldots,n\}$ of leaves as follows.
First specify that there is a left edge and a right edge at the root,
leading to a left subtree which will have the $\LL_n$ leaves
$\{1,\ldots,\LL_n\}$ 
and a right subtree which will have the $\RR_n = n - \LL_n$ 
leaves $\{\LL_n + 1,\ldots, n\}$, where $\LL_n$ 
 (and also $\RR_n$, by symmetry) has distribution $q(n,\cdot)$. 
 Recursively, a subinterval with $m\ge 2$ leaves is split into two
 subintervals of random size 
  from the distribution $q(m,\cdot)$. 
  Continue until reaching intervals of size $1$, which are the leaves.
This yields a binary tree with the given $n$ leaves; each of the $n-1$
splits corresponds to an internal node (including the root).
For completeness, we also include the case $n=1$, in which there are no splits
and the tree just consists of the root.
Figure \ref{Fig:1d} (left) illustrates schematically the construction as
interval-splitting.

For our purposes, it will be convenient to draw the binary trees in a
non-standard way, shown in \refF{Fig:1d} (center and right).
Instead of drawing two edges from an internal node to its children as usual, 
we draw one vertical line to a ``branchpoint'' 
(representing the split but \emph{not} regarded as a node in
the tree) followed by two horizontal lines to the children. 
We regard the horizontal lines as having length 0; 
we may (for obvious practical reasons) draw them with arbitrary
lengths in the figures, but these lengths have no significance.
In \refF{Fig:1d} (right), the horizontal lines to leaves are drawn with
their true length 0.

We regard the splitting process as evolving in time, 
and consider two versions, one in discrete time and one in continuous time,
which we call $\DTCS(n)$ and $\CTCS(n)$, respectively.%
\footnote{DTCS and CTCS are
  abbreviations for Discrete Time  Critical Splitting and Continuous Time
  Critical Splitting, for reasons explained in Section \ref{sec:motive}.}
In $\DTCS(n)$, the root clade splits at time 1, its children at time 2, and
so on. 
In $\CTCS(n)$, a clade with  $m \ge 2$ leaves is split at rate $h_{m-1}$,
that is after an $\Expo(h_{m-1})$ random time (independent of everything
else).%
\footnote{$\Expo(r)$ denotes a random variable with an exponential
  distribution with rate $r$, and thus expectation $1/r$.}
In both versions, we start at time 0.
Each node is regarded as born at the time the corresponding clade appears;
we also regard this time as the height of the node in the tree.
Hence, in the $\DTCS(n)$, all edges have length $1$ and the height equals
the usual graph distance to the root; in $\CTCS(n)$ the edges have
different, random, lengths. 

In other words, in our graphical representation of the binary tree, each clade with size $>1$
is represented by one vertical line (and conversely); 
this line thus has length 1 in $\DTCS(n)$
and has length $\Expo(h_{m-1})$ for a clade of size $m$ in $\CTCS(n)$.

Recall that apart from edge lengths, $\DTCS(n)$
and $\CTCS(n)$ define the same binary tree; in particular, we can always
recover $\DTCS(n)$ from $\CTCS(n)$ by ignoring the edge lengths.
Figure \ref{Fig:1c} shows a schematic realization of $\CTCS(20)$ as a
``continuization" of the realization of $\DTCS(20)$ in Figure \ref{Fig:1d}.
Figure \ref{Fig:400} shows an actual realization of $\CTCS(400)$.
Figure \ref{Fig:1c} shows also, as an example,
a distinguished leaf (11); 
the path from the root to the distinguished leaf 11 
passes through successive clades
$$
[[1,20]], \ [[4,20]], \ [[5,20]], \ [[9,20]], \ [[9,19]], \ [[9,17]],
[[9,13]], \ [[9,11]], \ [[11]]
$$
which have successive sizes (number of leaves) 
$20,17,16,12,11,9,5,3,1$.

The choice of rate $h_{m-1}$ in the definition of $\CTCS(n)$ may seem
arbitrary, but it is justified by the consistency result below
(Theorem \ref{T:consistent}, see also \eqref{f1} in its proof),  which 
suggests that $h_{m-1}$ is the canonical choice of splitting rates for the
continuization. 
Note that we, equivalently, can say that a clade of size $m\ge2$  
splits into two clades of sizes $i$ and $m-i$ 
(taking, as always below, the left subclade first for definiteness)
with rate
\begin{align}\label{ok1}
\xq(m,i):=  h_{m-1}q(m,i) =
\frac{m}{2i(m-i)}= \frac{1}{2i}+\frac{1}{2(m-i)},
\end{align}
for every $1\le i\le m-1$.

Regarding terminology, remember that ``time" and ``height"%
\footnote{Or depth, if one draws trees upside-down.}
are the same.
Within the mathematical analysis of random processes we generally follow the
usual ``time" convention, 
while in stating results we generally use the tree-related terminology of
``height". 

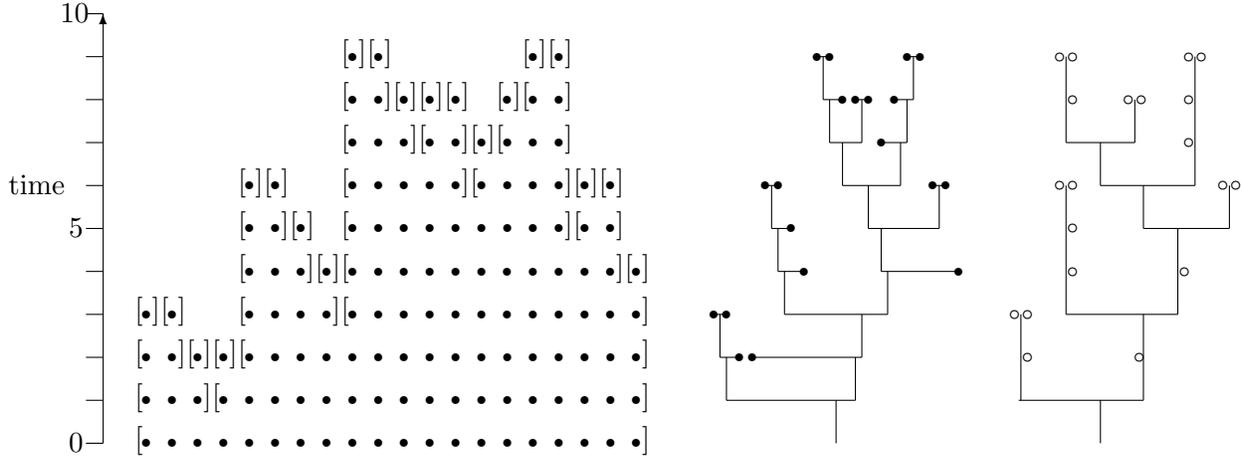
\begin{figure}[t]
\setlength{\unitlength}{0.045in}
\begin{picture}(160,50)
\put(-5,0){\vector(0,1){50}}
\put(-9,-1){0}
\put(-9,24){5}
\put(-10,49){10}
\put(-16,29){time}
\multiput(-5,0)(0,5){11}{\line(-1,0){2}}
\multiput(0,0)(3,0){20}{\circle*{0.9}}
\put(-1.3,-0.5){[}
\put(57.7,-0.5){]}
\multiput(0,5)(3,0){20}{\circle*{0.9}}
\put(-1.3,4.5){[}
\put(6.7,4.5){]}
\put(7.7,4.5){[}
\put(57.7,4.5){]}
\multiput(0,10)(3,0){20}{\circle*{0.9}}
\put(-1.3,9.5){[}
\put(3.7,9.5){]}
\put(4.7,9.5){[}
\put(6.7,9.5){]}
\put(7.7,9.5){[}
\put(9.7,9.5){]}
\put(10.7,9.5){[}
\put(57.7,9.5){]}
\multiput(0,15)(3,0){2}{\circle*{0.9}}
\put(-1.3,14.5){[}
\put(0.7,14.5){]}
\put(1.7,14.5){[}
\put(3.7,14.5){]}
\multiput(12,15)(3,0){16}{\circle*{0.9}}
\put(10.7,14.5){[}
\put(21.7,14.5){]}
\put(22.7,14.5){[}
\put(57.7,14.5){]}
\multiput(12,20)(3,0){16}{\circle*{0.9}}
\put(10.7,19.5){[}
\put(18.7,19.5){]}
\put(19.7,19.5){[}
\put(21.7,19.5){]}
\put(22.7,19.5){[}
\put(54.7,19.5){]}
\put(55.7,19.5){[}
\put(57.7,19.5){]}
\multiput(12,25)(3,0){3}{\circle*{0.9}}
\put(10.7,24.5){[}
\put(15.7,24.5){]}
\put(16.7,24.5){[}
\put(18.7,24.5){]}
\multiput(24,25)(3,0){11}{\circle*{0.9}}
\put(22.7,24.5){[}
\put(48.7,24.5){]}
\put(49.7,24.5){[}
\put(54.7,24.5){]}
\multiput(12,30)(3,0){2}{\circle*{0.9}}
\put(10.7,29.5){[}
\put(12.7,29.5){]}
\put(13.7,29.5){[}
\put(15.7,29.5){]}
\multiput(24,30)(3,0){11}{\circle*{0.9}}
\put(22.7,29.5){[}
\put(36.7,29.5){]}
\put(37.7,29.5){[}
\put(48.7,29.5){]}
\put(49.7,29.5){[}
\put(51.7,29.5){]}
\put(52.7,29.5){[}
\put(54.7,29.5){]}
\multiput(24,35)(3,0){9}{\circle*{0.9}}
\put(22.7,34.5){[}
\put(30.7,34.5){]}
\put(31.7,34.5){[}
\put(36.7,34.5){]}
\put(37.7,34.5){[}
\put(39.7,34.5){]}
\put(40.7,34.5){[}
\put(48.7,34.5){]}
\multiput(24,40)(3,0){5}{\circle*{0.9}}
\put(22.7,39.5){[}
\put(27.7,39.5){]}
\put(28.7,39.5){[}
\put(30.7,39.5){]}
\put(31.7,39.5){[}
\put(33.7,39.5){]}
\put(34.7,39.5){[}
\put(36.7,39.5){]}
\multiput(42,40)(3,0){3}{\circle*{0.9}}
\put(40.7,39.5){[}
\put(42.7,39.5){]}
\put(43.7,39.5){[}
\put(48.7,39.5){]}
\multiput(24,45)(3,0){2}{\circle*{0.9}}
\put(22.7,44.5){[}
\put(24.7,44.5){]}
\put(25.7,44.5){[}
\put(27.7,44.5){]}
\multiput(45,45)(3,0){2}{\circle*{0.9}}
\put(43.7,44.5){[}
\put(45.7,44.5){]}
\put(46.7,44.5){[}
\put(48.7,44.5){]}
\put(66,15){\line(1,0){1.5}}
\put(66.75,15){\line(0,-1){5}}
\put(66.75,10){\line(1,0){2.25}}
\put(67.5,10){\line(0,-1){5}}
\put(82.5,10){\line(0,-1){5}}
\put(67.5,5){\line(1,0){15}}
\put(83.25,15){\line(0,-1){5}}
\put(70.5,10){\line(1,0){12.75}}
\put(74.25,20){\line(0,-1){5}}
\put(86.25,20){\line(0,-1){5}}
\put(74.25,15){\line(1,0){12}}
\put(73.5,25){\line(0,-1){5}}
\put(73.5,20){\line(1,0){3}}
\put(85.5,25){\line(0,-1){5}}
\put(85.5,20){\line(1,0){9}}
\put(72.75,30){\line(0,-1){5}}
\put(72.75,25){\line(1,0){2.5}}
\put(72,30){\line(1,0){1.5}}
\put(84,30){\line(0,-1){5}}
\put(91.5,30){\line(1,0){1.5}}
\put(92.25,30){\line(0,-1){5}}
\put(84,25){\line(1,0){8.25}}
\put(81,35){\line(0,-1){5}}
\put(87.75,35){\line(0,-1){5}}
\put(81,30){\line(1,0){6.75}}
\put(79.5,40){\line(0,-1){5}}
\put(83.25,40){\line(0,-1){5}}
\put(79.5,35){\line(1,0){3.75}}
\put(88.5,40){\line(0,-1){5}}
\put(85.5,35){\line(1,0){3}}
\put(78.75,45){\line(0,-1){5}}
\put(78.75,40){\line(1,0){2.25}}
\put(82.5,40){\line(1,0){1.5}}
\put(89.25,45){\line(0,-1){5}}
\put(87,40){\line(1,0){2.25}}
\put(78,45){\line(1,0){1.5}}
\put(88.5,45){\line(1,0){1.5}}
\put(66,15){\circle*{0.9}}
\put(67.5,15){\circle*{0.9}}
\put(69,10){\circle*{0.9}}
\put(70.5,10){\circle*{0.9}}
\put(72,30){\circle*{0.9}}
\put(73.5,30){\circle*{0.9}}
\put(75,25){\circle*{0.9}}
\put(76.5,20){\circle*{0.9}}
\put(78,45){\circle*{0.9}}
\put(79.5,45){\circle*{0.9}}
\put(81,40){\circle*{0.9}}
\put(82.5,40){\circle*{0.9}}
\put(84,40){\circle*{0.9}}
\put(85.5,35){\circle*{0.9}}
\put(87,40){\circle*{0.9}}
\put(88.5,45){\circle*{0.9}}
\put(90,45){\circle*{0.9}}
\put(91.5,30){\circle*{0.9}}
\put(93,30){\circle*{0.9}}
\put(94.5,20){\circle*{0.9}}
\put(80.25,5){\line(0,-1){5}}
\put(101,15){\circle{0.9}}
\put(102.5,15){\circle{0.9}}
\put(101.75,15){\line(0,-1){10}}
\put(102.5,10){\circle{0.9}}
\put(101.5,5){\line(1,0){14.5}}
\put(111,5){\line(0,-1){5}}
\put(116,5){\line(0,1){10}}
\put(115.5,10){\circle{0.9}}
\put(107,15){\line(1,0){13}}
\put(107,15){\line(0,1){15}}
\put(107.75,20){\circle{0.9}}
\put(107.75,25){\circle{0.9}}
\put(107.75,30){\circle{0.9}}
\put(106.25,30){\circle{0.9}}
\put(120,15){\line(0,1){10}}
\put(120.75,20){\circle{0.9}}
\put(116,25){\line(1,0){10}}
\put(126,25){\line(0,1){5}}
\put(126.75,30){\circle{0.9}}
\put(125.25,30){\circle{0.9}}
\put(116,25){\line(0,1){5}}
\put(111,30){\line(1,0){11}}
\put(122,30){\line(0,1){15}}
\put(121.25,35){\circle{0.9}}
\put(121.25,40){\circle{0.9}}
\put(121.25,45){\circle{0.9}}
\put(122.75,45){\circle{0.9}}
\put(111,30){\line(0,1){5}}
\put(107,35){\line(1,0){8}}
\put(107,35){\line(0,1){10}}
\put(107.75,40){\circle{0.9}}
\put(107.75,45){\circle{0.9}}
\put(106.25,45){\circle{0.9}}
\put(115,35){\line(0,1){5}}
\put(115.75,40){\circle{0.9}}
\put(114.25,40){\circle{0.9}}

\end{picture}
\caption{Equivalent representations of a realization of DTCS(20).}
\label{Fig:1d}
\end{figure}

\begin{figure}
\setlength{\unitlength}{0.043in}
\begin{picture}(160,60)
\put(-5,0){\vector(0,1){55}}
\put(-16,33){time}
\multiput(-5,0)(0,12){5}{\line(-1,0){2}}
\put(-9,-1){0}
\put(-9,11){1}
\put(-9,23){2}
\put(-9,35){3}
\put(-10,47){4}
\multiput(0,0)(3,0){20}{\circle*{0.9}}
\put(30,0){\circle{1.8}}
\put(30,40){\circle{1.8}}
\put(-1.3,-0.5){[}
\put(57.7,-0.5){]}
\multiput(0,3)(3,0){20}{\circle*{0.9}}
\put(-1.3,2.5){[}
\put(6.7,2.5){]}
\put(7.7,2.5){[}
\put(57.7,2.5){]}

\multiput(0,14)(3,0){3}{\circle*{0.9}}
\multiput(0,27)(3,0){2}{\circle*{0.9}}

\put(-1.3,13.5){[}
\put(3.7,13.5){]}
\put(4.7,13.5){[}
\put(6.7,13.5){]}

\put(-1.3,26.5){[}
\put(1.7,26.5){[}
\put(0.7,26.5){]}
\put(3.7,26.5){]}

\multiput(9,6)(3,0){17}{\circle*{0.9}}

\put(7.7,5.5){[}
\put(9.7,5.5){]}
\put(10.7,5.5){[}
\put(57.7,5.5){]}

\multiput(12,11)(3,0){16}{\circle*{0.9}}

\put(10.7,10.5){[}
\put(21.7,10.5){]}
\put(22.7,10.5){[}
\put(57.7,10.5){]}

\multiput(12,20)(3,0){4}{\circle*{0.9}}
\multiput(24,16)(3,0){12}{\circle*{0.9}}

\put(10.7,19.5){[}
\put(18.7,19.5){]}
\put(19.7,19.5){[}
\put(21.7,19.5){]}
\put(22.7,15.5){[}
\put(54.7,15.5){]}
\put(55.7,15.5){[}
\put(57.7,15.5){]}

\multiput(12,25)(3,0){3}{\circle*{0.9}}
\put(10.7,24.5){[}
\put(15.7,24.5){]}
\put(16.7,24.5){[}
\put(18.7,24.5){]}

\multiput(24,21)(3,0){11}{\circle*{0.9}}
\put(22.7,20.5){[}
\put(48.7,20.5){]}
\put(49.7,20.5){[}
\put(54.7,20.5){]}

\multiput(12,40)(3,0){2}{\circle*{0.9}}
\put(10.7,39.5){[}
\put(12.7,39.5){]}
\put(13.7,39.5){[}
\put(15.7,39.5){]}

\multiput(24,26)(3,0){9}{\circle*{0.9}}
\put(22.7,25.5){[}
\put(36.7,25.5){]}
\put(37.7,25.5){[}
\put(48.7,25.5){]}

\multiput(51,33)(3,0){2}{\circle*{0.9}}
\put(49.7,32.5){[}
\put(51.7,32.5){]}
\put(52.7,32.5){[}
\put(54.7,32,5){]}

\multiput(24,33)(3,0){5}{\circle*{0.9}}
\put(22.7,32.5){[}
\put(30.7,32.5){]}
\put(31.7,32.5){[}
\put(36.7,32.5){]}

\multiput(39,35)(3,0){4}{\circle*{0.9}}
\put(37.7,34.5){[}
\put(39.7,34.5){]}
\put(40.7,34.5){[}
\put(48.7,34.5){]}

\multiput(24,40)(3,0){3}{\circle*{0.9}}
\put(22.7,39.5){[}
\put(27.7,39.5){]}
\put(28.7,39.5){[}
\put(30.7,39.5){]}

\multiput(33,44)(3,0){2}{\circle*{0.9}}
\put(31.7,43.5){[}
\put(33.7,43.5){]}
\put(34.7,43.5){[}
\put(36.7,43.5){]}

\multiput(42,40)(3,0){3}{\circle*{0.9}}
\put(40.7,39.5){[}
\put(42.7,39.5){]}
\put(43.7,39.5){[}
\put(48.7,39.5){]}

\multiput(24,55)(3,0){2}{\circle*{0.9}}
\put(22.7,54.5){[}
\put(24.7,54.5){]}
\put(25.7,54.5){[}
\put(27.7,54.5){]}

\multiput(45,49)(3,0){2}{\circle*{0.9}}
\put(43.7,48.5){[}
\put(45.7,48.5){]}
\put(46.7,48.5){[}
\put(48.7,48.5){]}


\put(68,3){\line(1,0){27}}
\put(93.5,3){\line(0,-1){3}}

\put(66.5,14){\line(1,0){4.5}}
\put(68,3){\line(0,1){11}}
\put(71,14){\circle*{0.9}}

\put(65,27){\line(1,0){3}}
\put(66.5,14){\line(0,1){13}}
\put(65,27){\circle*{0.9}}
\put(68,27){\circle*{0.9}}

\put(95,3){\line(0,1){3}}
\put(74,6){\line(1,0){25.5}}
\put(74,6){\circle*{0.9}}

\put(99.5,6){\line(0,1){5}}
\put(81.5,11){\line(1,0){24}}
\put(81.5,11){\line(0,1){9}}

\put(81.5,20){\line(1,0){4.5}}
\put(86,20){\circle*{0.9}}
\put(81.5,20){\line(-1,0){1.5}}

\put(80,20){\line(0,1){5}}
\put(80,25){\line(-1,0){1.5}}
\put(80,25){\line(1,0){3}}
\put(83,25){\circle*{0.9}}

\put(78.5,25){\line(0,1){15}}
\put(77,40){\line(1,0){3}}
\put(77,40){\circle*{0.9}}
\put(80,40){\circle*{0.9}}

\put(105.5,11){\line(0,1){5}}
\put(105.5,16){\line(1,0){16.5}}
\put(122,16){\circle*{0.9}}
\put(105.5,16){\line(-1,0){1.5}}

\put(104,16){\line(0,1){5}}
\put(104,21){\line(1,0){13.5}}
\put(104,21){\line(-1,0){3}}

\put(101,21){\line(0,1){5}}
\put(101,26){\line(1,0){7.5}}
\put(101,26){\line(-1,0){6}}

\put(95,26){\line(0,1){7}}
\put(95,33){\line(1,0){4.5}}
\put(95,33){\line(-1,0){3}}

\put(92,33){\line(0,1){7}}
\put(92,40){\line(1,0){3}}
\put(92,40){\line(-1,0){1.5}}
\put(95,40){\circle*{0.9}}
\put(95,40){\circle{1.8}}

\put(90.5,40){\line(0,1){15}}
\put(89,55){\line(1,0){3}}
\multiput(89,55)(3,0){2}{\circle*{0.9}}

\put(99.5,33){\line(0,1){11}}
\put(98,44){\line(1,0){3}}
\multiput(98,44)(3,0){2}{\circle*{0.9}}

\put(108.5,26){\line(0,1){9}}
\put(107,35){\line(1,0){3}}
\put(110,35){\line(0,1){5}}

\put(107,40){\line(1,0){4.5}}
\put(107,40){\circle*{0.9}}
\put(111.5,40){\line(0,1){9}}
\multiput(110,49)(3,0){2}{\circle*{0.9}}
\put(110,49){\line(1,0){3}}

\put(117.5,21){\line(0,1){12}}
\multiput(116,33)(3,0){2}{\circle*{0.9}}
\put(116,33){\line(1,0){3}}

\put(104,35){\circle*{0.9}}
\put(107,35){\line(-1,0){3}}

\put(91,-3){root}
 \end{picture}
\caption{Equivalent representations of a realization of $\CTCS(20)$.
One distinguished leaf is marked.}
\label{Fig:1c}
\end{figure}
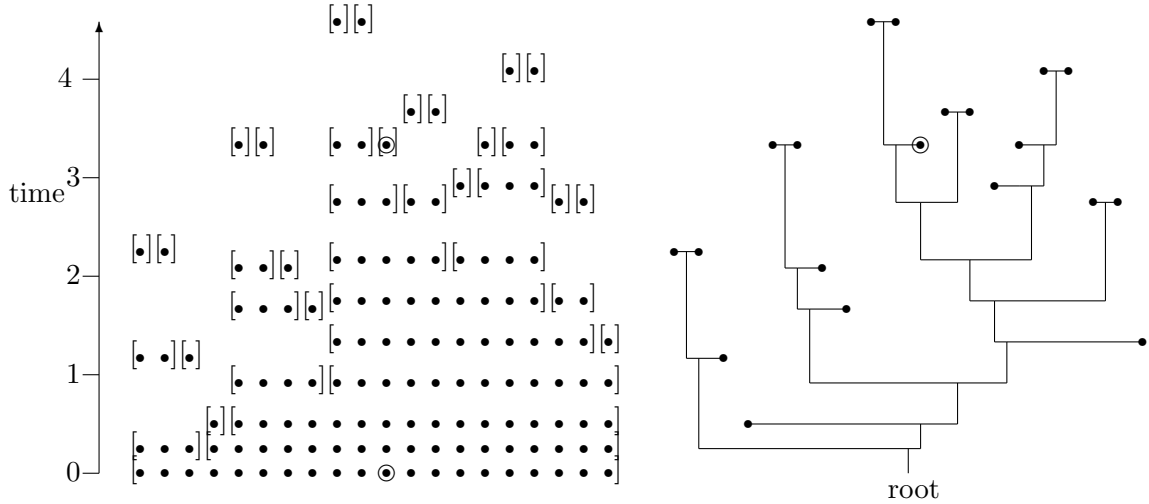

\begin{figure}
\hspace*{-1.9in}
\includegraphics[width=8.5in]{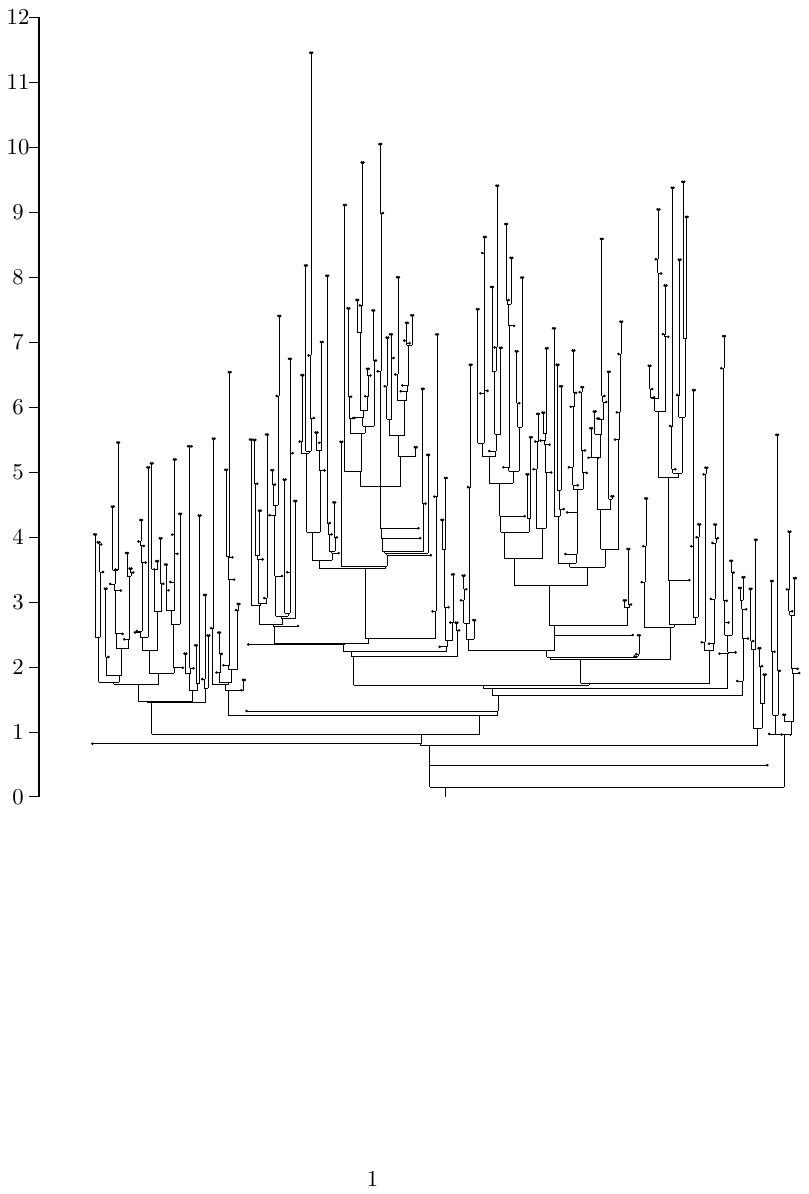}
\vspace*{-4.1in}
\caption{A realization of the tree-representation of the $\CTCS(n)$ model with $n = 400$. 
Drawn as in the previous Figure, so the width of subtrees above a given time level are the sizes of clades at that time.}
\label{Fig:400}
\end{figure}

\begin{Remark}\label{R:infty}
In our representations of the trees, we stop at each leaf.
It is sometimes advantageous to consider an \emph{extended representation} 
where we add a vertical line to infinity from each leaf;
then every clade is represented by a vertical line, extending from the time
the clade is created until it splits (if ever).
This is particularly attractive for $\CTCS(n)$, since leaves split with rate
$0=h_0$ (i.e., never), and thus 
the extended representation 
has for each clade of size $m$ (including leaves)
a vertical line of length $\Expo(h_{m-1})$ 
(interpreted as $\infty$ when $m=1$ so the rate $h_{m-1}=0$)
showing the interval of time that the clade lives.
In this representation, the leaves are located at the bottoms of the
infinite lines (i.e., where the lines branch off from other lines, going
from there to infinity without further branching).
\end{Remark}

Note that in the construction above,
we have labelled the leaves $1,\dots,n$ from left to right.
This is sometimes convenient, but it is often artificial. In particular, the
leaves are not equivalent; for example $1$ and $n$ are the only leaves that
can have height 1, and quantities such as the expectation of the height of
leaf $i$ (in either $\CTCS(n)$ or $\DTCS(n)$) will depend on $i$.
We call this version of $\CTCS(n)$ and $\DTCS(n)$ \emph{ordered};
in the next subsection we consider an alternative.

\subsection{The unordered versions}\label{SSunordered}

In the versions of the construction in \refSS{SSordered},
the leaves are ordered before we start.
Conceptually we are instead usually thinking of recursively splitting a set of
objects which have labels (so they are distinguishable) but  
without any prior structure on the label-set.
Without loss of generality, we can still assume that the set of labels is
$[n]$, but now, each time a clade of size $m$ is to be split into a left
subclade of size $i$ and a right subclade of size $m-i$, we choose the left
subclade uniformly at random among all $\binom mi$ subsets of size $i$.
Otherwise, the construction is exactly as in \refSS{SSordered}.
This yields the \emph{unordered} versions of $\DTCS(n)$ and $\CTCS(n)$.

Note that we may obtain the unordered versions from the ordered ones
by applying a uniform random permutation to the labels of the leaves.
Conversely, we may obtain the ordered versions from the unordered ones by
relabelling the vertices in order from left to right.
Consequently, any properties of the tree that do not depend on the labels of
the leaves are the same for the ordered and unordered versions; two
examples are the height of the tree (i.e., the maximum of the heights of the
leaves), and, for CTCS, the sum of all edge lengths.
Moreover, any properties of the path to a uniform random leaf
will have the same distribution for the ordered and unordered version,

The unordered versions of $\DTCS(n)$ and $\CTCS(n)$
have for our purposes the advantage that they (by definition)
are invariant under permutation of the leaf-labels.
(This of course is a certain type of finite exchangeability.)
Hence, for example, the ``path to a uniform random leaf" 
is equivalent (in distribution) to ``path to leaf $1$".
And ``delete a uniform random leaf" is equivalent to ``delete leaf $n$".
(Recall that this does not hold for the ordered versions.)

Another important advantage of the unordered versions
is the consistency  property in the next subsection.
For this reason,
{\em
in the sequel we will always use the unordered versions unless we
explicitly say otherwise.
}

\begin{Remark}\label{R:unlabelled}
  We may also consider \emph{unlabelled versions} where leaves are not
  labelled.
The left/right distinction still matters, and thus the leaves can still be
identified by their positions.
Hence, the unlabelled versions are equivalent to the unordered versions;
we may obtain the unordered versions by randomly labelling the leaves in the
unlabelled versions. 
\end{Remark}

\subsection{The consistency property and the growth algorithm}
\label{sec:consistency}

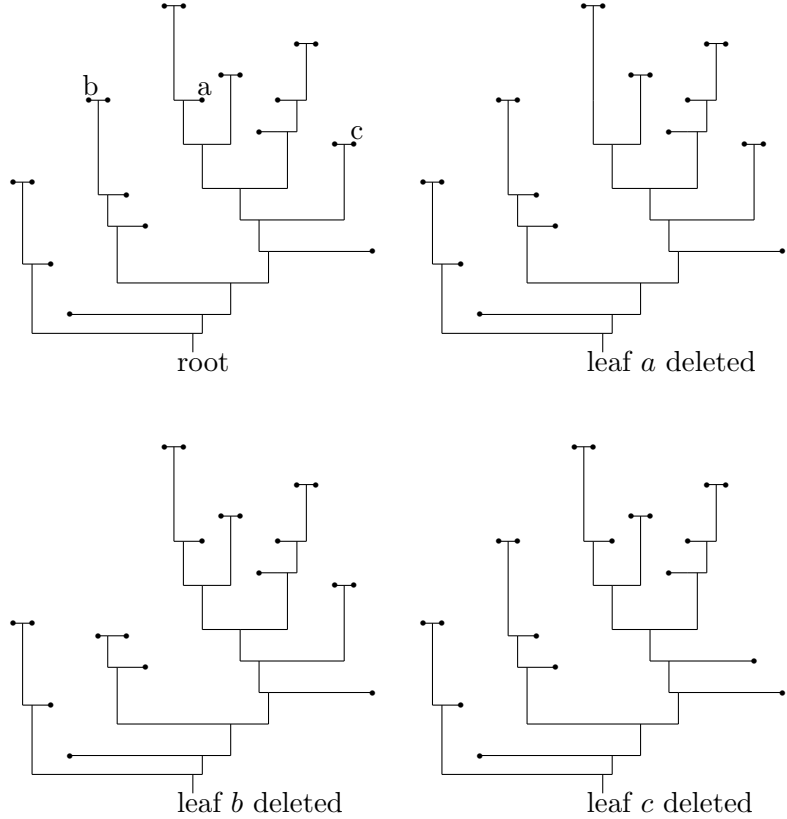
\begin{figure}
\setlength{\unitlength}{0.033in}
\begin{picture}(140,140)(0,-5)

\put(3,3){\line(1,0){27}}
\put(28.5,3){\line(0,-1){3}}

\put(1.5,14){\line(1,0){4.5}}
\put(3,3){\line(0,1){11}}
\put(6,14){\circle*{0.9}}

\put(0,27){\line(1,0){3}}
\put(1.5,14){\line(0,1){13}}
\put(0,27){\circle*{0.9}}
\put(3,27){\circle*{0.9}}

\put(30,3){\line(0,1){3}}
\put(9,6){\line(1,0){25.5}}
\put(9,6){\circle*{0.9}}

\put(34.5,6){\line(0,1){5}}
\put(16.5,11){\line(1,0){24}}
\put(16.5,11){\line(0,1){9}}

\put(16.5,20){\line(1,0){4.5}}
\put(21,20){\circle*{0.9}}
\put(16.5,20){\line(-1,0){1.5}}

\put(15,20){\line(0,1){5}}
\put(15,25){\line(-1,0){1.5}}
\put(15,25){\line(1,0){3}}
\put(18,25){\circle*{0.9}}
\put(13.5,25){\circle*{0.9}}

\put(40.5,11){\line(0,1){5}}
\put(40.5,16){\line(1,0){16.5}}
\put(57,16){\circle*{0.9}}
\put(40.5,16){\line(-1,0){1.5}}

\put(39,16){\line(0,1){5}}
\put(39,21){\line(1,0){13.5}}
\put(39,21){\line(-1,0){3}}

\put(36,21){\line(0,1){5}}
\put(36,26){\line(1,0){7.5}}
\put(36,26){\line(-1,0){6}}

\put(30,26){\line(0,1){7}}
\put(30,33){\line(1,0){4.5}}
\put(30,33){\line(-1,0){3}}

\put(27,33){\line(0,1){7}}
\put(27,40){\line(1,0){3}}
\put(27,40){\line(-1,0){1.5}}
\put(30,40){\circle*{0.9}}

\put(25.5,40){\line(0,1){15}}
\put(24,55){\line(1,0){3}}
\multiput(24,55)(3,0){2}{\circle*{0.9}}

\put(34.5,33){\line(0,1){11}}
\put(33,44){\line(1,0){3}}
\multiput(33,44)(3,0){2}{\circle*{0.9}}

\put(43.5,26){\line(0,1){9}}
\put(42,35){\line(1,0){3}}
\put(45,35){\line(0,1){5}}

\put(42,40){\line(1,0){4.5}}
\put(42,40){\circle*{0.9}}
\put(46.5,40){\line(0,1){9}}
\multiput(45,49)(3,0){2}{\circle*{0.9}}
\put(45,49){\line(1,0){3}}

\put(52.5,21){\line(0,1){12}}
\multiput(51,33)(3,0){2}{\circle*{0.9}}
\put(51,33){\line(1,0){3}}

\put(39,35){\circle*{0.9}}
\put(42,35){\line(-1,0){3}}

\put(26,-3){leaf $b$ deleted}


\put(68,3){\line(1,0){27}}
\put(93.5,3){\line(0,-1){3}}

\put(66.5,14){\line(1,0){4.5}}
\put(68,3){\line(0,1){11}}
\put(71,14){\circle*{0.9}}

\put(65,27){\line(1,0){3}}
\put(66.5,14){\line(0,1){13}}
\put(65,27){\circle*{0.9}}
\put(68,27){\circle*{0.9}}

\put(95,3){\line(0,1){3}}
\put(74,6){\line(1,0){25.5}}
\put(74,6){\circle*{0.9}}

\put(99.5,6){\line(0,1){5}}
\put(81.5,11){\line(1,0){24}}
\put(81.5,11){\line(0,1){9}}

\put(81.5,20){\line(1,0){4.5}}
\put(86,20){\circle*{0.9}}
\put(81.5,20){\line(-1,0){1.5}}

\put(80,20){\line(0,1){5}}
\put(80,25){\line(-1,0){1.5}}
\put(80,25){\line(1,0){3}}
\put(83,25){\circle*{0.9}}

\put(78.5,25){\line(0,1){15}}
\put(77,40){\line(1,0){3}}
\put(77,40){\circle*{0.9}}
\put(80,40){\circle*{0.9}}

\put(105.5,11){\line(0,1){5}}
\put(105.5,16){\line(1,0){16.5}}
\put(122,16){\circle*{0.9}}
\put(105.5,16){\line(-1,0){1.5}}

\put(104,16){\line(0,1){5}}
\put(104,21){\line(1,0){13.5}}
\put(104,21){\line(-1,0){3}}

\put(101,21){\line(0,1){5}}
\put(101,26){\line(1,0){7.5}}
\put(101,26){\line(-1,0){6}}

\put(95,26){\line(0,1){7}}
\put(95,33){\line(1,0){4.5}}
\put(95,33){\line(-1,0){3}}

\put(92,33){\line(0,1){7}}
\put(92,40){\line(1,0){3}}
\put(92,40){\line(-1,0){1.5}}
\put(95,40){\circle*{0.9}}

\put(90.5,40){\line(0,1){15}}
\put(89,55){\line(1,0){3}}
\multiput(89,55)(3,0){2}{\circle*{0.9}}

\put(99.5,33){\line(0,1){11}}
\put(98,44){\line(1,0){3}}
\multiput(98,44)(3,0){2}{\circle*{0.9}}

\put(108.5,26){\line(0,1){9}}
\put(107,35){\line(1,0){3}}
\put(110,35){\line(0,1){5}}

\put(107,40){\line(1,0){4.5}}
\put(107,40){\circle*{0.9}}
\put(111.5,40){\line(0,1){9}}
\multiput(110,49)(3,0){2}{\circle*{0.9}}
\put(110,49){\line(1,0){3}}

\put(117.5,21){\circle*{0.9}}

\put(104,35){\circle*{0.9}}
\put(107,35){\line(-1,0){3}}

\put(91,-3){leaf $c$ deleted}


\put(3,73){\line(1,0){27}}
\put(28.5,73){\line(0,-1){3}}

\put(1.5,84){\line(1,0){4.5}}
\put(3,73){\line(0,1){11}}
\put(6,84){\circle*{0.9}}

\put(0,97){\line(1,0){3}}
\put(1.5,84){\line(0,1){13}}
\put(0,97){\circle*{0.9}}
\put(3,97){\circle*{0.9}}

\put(30,73){\line(0,1){3}}
\put(9,76){\line(1,0){25.5}}
\put(9,76){\circle*{0.9}}

\put(34.5,76){\line(0,1){5}}
\put(16.5,81){\line(1,0){24}}
\put(16.5,81){\line(0,1){9}}

\put(16.5,90){\line(1,0){4.5}}
\put(21,90){\circle*{0.9}}
\put(16.5,90){\line(-1,0){1.5}}

\put(15,90){\line(0,1){5}}
\put(15,95){\line(-1,0){1.5}}
\put(15,95){\line(1,0){3}}
\put(18,95){\circle*{0.9}}

\put(13.5,95){\line(0,1){15}}
\put(12,110){\line(1,0){3}}
\put(12,110){\circle*{0.9}}
\put(15,110){\circle*{0.9}}
\put(11.0,110.9){b}

\put(40.5,81){\line(0,1){5}}
\put(40.5,86){\line(1,0){16.5}}
\put(57,86){\circle*{0.9}}
\put(40.5,86){\line(-1,0){1.5}}

\put(39,86){\line(0,1){5}}
\put(39,91){\line(1,0){13.5}}
\put(39,91){\line(-1,0){3}}

\put(36,91){\line(0,1){5}}
\put(36,96){\line(1,0){7.5}}
\put(36,96){\line(-1,0){6}}

\put(30,96){\line(0,1){7}}
\put(30,103){\line(1,0){4.5}}
\put(30,103){\line(-1,0){3}}

\put(27,103){\line(0,1){7}}
\put(27,110){\line(1,0){3}}
\put(27,110){\line(-1,0){1.5}}
\put(30,110){\circle*{0.9}}
\put(29.2,110.8){a}

\put(25.5,110){\line(0,1){15}}
\put(24,125){\line(1,0){3}}
\multiput(24,125)(3,0){2}{\circle*{0.9}}

\put(34.5,103){\line(0,1){11}}
\put(33,114){\line(1,0){3}}
\multiput(33,114)(3,0){2}{\circle*{0.9}}

\put(43.5,96){\line(0,1){9}}
\put(42,105){\line(1,0){3}}
\put(45,105){\line(0,1){5}}

\put(42,110){\line(1,0){4.5}}
\put(42,110){\circle*{0.9}}
\put(46.5,110){\line(0,1){9}}
\multiput(45,119)(3,0){2}{\circle*{0.9}}
\put(45,119){\line(1,0){3}}

\put(52.5,91){\line(0,1){12}}
\multiput(51,103)(3,0){2}{\circle*{0.9}}
\put(51,103){\line(1,0){3}}
\put(53.5,103.8){c}

\put(39,105){\circle*{0.9}}
\put(42,105){\line(-1,0){3}}

\put(26,67){root}


\put(68,73){\line(1,0){27}}
\put(93.5,73){\line(0,-1){3}}

\put(66.5,84){\line(1,0){4.5}}
\put(68,73){\line(0,1){11}}
\put(71,84){\circle*{0.9}}

\put(65,97){\line(1,0){3}}
\put(66.5,84){\line(0,1){13}}
\put(65,97){\circle*{0.9}}
\put(68,97){\circle*{0.9}}

\put(95,73){\line(0,1){3}}
\put(74,76){\line(1,0){25.5}}
\put(74,76){\circle*{0.9}}

\put(99.5,76){\line(0,1){5}}
\put(81.5,81){\line(1,0){24}}
\put(81.5,81){\line(0,1){9}}

\put(81.5,90){\line(1,0){4.5}}
\put(86,90){\circle*{0.9}}
\put(81.5,90){\line(-1,0){1.5}}

\put(80,90){\line(0,1){5}}
\put(80,95){\line(-1,0){1.5}}
\put(80,95){\line(1,0){3}}
\put(83,95){\circle*{0.9}}

\put(78.5,95){\line(0,1){15}}
\put(77,110){\line(1,0){3}}
\put(77,110){\circle*{0.9}}
\put(80,110){\circle*{0.9}}

\put(105.5,81){\line(0,1){5}}
\put(105.5,86){\line(1,0){16.5}}
\put(122,86){\circle*{0.9}}
\put(105.5,86){\line(-1,0){1.5}}

\put(104,86){\line(0,1){5}}
\put(104,91){\line(1,0){13.5}}
\put(104,91){\line(-1,0){3}}

\put(101,91){\line(0,1){5}}
\put(101,96){\line(1,0){7.5}}
\put(101,96){\line(-1,0){6}}

\put(95,96){\line(0,1){7}}
\put(95,103){\line(1,0){4.5}}
\put(95,103){\line(-1,0){3}}

\put(92,103){\line(0,1){7}}

\put(92,110){\line(0,1){15}}
\put(90.5,125){\line(1,0){3}}
\multiput(90.5,125)(3,0){2}{\circle*{0.9}}

\put(99.5,103){\line(0,1){11}}
\put(98,114){\line(1,0){3}}
\multiput(98,114)(3,0){2}{\circle*{0.9}}

\put(108.5,96){\line(0,1){9}}
\put(107,105){\line(1,0){3}}
\put(110,105){\line(0,1){5}}

\put(107,110){\line(1,0){4.5}}
\put(107,110){\circle*{0.9}}
\put(111.5,110){\line(0,1){9}}
\multiput(110,119)(3,0){2}{\circle*{0.9}}
\put(110,119){\line(1,0){3}}

\put(117.5,91){\line(0,1){12}}
\multiput(116,103)(3,0){2}{\circle*{0.9}}
\put(116,103){\line(1,0){3}}

\put(104,105){\circle*{0.9}}
\put(107,105){\line(-1,0){3}}

\put(91,67){leaf $a$ deleted}

 \end{picture}
\caption{The delete and prune operation: effect of deleting leaf $a$ or $b$ or $c$ from the top left tree.}
\label{Fig:dpo}
\end{figure}

Observe that all versions of the construction above are for a given $n$;
there is no direct connection between the model (discrete or continuous) for
$n$ and the model for $n+1$. 
Nevertheless, for the unordered versions
we have the following {\em consistency property}.

Note that if we delete a leaf $k$,
then we also have to delete the 
internal node that is the mother of $k$ (merging the other two edges at that
node into one), and if the mother has only one other child, we also have to
reduce the height of that child to its grandmother's;
see \refF{Fig:dpo}. We call this  operation ``delete and prune leaf $k$''.

 \begin{Theorem} [Consistency property] 
 \label{T:consistent}
 The operation ``delete and prune leaf $n+1$ from $\CTCS(n+1)$" gives a tree
 distributed as $\CTCS(n)$, and similarly for $\DTCS(n+1)$ and $\DTCS(n)$.
 \end{Theorem}

This consistency property is in fact a special case of 
\cite[Theorem 1 and Proposition 3]{haas-pitman},
where \emph{all} consistent splitting rules are characterized using the
theory of homogeneous fragmentation processes; the connection to such
processes will be discussed in \refS{sec:exch} below.
We will give a simple direct proof of \refT{T:consistent} in \refS{sec:ghost},
which furthermore leads to a proof of the growth algorithm in
\refT{T:growth} below. 
An alternative, elementary but longer, proof is given in \refApp{sec:proofCP}.

\refT{T:consistent} implies that if we start
at some large $N$ and repeatedly delete and prune the last leaf,
we obtain a realization of the sequence $(\CTCS(n))_{n=1}^N$ with the correct
marginal distributions. By Kolmogorov's extension theorem,
there exists an  infinite {\em consistent growth process} 
$(\CTCS(n), n = 1,2,3,\ldots)$ such that, for each $n$,  
 ``delete and prune leaf $n$ from the realization of $\CTCS(n+1)$" gives
 exactly the realization of $\CTCS(n)$.
Conversely, the realization of $\CTCS(n+1)$ is obtained by adding a new leaf
$n+1$ to $\CTCS(n)$ at the appropriate place 
(i.e., at a random place with a specified distribution).
It turns out that this addition can be
described by the following explicit {\em growth algorithm}.

In the context of {\em growth} of trees, it is more evocative to use the
word {\em buds} instead of {\em leaves}, which we use in the following.
In Figure \ref{Fig:dpo} we see {\em side-buds} such as $a$, and
{\em bud-pairs} such as $b,c$.

\begin{Theorem}[The growth algorithm]\label{T:growth}
   Given a realization of\/ $\CTCS(n)$ for some $n \ge 1$:
 \begin{enumerate}
 \item Pick a uniform random bud; move up the path from the root toward that
   bud. 
A \stopp{} event occurs at rate = $1$/(size of clade from current position).
 \item If\/ \stopp{} before reaching the target bud, make a side-bud at that
   point, random on left or right. 
 \item Otherwise, extend the target bud into a branch of\/ $\Expo(1)$
   length to make a bud-pair. 
\end{enumerate}
 Then the result is a realization of\/ $\CTCS(n+1)$.
Consequently, 
we obtain a realization of the growth process $\bigpar{\CTCS(n), n=1,2,\dots}$
by starting with $\CTCS(1)$, which has a single bud at the root,
and then repeating this algorithm ad infinitum.
\end{Theorem}
The proof is given in  \refS{sec:ghost}; an alternative proof, 
which also gives explicit formulas for probability densities, 
is given in \refApp{sec:proofCP}.

To visualize the growth step, 
the addition of a new bud can happen in one of three qualitative ways,
illustrated in Figure \ref{Fig:3b}, as the reverse of the ``cut" in Figure
\ref{Fig:dpo}.
The addition is either
what we will call a \emph{side-bud addition}
(case {\bf a}  in Figure \ref{Fig:3b})
in which a side-bud is
attached at the interior of some existing edge,
or a \emph{branch extension} (case {\bf b}) in which one bud of a terminal pair grows into a new branch to a 
terminal pair of buds, 
or a {\em side-bud extension} (case {\bf c}) in which a side-bud grows into a new branch with two terminal buds.

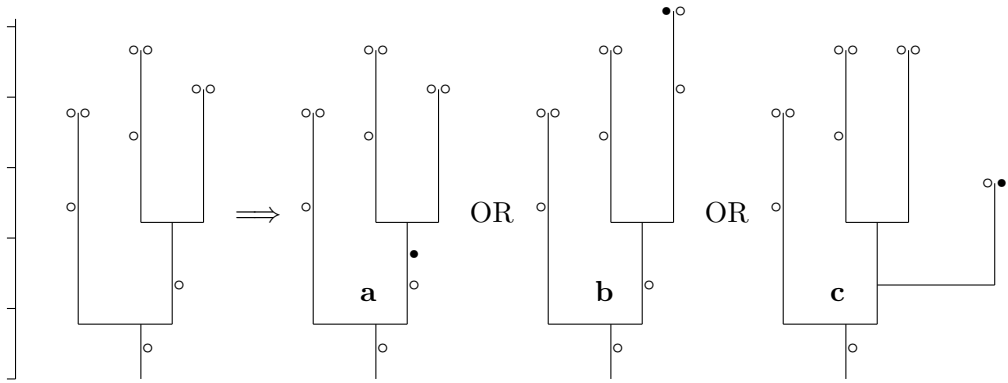
\begin{figure}[ht]
\setlength{\unitlength}{0.041in}
\begin{picture}(60,50)(80,0)
\put(44,0){\line(0,1){46}}
\multiput(44,0)(0,9){7}{\line(-1,0){1.1}}

\put(60,0){\line(0,1){7}}
\put(60.9,4){\circle{1.1}}
\put(52,7){\line(1,0){12}}
\put(52,7){\line(0,1){27}}
\put(51.1,22){\circle{1.1}}
\put(51.1,34){\circle{1.1}}
\put(52.9,34){\circle{1.1}}
\put(64,7){\line(0,1){13}}
\put(60,20){\line(1,0){8}}
\put(64.9,12){\circle{1.1}}
\put(60,20){\line(0,1){22}}
\put(59.1,31){\circle{1.1}}
\put(59.1,42){\circle{1.1}}
\put(60.9,42){\circle{1.1}}
\put(68,20){\line(0,1){17}}
\put(67.1,37){\circle{1.1}}
\put(68.9,37){\circle{1.1}}
\put(90,0){\line(0,1){7}}
\put(90.9,4){\circle{1.1}}
\put(82,7){\line(1,0){12}}
\put(82,7){\line(0,1){27}}
\put(81.1,22){\circle{1.1}}
\put(81.1,34){\circle{1.1}}
\put(82.9,34){\circle{1.1}}
\put(94,7){\line(0,1){13}}
\put(90,20){\line(1,0){8}}
\put(94.9,12){\circle{1.1}}
\put(90,20){\line(0,1){22}}
\put(89.1,31){\circle{1.1}}
\put(89.1,42){\circle{1.1}}
\put(90.9,42){\circle{1.1}}
\put(98,20){\line(0,1){17}}
\put(97.1,37){\circle{1.1}}
\put(98.9,37){\circle{1.1}}

\put(94.9,16){\circle*{1.1}}
\put(120,0){\line(0,1){7}}
\put(120.9,4){\circle{1.1}}
\put(112,7){\line(1,0){12}}
\put(112,7){\line(0,1){27}}
\put(111.1,22){\circle{1.1}}
\put(111.1,34){\circle{1.1}}
\put(112.9,34){\circle{1.1}}
\put(124,7){\line(0,1){13}}
\put(120,20){\line(1,0){8}}
\put(124.9,12){\circle{1.1}}
\put(120,20){\line(0,1){22}}
\put(119.1,31){\circle{1.1}}
\put(119.1,42){\circle{1.1}}
\put(120.9,42){\circle{1.1}}
\put(128,20){\line(0,1){27}}
\put(128.9,37){\circle{1.1}}
\put(127.1,47){\circle*{1.1}}
\put(128.9,47){\circle{1.1}}

\put(150,0){\line(0,1){7}}
\put(150.9,4){\circle{1.1}}
\put(142,7){\line(1,0){12}}
\put(142,7){\line(0,1){27}}
\put(141.1,22){\circle{1.1}}
\put(141.1,34){\circle{1.1}}
\put(142.9,34){\circle{1.1}}
\put(154,7){\line(0,1){13}}
\put(150,20){\line(1,0){8}}
\put(154,12){\line(1,0){15}}
\put(169,12){\line(0,1){13}}
\put(169.9,25){\circle*{1.1}}
\put(168.1,25){\circle{1.1}}

\put(150,20){\line(0,1){22}}
\put(149.1,31){\circle{1.1}}
\put(149.1,42){\circle{1.1}}
\put(150.9,42){\circle{1.1}}
\put(158,20){\line(0,1){22}}
\put(158.9,42){\circle{1.1}}
\put(157.1,42){\circle{1.1}}

\put(88,10){{\bf a}}
\put(118,10){{\bf b}}
\put(148,10){{\bf c}}

\put(72,20){$\Longrightarrow$}
\put(102,20){OR}
\put(132,20){OR}
\end{picture}
\caption{Possible transitions from CTCS(10) to CTCS(11): the added bud is $\bullet$.}
\label{Fig:3b}
\end{figure}

\begin{Remark}\label{R:growth2}
Using the extended representation in \refR{R:infty}, the growth
algorithm in \refT{T:growth} has an even simpler description:
 \begin{enumerate}
 \item Pick a uniform random path to infinity (corresponding to a uniform
   random bud); 
move up this path from the root toward infinity. 
 A \stopp{} event occurs at rate = $1$/(size of clade from current position).
 \item At \stopp, make a side-bud at that point, random on left or right. 
Add a vertical line from the new bud to infinity.
If the current clade at \stopp{} had size 1, so \stopp{} occurred on the
line from some bud to infinity, move also that bud up along the line to the
same height as the new bud.
\end{enumerate}
\end{Remark}

\begin{Corollary}
\label{CP:split}
Let $B_n$ denote the height of the branchpoint between the paths to two
uniform random distinct leaves of $\CTCS(n)$. 
Then, for each $n \geq 2$, $B_n$ has exactly $\Expo(1)$ distribution.
\end{Corollary}

\begin{proof}
  By  exchangeability, $B_n$ has the same distribution as the height of
  the branchpoint between the paths to leaves 1 and 2 in $\CTCS(n)$.
If we consider the growth process given by the growth algorithm
in \refT{T:growth}, then this branchpoint remains the same in $\CTCS(n)$ for all $n\ge2$.
(Leaves may move to higher positions, but branchpoints will not
move.)
Hence, $B_n\eqd B_2$, which by definition has the distribution
$\Expo(h_1)=\Expo(1)$. 
\end{proof}

\begin{remark}
Note that the growth algorithm is for the continuous-time $\CTCS(n)$ only, since
it depends on  the lengths of the edges.
Kolmogorov's extension theorem applies to $\DTCS(n)$ too and yields a
consistent sequence 
$(\DTCS(n), n = 1,2,3,\ldots)$, but this is of less interest since both
leaves and internal nodes move towards infinity as new leaves are added,
while for $\CTCS(n)$, all internal vertices (branchpoints) remain at the
same height when new leaves are added, as seen in Corollary \ref{CP:split}. 
 This happens because in the continuous-time model there is an offsetting feature, that
 the initial splitting rate $h_{n-1}$ is increasing with $n$,
which remarkably compensates exactly  in \refC{CP:split}
\end{remark}

\begin{remark}\label{RCTCSinfty}
Using the growth process $(\CTCS(n), n = 1,2,\ldots)$, we can define 
a limiting object $\CTCS(\infty)$ as the union $\bigcup_{n=1}^\infty\CTCS(n)$,
suitably interpreted. This is perhaps best done with the version in
\refR{R:growth2} with lines to $\infty$.  In that version,  by 
ignoring the buds (which may move when adding a new bud)
and considering the lines only, the lines
form a type of tree structure that grows with $n$ by adding new lines to
$\infty$ at random branchpoints. 

We can regard $\CTCS(\infty)$ as a (non-compact) real tree -- see e.g. \cite{evans}.
Note that this is not  the usual kind of ``locally finite" infinite tree%
\footnote{Such as a supercritical Galton-Watson tree.},
 because a realization has a countable infinite dense set of branchpoints.
We will not study this limit object further; instead we consider a different (though conceptually equivalent) formalization in
\refS{sec:exch}.

Here is another viewpoint on the existence of the limit object.
For any given buds $i$ and $j$ with $i<j$, the branchpoint $B_{ij}$ 
between the paths from the root to $i$ and $j$ in $\CTCS(n)$ is the same for
all $n\ge j$ (its height is $\Expo(1)$ by \refC{CP:split}). 
In the growth process above,
if we consider only the instances where we
happen to move past $B_{ij}$, and record whether we turn towards $i$ or towards $j$,
then this process can be modelled by a P\'{o}lya urn; consequently, 
the proportion of leaves in each branch converges a.s.\ to a random non-zero
limit.
\end{remark}

\subsection{Proof of the consistency property and growth  algorithm}
\label{sec:ghost}

\begin{proof}[Proof of \refT{T:consistent}]
It suffices to consider $\CTCS(n)$, since the result for $\DTCS(n)$ then
follows by ignoring edge-lengths.

Recall from \eqref{ok1} that
the rate at which a clade of size $m$ splits into two clades of sizes $i$
and $m-i$ is $\xq(m,i)$ given by \eqref{ok1}.

Consider $\CTCS(n+1)$, but kill leaf $n+1$ and replace it with an
invisible ghost.
Consider a clade in the tree
with $m$ visible elements plus the ghost. This clade really
has $m+1$ elements and thus splits with rate $h_m$ in $\CTCS(n+1)$, but the
two cases when only the ghost is split off from the rest are invisible.
A visible split into subsets with $j$ and $m-j$ visible elements may have
the ghost in either of the two, and so,
taking into account the probability that the ghost appears in the proper subclade,
the rate is by \eqref{ok1}
\begin{align}\label{f1}
&  \frac{j+1}{m+1}\xq(m+1,j+1) + \frac{m+1-j}{m+1}\xq(m+1,j)
\\&\notag\qquad
=\frac{j+1}{2(j+1)(m-j)}+\frac{m+1-j}{2j(m+1-j)}
=\frac{1}{2(m-j)}+\frac{1}{2j}
= \xq(m,i).
\end{align}
In other words, the ghost does not affect the visible splitting rates.
Hence, if we delete and prune leaf $n+1$ from $\CTCS(n+1)$, we obtain
$\CTCS(n)$, which proves Theorem \ref{T:consistent}.
\end{proof}

\begin{proof}[Proof of \refT{T:growth}]
We continue to consider $\CTCS(n+1)$ with leaf $n+1$ replaced by an
invisible ghost.
In the argument above, we see from the calculation in \eqref{f1}
that if a clade containing the ghost and $m$ visible leaves splits
into new clades of visible sizes $j$ and $m-j$, then
the ghost will be in the left clade, of size $j$, with probability
\begin{align}
 \frac{1/(2(m-j))}{\xq(m,j)}=\frac{j}{m}.
\end{align}
In other words, the ghost moves as if it accompanies a uniformly chosen
visible leaf in the clade. Note also that when the ghost belongs to a clade with
$m$ visible elements, it splits off on its own at the rate
\begin{align}
  \frac{2}{m+1}\xq(m+1,1)=\frac{1}m.
\end{align}
(This follows also because the splitting rate in
$\CTCS(n+1)$ is $h_m$, of which the visible splits have rate $h_{m-1}$;
hence the rate of an invisible split is $h_m-h_{m-1}$.)
This means that given $\CTCS(n)$, the life of the ghost can (up to identity in
distribution) be described by: Choose a leaf in $\CTCS(n)$ uniformly at
random, and follow the branch towards it. With rate
$1/\text{(current size of the clade (excluding the ghost))}$, branch off alone;
if the ghost reaches the chosen leaf, continue together with it and branch
off from it with the same rate (now 1).

We may thus construct $\CTCS(n+1)$ from $\CTCS(n)$ by the procedure just
described, but giving life to the ghost as leaf $n+1$.
This gives precisely the growth algorithm in \refT{T:growth}
(or \refR{R:growth2}).
\end{proof}

\begin{remark}[Alternative proofs]
As said above, \refT{T:consistent} follows also from general results in 
\cite{haas-pitman}, but we do not know any analog of \refT{T:growth} in
the generality studied there.
Before finding the rather ``conceptual" proofs above,  we found a more pedestrian  argument based on explicitly describing the joint distribution
of $(\CTCS(n+1),\CTCS(n))$.  
That argument is given in  \refApp{sec:proofCP}.
There is also a direct (not using the consistency theorem) proof of the branchpoint result  (\refC{CP:split})  
via stochastic calculus -- see \refApp{sec:Exp1}.
Another discussion of exchangeability and consistency of random tree models can be found in \cite{hollering}
but we do not see any direct application to our model.
\end{remark}

\section{Leaf height and the harmonic descent chain}  
\label{sec:height}
\subsection{Leaf height}
\label{sec:LH}
Before continuing to study a formalization of the limit process $\CTCS(\infty)$  and its quantitative properties (Section \ref{sec:paintbox}),
let us describe some relevant quantitative work on another aspect of the model, which is {\em leaf height}.
We let $D_n$ be the height of a uniform random leaf $\ell$ in $\CTCS(n)$,
and let $L_n$ be the height of a uniform random leaf $\ell$ in $\DTCS(n)$.
Equivalently, recalling the relation between $\DTCS(n)$ and $\CTCS(n)$,
$D_n$ is the total length of the path from the root to $\ell$ in
$\CTCS(n)$, while 
$L_n$ is the \emph{hop-height}, i.e., 
the number of edges on the path, 
in either of $\DTCS(n)$ or $\CTCS(n)$.

Recall that in the unordered versions (which we normally use), $D_n$ and
$L_n$ can just as well be defined by taking the path to a fixed leaf
$\ell\in[n]$, for example $\ell=1$.

The limit behavior of both $D_n$ and $L_n$ is studied in great detail in  \cite{beta1}, though here we consider only $D_n$. 
It is easy to see that $t_n := \Ex[D_n]$ satisfies the recurrence
\begin{equation}
t_n =    \sfrac{1}{h_{n-1}} \Bigpar{1 + \sum_{i=1}^{n-1} \sfrac{t_i}{n-i} }; \ n \ge 2
\label{tn2}
\end{equation}
with $t_1 = 0$.
One can see the first order result $\Ex[D_n] \sim  \sfrac{6}{\pi^2}\log n$ heuristically by plugging $c \log n$ into the recursion and taking the natural first-order
approximation to the right side; the constant $c$  emerges as the inverse of the constant
\begin{equation}
\int_0^1\sfrac{\log(1/x)}{1-x}\dd x=\zeta(2)=\sfrac{\pi^2}{6}
\label{zeta2}
\end{equation}
and this heuristic goes back to \cite{me_clad}.
It has recently been proved 
{\cite[Theorems 1.1 and 1.7]{beta1}}

\begin{align}
\mathbb E[D_n] &= \tfrac{1}{\zeta(2)} \log n + O(1), \label{Dn0}\\
\var(D_n)&=(1+o(1))\tfrac{2\zeta(3)}{\zeta^3(2)}\log n  \label{var0}
\end{align}
and the corresponding CLT holds for $D_n$.
These and related results (and analogs for $L_n$) are proved  in
\cite{beta1} by detailed analysis of recursions analogous to \eqref{tn2};
further results will be given in \cite{beta4}.
After the  preprint version of \cite{beta1} was posted, alternative proofs of the CLT have been announced: 
see \cite{iksanovCLT,kolesnik}.
Related weaker results about tree height, that is maximum leaf height, are given in \cite{beta2-arxiv,beta1}.

\subsection{The harmonic descent chain}
\label{sec:HD}
We can characterize $D_n$ in an alternate way, as follows.  In the discrete construction, 
the sequence of clade sizes along the path from the root to $\ell$ is the discrete-time Markov chain, starting in state $n$,  whose transition 
($m \to i)$ probabilities $q^*(m,i)$ 
are obtained by size-biasing the $q(m,\cdot)$ distribution; so
\begin{equation}
 q^*(m,i) :=   \sfrac{2i}{m} q(m,i) = \sfrac{1}{h_{m-1}} \cdot \sfrac{1}{m-i}, 
 \qquad 1 \le i \le m-1 , \ m \ge 2 
 \label{qstar}
 \end{equation}
from \eqref{01}.
Because the continuous-time CTCS process exits $m$ at rate $h_{m-1}$, the continuous-time process of clade sizes 
as one moves at speed $1$ along the path is the 
continuous-time Markov process on states $\{1,2,3,\ldots\}$ with transition rates
\begin{equation}
\lambda_{m,i} := \sfrac{1}{m-i}, \qquad 1 \le i \le m-1, \ m \ge 2
\label{lambda-rates}
\end{equation}
with state $1$ absorbing. 
So $D_n$ is the absorption time for this chain, started at state $n$.
Let us call this the 
(continuous-time)
{\em harmonic descent} (HD) chain.\footnote{{\em Descent} is a reminder that
  the chain is decreasing.  Despite its simple form, the HD chain has
  apparently never been studied before.}

The HD chain is relevant to the current article in two ways.
First,
there is a simple probabilistic heuristic for the behavior of the harmonic descent chain, 
leading to the approximation \eqref{approx} below.
Write $\bX = (X_t, t \ge 0)$ for the HD chain with rates \eqref{lambda-rates}, 
or $\bX^{(n)} = (X^{(n)}_t, t \ge 0)$ for this chain starting with $X^{(n)}_0 = n$.
The key idea is to study the process  $\log \bX = (\log X_t, \ t \ge 0)$.
By considering its transitions, 
one quickly sees that, for large $n$,  there should be a good approximation 
\begin{equation}
 \log X^{(n)}_t \approx \log n - Y_t    \mbox{ while } Y_t < \log n 
 \label{approx}
 \end{equation}
 where  $(Y_t, 0 \le t < \infty)$ is the subordinator with  {\em L\'{e}vy measure} 
  $\psi_\infty$ and corresponding $\sigma$-finite density $f_\infty$ on $(0,\infty)$ defined as
\begin{equation}
 \psi_\infty[a, \infty) :=  - \log (1 - e^{-a}); 
\quad  f_\infty(a) :=  \sfrac{e^{-a}}{1 - e^{-a}}, \quad
  \ 0 < a < \infty .
  \label{muinf}
  \end{equation} 
  Recall that a {\em subordinator} \cite{bertoin}
 is the continuous-time analog of the discrete-time process of partial sums of i.i.d.\ positive summands: informally
 \begin{align}
\Pr(Y_{t+dt}  - Y_{t} \in \ddx a) = f_\infty(a) \dd a\dd t . 
\end{align}
 Such a subordinator satisfies the law of large numbers
\begin{equation}
 t^{-1} Y_t \to \rho \qquad \mbox{ a.s.\ as }  t \to \infty
 \label{LLN}
 \end{equation}
where the limit is the mean
\begin{equation} 
\rho = \int_0^\infty \psi_\infty[a, \infty) \dd a 
=  \int_0^\infty -  \log (1 - e^{-a}) \dd a
= \pi^2/6  .
\label{psirho}
\end{equation}

So the approximation \eqref{approx} provides a heuristic explanation of why $\sfrac{D_n}{ \log n} \to 6/\pi^2$,
and by the CLT for subordinators one can derive a heuristic for the explicit form  \eqref{var0} of the variance. 
This method can, with some effort, be made into a proof of the CLT -- see \cite{beta2-arxiv}.
But instead of asymptotics of  $\CTCS(n)$, we shall show in Section \ref{sec:exch} that the subordinator arises {\em exactly} within the 
limit structure  $\CTCS(\infty)$.

 \subsection{The occupation measure}
 \label{sec:OP}
 Here is the second way in which the HD chain is relevant to this article.
 The chain describes the number of descendant leaves of a node, as one moves at speed $1$
along the path from the root to a uniform random leaf.
We study the ``occupation measure", that is
 \begin{equation}
 \mbox{
 $a(n,i) := $ probability that the chain started at state $n$ is ever in state $i$.
 }
 \label{def:ani}
 \end{equation}
 So $a(n,n) = a(n,1) = 1$.
To see  the relevance of $a(n,i)$ to the tree model,
we 
let $N_n(j)$ be the number of subtrees of $\CTCS(n)$ that have $j$ leaves;
thus, 
for $j\ge2$,  $N_n(j)$ is the number of internal nodes of $\CTCS(n)$ that
have exactly $j$ leaves as descendants. 
Then, conditioned on $\CTCS(n)$, the number of leaves that are in some subtree
with $i$ leaves is $iN_n(i)$, and thus the (conditional) probability that a
random leaf is in such a subtree is $iN_n(i)/n$. Taking the expectation we
find
\begin{align}\label{jup}
  a(n,i) = \frac{i\E[N_n(i)]}{n}
\end{align}
and, conversely,
 \begin{equation}\label{jun12}
\E[N_n(i)]
= n a(n,i)/i .
 \end{equation}
 It seems very intuitive (but not obvious at a rigorous level) that the limits $a(i) = \lim_{n \to \infty} a(n,i)$ exist.
Note that   $\sum_{i=2}^n a(n,i)/h_{i-1} $ is just the mean absorption time $\Ex[ D_n]$, 
 so (from \eqref{Dn0}) we anticipate that,
 assuming the limits exist, 
 \begin{align}
 \sum_{i=2}^n \sfrac{a(i)}{\log i }\sim \Ex [D_n] \sim (6/\pi^2) \log n \mbox{ as } n \to \infty.
\end{align}
 This in turn suggests
 \begin{align}
a(i) \sim \sfrac{6}{\pi^2} \sfrac{\log i}{i}  \mbox{ as } i \to \infty.
\end{align}
 However, there seems no intuitive reason to think there should be some simple formula for the limits $a(i)$.
 So the following result was surprising to us.
  \begin{Theorem} [Occupation measure]
  \label{T:alimit}
 For each $i = 2,3,\ldots$,
 \begin{align}
a(i) : = \lim_{n \to \infty} a(n,i) &= \frac{6 h_{i-1}}{\pi^2 (i-1)} .\label{talimit}
 \end{align}
 And $a(1) = 1$.
 \end{Theorem}
 This is the starting point for our analysis of the  {\em fringe distribution} in Section \ref{sec:OPfringe}.
 We currently know 3 quite different proofs of Theorem \ref{T:alimit}.

\smallskip \noindent
{\bf 1.} One method \cite{HDchain} (straightforward in outline, though somewhat tedious in detail)\footnote{A simplification of that proof has been found by  Luca Pratelli and Pietro Rigo (personal communication).}   
is to first prove by coupling that the limits $a(i)$ exist. 
The limits must satisfy a certain infinite set of equations; the one solution $\frac{6 h_{i-1}}{\pi^2 (i-1)}$ was found by inspired guesswork.
  Then check that the solution is unique.

\smallskip \noindent
  {\bf 2.} Iksanov  \cite{iksanovHD} repeats his method for proving the CLT \cite{iksanovCLT} by exploiting the exact relationship with 
  regenerative composition structures, enabling  a shorter derivation of Theorem \ref{T:alimit} from known results in that theory.
  This methodology is clearly worth further consideration.

\smallskip \noindent
{\bf 3.} In Section \ref{sec:surprise} we give a third proof, illustrating how to exploit the exchangeable representation of $\CTCS(\infty)$.

\section{The exchangeable partitions representation}
\label{sec:exch}
In \refR{RCTCSinfty} we discussed briefly 
a limiting object $\CTCS(\infty)$, which formally is a real tree.
In this section
we will define and study another, related, limiting object,
which formally is a nested family $(\Pi(t))_{t\ge0}$ 
of partitions of $\bbN$.
This uses an existing formalism via  Kingman's theory of exchangeable
partitions; 
a standard reference is  \cite[Section 2.3]{bertoin}  
-- see also \cite{Bertoin-PTRF} and  \cite[Chapter 2]{Pitman}.
The key feature of this approach is Kingman's {\em paintbox theorem}, 
which is stated in our setting in \refT{T:paintbox} below.

The relation between trees and nested families of partitions
has been used at least since \cite{haas-miermont2004}.
For completeness, we develop it below  in detail for our case;
we refer also to \cite{haas-pitman} where this relation is studied
in a more general situation.
(See also Section \ref{sec:CRT} for further discussion.)
The idea is simple: Given a finite tree with edge-lengths and leaves
labelled $1,\dots,n$ 
we define a partition $\Pi(t)$ of $[n]$ for each $t\ge0$
by cutting the tree at time (=height) $t$; conversely, it is easy to see
that, provided there are no vertices with outdegree 1, 
the tree is determined by this family of partitions.
This extends to infinite trees, and for exchangeable infinite trees, 
such as $\CTCS(\infty)$, we obtain a family of exchangeable partitions and
can employ Kingman's theory.


\subsection{Exchangeable partitions}
\label{sec:paintbox}

Fix a level (time) $t\ge0$.
For each $n$, the clades of $\CTCS(n)$ at time $t$
define a partition $\Pi\nnn(t)$ of
$[n]:=\setn$.
If we represent the tree $\CTCS(n)$  as in \refR{R:infty}, with lines extending
to infinity from each node, then $\Pi\nnn(t)$ is the partition obtained
by cutting the tree $\CTCS(n)$ at level $t$;
that is, 
$i$ and $j$ are in the same part if and only if
the branchpoint separating the paths to leaves $i$ and $j$ has height $>t$. 

We use 
the consistent growth process to define $\CTCS(n)$ for all $n\ge1$, and then
these partitions $\Pi\nnn(t)$ are consistent
and define
a partition $\Pi(t)$ of $\bbN:=\set{1,2,\dots}$ into  clades at time $t$.
Explicitly,
$i$ and $j$ (with $i,j\in\bbN$)
are in the same part if and only if
the branchpoint separating the paths to leaves $i$ and $j$ has height $>t$,
in $\CTCS(n)$ for any $n\ge\max(i,j)$.
In other words, $\Pi(t)$ is the partition of $\bbN$ into
the clades defined by the infinite tree
$\CTCS(\infty)$.
Obviously,  $\Pi(0)$ is the trivial partition into a single class.

Because each $\CTCS(n)$ is exchangeable,  $\Pi(t)$ is an exchangeable random
partition of $\bbN$, so we can exploit the theory of exchangeable
partitions. 
Denote the clades at time $t$, that is the parts of $\Pi(t)$, by
$\Pi(t)_1,\Pi(t)_2,\dots$, enumerated in order of the least elements.
In particular, the clade of leaf 1 is $\Pi(t)_1$. 
The clades $\Pi(t)_\ell$ are thus subsets of $\bbN$, and the clades of 
$\CTCS(n)$ are the sets $\Pi(t)_\ell\cap[n]$ that are non-empty.

Writing $\left| \ \cdot \ \right|$ for cardinality, it is easy to show the following
(proofs of the results stated here are given in \refSS{SS:exch-pf}).
\begin{Lemma}
\label{Lpe1a}
Let $t>0$. Then, 
a.s., all clades $\Pi(t)_\ell$ are infinite, that is\/  
$|\Pi(t)_\ell|=\infty$ for every $\ell\ge1$.
\end{Lemma}
Write, for $\ell,n\ge1$,
\begin{align}\label{d1}
  \KX^{(n)}_{t,\ell}:=\bigabs{\Pi(t)_\ell\cap[n]};
\end{align}
the sequence $\KX^{(n)}_{t,1},\KX^{(n)}_{t,2},\dots$ is thus the  sequence of {\em sizes} of
the clades in $\CTCS(n)$, extended by 0's to an infinite sequence.
Lemma \ref{Lpe1a} shows that for every $t>0$,
$\KX^{(n)}_{t,\ell}\to\infty$ as \ntoo{} for every $\ell$.
By Kingman's fundamental result \cite[Theorem 2.1]{bertoin},
the asymptotic proportionate clade sizes,  
that is the limits
\begin{align}\label{d2}
  \XP_{t,\ell}:=\lim_\ntoo\frac{\KX^{(n)}_{t,\ell}}{n},
\end{align}
exist a.s.\ for every $\ell\ge1$, and
the random partition $\Pi(t)$ may be reconstructed (in distribution)
from the limits $(\XP_{t,\ell})_\ell$ by Kingman's paintbox construction,
which we state as the following theorem.
Obviously $\XP_{0,\ell}=\gd_{1\ell}$.

\begin{Theorem}\label{T:paintbox}
Let $t\ge0$.
  \begin{romenumerate}
  \item \label{T:paintbox1}
If\/ $t>0$, then 
a.s.\ each $\XP_{t,\ell}\in(0,1)$, and\/ $\sum_\ell \XP_{t,\ell}=1$.
\item\label{T:paintbox2}     
Given a realization of $(\XP_{t,\ell})_\ell$,
give each integer $i\in \bbN $ a random color $\ell$, 
with probability distribution $(\XP_{t,\ell})_\ell$,
independently for different $i$. These colors define a random partition of
$\bbN$, which has the same distribution as $\Pi(t)$.
  \end{romenumerate}
\end{Theorem}

Note that the paintbox construction in Theorem \ref{T:paintbox} starts with
the limits $\XP_{t,\ell}$, but gives as the result 
(in distribution)
$\Pi(t)$ and thus also 
the partition $\Pi\nnn(t)=\Pi(t)\cap[n]$ for every finite $n$.

Regarding $\CTCS(\infty)$ as a real
tree, the process $(\XP_{t,1}, t \ge 0)$ is
the relative size of the subclade at time $t$, as one moves at speed $1$
down the path to a uniform random leaf on the infinite boundary.

\subsection{The homogeneous fragmentation process}
\label{SS:exch-homo}

We have in \refSS{sec:paintbox} studied a fixed $t$; now consider 
the family of nested partitions $(\Pi(t))_{t\ge0}$. 
It is easy to see that this is a \emph{homogeneous fragmentation process}
as defined in \cite[Definition 3.2]{bertoin}.
To verify this, it suffices by \cite[Lemma 3.4]{bertoin}
to show that 
$(\Pi\nnn(t))_{t\ge0}$ is a homogeneous fragmentation process
for each $n\ge1$, 
which follows directly from the definition of $\CTCS(n)$.

For $n\ge1$, the family of nested partitions $(\Pi\nnn(t))_{t\ge0}$ 
determines when a clade
splits in $\CTCS(n)$, and how the clade splits, except for which subclade is
left and which is right.
Hence, the process $(\Pi\nnn(t))_{t\ge0}$ determines $\CTCS(n)$
up to the order of the children at each vertex;
conversely, $\CTCS(n)$ determines the partitions $\Pi\nnn(t)$ by definition.
Consequently, 
if we ignore the ordering of children in $\CTCS(n)$ (which in any case is
uniformly random), 
the 
process
$(\Pi(t))_{t\ge0}$ determines the entire growth process $(\CTCS(n))_{n\ge1}$ 
and conversely.

The conclusion 
is that 
we may regard the homogeneous fragmentation process 
$(\Pi(t))_{t\ge0}$ of partitions of $\bbN$ as another representation of the limit object
$\CTCS(\infty)$.
We continue to develop some properties of 
$(\Pi(t))_{t\ge0}$;
some of them will later be used to study $\CTCS(n)$ and $\DTCS(n)$.

\begin{remark}\label{R:Haas}
In this paper we start with the concrete definition of $\DTCS(n)$ and
$\CTCS(n)$ in  \refS{sec:tree}, 
and then find explicitly in the present section the corresponding
homogeneous fragmentation process $(\Pi(t))$.
An alternative approach, 
suggested to us by B\'{e}n\'{e}dicte Haas and 
using \cite{haas-pitman}, 
is to start with a general homogeneous fragmentation
process $(\Pi(t))_{t\ge0}$, which can be defined as in \cite{bertoin}
by an erosion coefficient $\bc$ and a  dislocation measure $\bnu$
(see \refSS{SS:jump}); then the restrictions of the partitions $\Pi(t)$
to $[n]$ correspond to a random tree $T_n$ in continuous time
and it is easy to see that the family $(T_n)_n$ of random trees is
consistent.\footnote{One of the results in \cite{haas-pitman} is the converse: every 
consistent family of random trees obtained by some splitting rule can be
obtained in this way.}
Moreover, if we choose the erosion coefficient 0 and the dislocation measure
defined by \eqref{ju8} in \refSS{SS:jump} below, 
then \cite[Theorem 1, (2)]{haas-pitman} 
shows (cf.\ the calculation in \eqref{ju9})  
that this family of random trees has the correct splitting
probabilities $q(m,i)$; furthermore, \eqref{ju4} below then shows that  
the splitting rate is $h_{m-1}$. 
Consequently, 
this constructs $\DTCS(n)$ and $\CTCS(n)$ 
starting from the correct homogeneous fragmentation process.

See also \refR{R:dis}.
\end{remark}

\subsection{Self-similarity}\label{SS:exch-self}
 As in \refSS{sec:paintbox}, consider the version of the tree $\CTCS(n)$
 with all branches extended up to $\infty$ (see Remark \ref{R:infty})
and cut it 
at a fixed height $t$, but now consider also the continuation to
higher levels;
that is, we consider the tree $\CTCS(n)$ restricted to times $u\ge t$, which
defines a forest $F_t\nn$.
The trees in the forest $F\nn_t$ then correspond to the clades at height $t$
in $\CTCS(n)$.

The roots are all at height $t$, but we may make an obvious time translation
so that all roots have height 0.

As $n$ grows, we have the following self-similar behaviour
as a consequence of the growth algorithm;
this can be seen as a consequence of the fact that the fragmentation process
$(\Pi(t))_{t\ge0}$ is homogeneous 
(see \refSS{SS:exch-homo} and \cite[p.~119]{bertoin}),
but we give also an elementary direct  proof in \refSS{SS:exch-pf}.

\begin{Theorem}\label{T:selfsim}
  Let $t\ge0$ be fixed and let $n$ grow from $1$ to $\infty$. 
At each increase of $n$, either one of the trees in $F\nn_t$ gets a new leaf,
or a new tree consisting only of a root is added to $F\nn_t$; in either case
all other trees in $F\nn_t$ remain unchanged.
Moreover, each tree in $F\nn_t$, considered only when it is born or
increases in size, grows as a copy of the process $\CTCS(n)$,
and different trees grow as  independent copies.
\end{Theorem}

\subsection{The subordinator within $\CTCS(\infty)$}\label{sec:sub}

Let us consider the clade containing a given (or random)
node and see how it develops
as time increases; by exchangeability, we may consider the clade containing 1.

For given $n$ the process
$(  \KX^{(n)}_{t,1} ,  t \ge 0)$ at \eqref{d1} of the clade size
is the harmonic descent (Section \ref{sec:HD}) chain 
 $(X^{[n]}_t, t \ge 0)$ 
started at state $n$.
We have described informally
the approximation \eqref{approx} of this $(  \KX^{(n)}_{t,1} ,  t \ge 0)$ 
by the subordinator 
 $(Y_t, 0 \le t < \infty)$ with  L\'{e}vy measure
  $\psi_\infty$ and corresponding $\sigma$-finite density $f_\infty$ on
  $(0,\infty)$ defined in \eqref{muinf}, which we for convenience repeat:
\begin{equation}
 \psi_\infty[a, \infty) :=  - \log (1 - e^{-a}); 
\quad  f_\infty(a) :=  \sfrac{e^{-a}}{1 - e^{-a}}, \quad
  \ 0 < a < \infty .
  \label{muinf3}
  \end{equation} 
The next theorem says that
this becomes exact in the $n \to \infty$ limit given  by \eqref{d2}.
We note that by  \cite[Theorem 3.2]{bertoin},
a.s.\ the limit $P_{t,1}$ in \eqref{d2} exists for all $t\ge0$
simultaneously.

\begin{Theorem}
\label{T:exact}
Define $Y_t := - \log  \XP_{t,1}$.
Then $(Y_t, 0 \le t < \infty)$ is the subordinator 
given by \eqref{muinf3}.
Moreover, for $t\ge0$ and complex $s$ with $\Re s>-1$,
\begin{align}\label{b5}
\Ex [\cXt^s] =\Ex [ e^{-s Y_t}]  
=e^{-t(\psi(s+1)-\psi(1))}
\end{align}
where $\psi(z):=\gG'(z)/\gG(z)$ 
is the digamma function. 
\end{Theorem}
We prove Theorem \ref{T:exact} in Section \ref{SS:exch-pf} 
by calculating moments.

As noted after Theorem \ref{T:paintbox}, for finite $n$  the partition of $\CTCS(n)$ 
into clades at a fixed level $t$ can be also described by the limits $P_{t,\ell}$. 
Similarly, considering only $\ell=1$ but all $t\ge0$
simultaneously, the harmonic descent chain describing the size of the first
clade can be reconstructed (in distribution) for any finite $n$ 
from the process $\XP_{t,1}$,
or equivalently from the subordinator $Y_t$, as shown by
Iksanov  \cite{iksanovHD,iksanovCLT}.

\subsection{Jump rates and dislocation measure}\label{SS:jump}
The general theory of homogeneous fragmentation processes in 
\cite[Sections 3.1--3.2]{bertoin} 
includes several  further 
objects associated with such processes that can be used to study and characterize them. 
In this subsection we calculate the objects below
for the process $(\Pi(t))_{t\ge0}$.
The results of this subsection will not be used in the present paper, 
but the results are included both for possible future use and to illustrate more
aspects of the general theory that apply to our setting.

For $n\in\bbN\cup\set{\infty}$, let $\cP_n$ denote the set of partitions of
$[n]$, where $[\infty]:=\bbN$. The trivial partition into a single class is
denoted $\one_{[n]}$.
Let $\cP_n':=\cP_n\setminus\set\onen$.
We denote the parts of the partition $\pi$ by $\pi_1,\pi_2,\dots$, 
in order of their least elements; the number of parts is $|\pi|\ge1$.
Thus $\cP_n':=\set{\pi\in\cP_n:|\pi|\ge2}$.

\subsubsection{}
The \emph{jump rates} $q_\pi$ \cite[p.~121]{bertoin}
are defined for (finite) partitions 
$\pi\in\bigcup_{1\le n<\infty}\cP_n'$
as the jump rates from $\onen$ in the (Markov) process $\Pi\nnn(t)$.
They are thus equal to the rate that the initial clade $[n]$ splits in
$\CTCS(n)$ according to the partition $\pi$.
Hence, 
\begin{align}\label{ju1}
q_\pi=0 
\qquad\text{if $|\pi|\ge3$},  
\end{align}
while if $|\pi|=2$, then by \eqref{ok1}, noting that we now specify the
parts as subsets of $[n]$ (and not just their sizes as in \eqref{ok1}), 
and, on the other hand, 
that we ignore the left/right distinction which gives a factor 2,
\begin{align}\label{ju2}
  q_\pi= \frac{2}{\binom{n}{|\pi_1|}}
\xq(n,|\pi_1|)
=2\frac{|\pi_1|!\,|\pi_2|!}{n!}\frac{n}{2|\pi_1||\pi_2|}
=\frac{(|\pi_1|-1)!\,(|\pi_2|-1)!}{(n-1)!}
.\end{align}

\subsubsection{}
The \emph{splitting rate} \cite[p.~122]{bertoin}
is a (possibly infinite) measure $\bmu$ on $\cP_\infty$
with $\bmu\set{\oneoo}=0$ characterized by
\begin{align}\label{ju3}
\bmu\bigset{\pi'\in\cP_\infty:\pi'|_{[n]}=\pi}
=q_\pi  
\end{align}
for every finite $n$ and every $\pi\in\cP_n'$.
It follows from \eqref{ju1} that $\bmu$ is supported on the set 
$\set{\pi\in\cP_\infty:|\pi|=2}$.
We note that 
\eqref{ju1}--\eqref{ju3} yield
\begin{align}\label{ju4}
\bmu\bigset{\pi'\in\cP_\infty:\pi'|_{[n]}\neq\onen}
=\sum_{\pi\in\cP_n'}q_\pi 
=\sum_{i=1}^{n-1} \xq(n,i)=h_{n-1}.
\end{align}
It follows that the total mass $\bmu(\cP_\infty')=\infty$.

\subsubsection{}
In general, the splitting rate can be 
decomposed as a sum of two measures
which are determined by the \emph{erosion coefficient} $\bc$ and the
\emph{dislocation measure} $\bnu$, respectively
\cite[Theorem 3.1 and p.~128]{bertoin}.
The erosion coefficient equals the mass $\bmu(\beps)$ of the partition
$\beps\in\cP_\infty$ with two blocks: $\set1$ and $\bbN\setminus\set1$.
For every $n\ge1$,
we have by \eqref{ju3} and \eqref{ju2}, 
with $\beps_n:=\beps|_{[n]}$, the partition of $[n]$ into $\set1$ and
$\set{2,\dots,n}$, 
\begin{align}
  \bc
= \bmu(\beps)
\le \bmu\bigset{\pi'\in\cP_\infty:\pi'|_{[n]}=\beps_n}
=q_{\beps_n}=\frac{1}{n-1}.
\end{align}
Thus, the erosion coefficient $\bc=0$.

\subsubsection{}
The dislocation measure $\bnu$ is a (possibly infinite) measure
on the space $\cPm$ of \emph{mass partitions}, where a mass partition $\bs$
is an infinite sequence $s_1\ge s_2\ge\dots \ge0$ such that
$\sum_{i=1}^\infty s_i\le1$ 
\cite[Definition 2.1]{bertoin}. 
The space $\cPm$ is a compact metric space,
see \cite{bertoin}.
Each mass partition $\bs$ defines a random partition of $\bbN$ by the
paintbox construction in \refT{T:paintbox}\ref{T:paintbox2}
(with obvious change of notation, and allowing for missing mass if
$\sum_is_i<1$, see \cite[Lemma 2.7]{bertoin}); the distribution of this
random partition is denoted $\brho_\bs$.
The dislocation measure $\bnu$ is characterized by, 
assuming for simplicity that $\bc=0$ as in our case, 
see \cite[pp.~126--128]{bertoin}, 
\begin{align}\label{ju6}
  \bmu=\int_{\cPm}\brho_\bs \dd\bnu(\bs).
\end{align}
Also, or as a consequence of \eqref{ju6} and $\bmu(\oneoo)=0$, 
$\bnu$ has no mass at the point $(1,0,0,\dots)\in\cPm$.
(Note that in \refT{T:paintbox}, we used the paintbox construction for a
fixed $t$; here we use it for the splitting rate $\bmu$, which can be seen
as a version for infinitesimally small $t$.)

It is easy to see that if $\bs$ has at least 3 non-zero terms, or if 
$\sum_i s_i<1$, then $\brho_\bs$ gives a positive probability to the set of
partitions with more than two parts; since $\bmu$ gives mass 0 to such
partitions, $\bnu$ is concentrated on the set of mass partitions
$\bs_x:=(x,1-x,0,\dots)$ for $x\in[\frac12,1)$. (We need $x\ge\frac12$
since $s_1\ge s_2$ is assumed.) Given a partition $\pi\in\cP_n$ with two
parts of sizes $i$ and $n-i$ (with $1\le i\le n-1$),
the paintbox construction using $\bs_x$ yields a probability,
using $\brho_{\bs_x}$ to denote also the induced probability distribution on
$\cP_n$, 
\begin{align}\label{ju7}
  \brho_{\bs_x}(\pi) = x^i(1-x)^{n-i}+x^{n-i}(1-x)^i.
\end{align}
We claim that $\bnu$ is the (infinite) measure on $\cPm$ obtained 
as the push-forward by the map $x\mapsto\bs_x$ of the measure
\begin{align}\label{ju8}
  \ddx\xbnu:=\frac{\ddx x}{x(1-x)}
\qquad \text{on } [\tfrac12,1)
.\end{align}
To verify this, it suffices
to calculate 
for a partition $\pi\in\cP_n$ as above, 
using \eqref{ju7}, \eqref{ju8}, and \eqref{ju2},
\begin{align}\label{ju9}
\int_{1/2}^1\brho_{\bs_x}(\pi) \dd\xbnu(x)&
=
\int_{1/2}^1\Bigpar{x^i(1-x)^{n-i}+x^{n-i}(1-x)^i}\frac{\ddx x}{x(1-x)}
\\\notag&
=
\intoi {x^i(1-x)^{n-i}}\frac{\ddx x}{x(1-x)}
=
\frac{\gG(i)\,\gG(n-i)}{\gG(n)}
=q_\pi
,\end{align}
which by \eqref{ju3} verifies \eqref{ju6}.

\begin{remark}\label{R:dis}
 The dislocation measure\footnote{In the symmetric version with $x\in(0,1)$.} 
  $\xbnu$ in \eqref{ju8} 
appears alternatively in 
the definition of $\DTCS(n)$ in \cite{me_clad}.
In fact, \cite[Section 4]{me_clad} considers first a general construction
of random binary splits. To split a clade with $n$ leaves, the leaves
are represented by i.i.d.\ uniformly distributed random points in $(0,1)$,
and then the unit interval is split at a random point $X$ with a given
density $f(x)$ in $(0,1)$; we condition on this giving a proper split.
The beta-splitting model is defined in \cite{me_clad} for $-1<\gb<\infty$ 
using this construction with the beta density $f(x)=c_\gb x^\gb(1-x)^\gb$;
for $-2<\gb\le-1$.
Here $f(x)$ is not a probability density, but the calculation
of splitting probabilities still makes  sense, and defines the model.
For $\gb=-1$, this calculation is just \eqref{ju9}.

As mentioned before, the framework of exchangeable partitions 
have been used by Haas et al \cite{haas-pitman,haas-miermont} 
in somewhat similar contexts -- 
see Section \ref{sec:CRT} for further discussion.
\end{remark}

\subsection{Proofs}\label{SS:exch-pf} 

\begin{proof}[Proof of Lemma \ref{Lpe1a}]
\label{sec:proofviaexch}
  It is easily seen from the
growth algorithm that a.s., as $n$ grows to $\infty$:
\begin{enumerate}
\item
Infinitely many buds of height $<t$ are added, and thus 
$\Pi(t)_\ell\neq\emptyset$ for every $\ell\ge1$, and
\item  
Once a clade $\Pi\nnn(t)_\ell$ is non-empty, new leaves will be added to it
an infinite number of times.
\end{enumerate}
The result follows. 
\end{proof}

\begin{proof}[Proof of \refT{T:paintbox}]
First, obviously 
$\XP_{t,\ell}\in\oi$, and $\sum_\ell \XP_{t,\ell}\le1$ by Fatou's lemma.
Part \ref{T:paintbox2} is Kingman's paintbox construction 
\cite[Theorem 12.1]{bertoin},  
stated for the special case when $\sum_\ell \XP_{t,\ell}=1$.
This holds a.s.\
since otherwise the general version
of the paintbox construction would imply that
$|\Pi(t)_\ell|=1$ for some $\ell$ 
\cite[Proposition 2.8(iii)]{bertoin},
which is ruled out by Lemma \ref{Lpe1a}.
\end{proof}

\begin{proof}[Proof of Theorem \ref{T:selfsim}]\label{PfT:selfsim}
  Consider the effect on the forest $F\nn_t$ of adding a new leaf by the
  growth algorithm.
We have the following cases:
\begin{romenumerate}
\item 
If 
the algorithm stops at height $u< t$,
then $\CTCS(n)$ gets a new leaf (bud) there, which means that $F\nn_t$ gets
a new tree consisting of a root only.
\item 
If the target leaf has height $\ge t$ and the algorithm does not stop before
reaching height $t$, then the algorithm will continue in the tree containing
the target exactly as it would if acting on this tree separately.
All other trees in $F\nn_t$ remain unchanged.
Note also that the probability of reaching height $t$ is the same for
all  target leaves in a given tree in $F\nn_t$;
hence the conditional distribution of the target leaf, given the tree in
$F\nn_t$ that it belongs to, is uniform.
\item 
If the target leaf has height $< t$ and the algorithm does not stop until
reaching the target, then the target leaf is extended into a branch of  $\Expo(1)$ 
length $L$ ending with a bud-pair.

If $u+L< t$, then the two buds in the pair define separate singleton trees in
$F\nni_t$, and thus the net effect is to add
a new tree consisting of a root only to $F\nn_t$.

On the other hand, if $u+L\ge t$, then the tree in $F\nn_t$ consisting of the
target leaf (only) becomes a tree with two leaves at the end of a branch of
length $L+u-t$. Since the exponential distribution has no memory, also this
branch length has $\Expo(1)$ (conditional) distribution, and thus this tree
has the distribution of $\CTCS(2)$.
\end{romenumerate}
All cases conform to the description in the statement.
\end{proof}

\begin{proof}[Proof of Theorem \ref{T:exact}]
This is, apart from the explicit formula \eqref{b5}, an instance of
\cite[Theorem 3.2]{bertoin}.
Nevertheless, we find it instructive to give an explicit proof, partly using
the same arguments as \cite{bertoin}.
We prove the theorem in 3 steps.

\medskip
\noindent
{\em Step 1. $Y_t$ is a subordinator.}
Recall that $\KX\nn_{t,1}$ is the size of the first clade of $\CTCS(n)$ at
time $t$.
Consider two fixed times $t$ and $t+h$, where $h>0$.
Then Theorem \ref{T:selfsim} and \eqref{d2} imply that a.s., as \ntoo,
\begin{align}\label{d3}
  \frac{\KX\nn_{t+h,1}}{\KX\nn_{t,1}}\to \XP'_{h,1},
\end{align}
where $\XP'_{h,1}$ is a copy of $\XP_{h,1}$ that is independent of $\XP_{t,1}$.
Consequently, a.s.
\begin{align}\label{d4}
  \frac{\KX\nn_{t+h,1}}{n}=
  \frac{\KX\nn_{t+h,1}}{\KX\nn_{t,1}}
\cdot
  \frac{\KX\nn_{t,1}}{n}
\to
\XP'_{h,1} \XP_{t,1}
\end{align}
and thus
\begin{align}\label{d5}
  \XP_{t+h,1}= \XP_{t,1} \XP'_{h,1}.
\end{align}
Hence, $Y_t:=-\log \XP_{t,1}$ is an increasing 
stochastic process with stationary independent increments,
i.e., a subordinator.
Note that $Y_t<\infty$ a.s.\ since $P_{t,1}>0$ by Theorem~\ref{T:paintbox}.

\medskip
\noindent
{\em Step 2. The L{\'e}vy measure is given by \eqref{muinf3}.}
In order to verify this, we calculate moments.
Let $k\ge0$.
By the paintbox construction 
in Theorem \ref{T:paintbox},
\begin{align}\label{bc1}
  \Pr \bigpar{\Pi(t)_1\cap[k+1]=[k+1]\mid (\XP_{t,\ell})_{\ell=1}^\infty}
=\suml\XP_{t,\ell}^{k+1}.
\end{align}
Furthermore,
also as a consequence of the paintbox construction, 
$\XP_{t,1}$ has the same distribution as a size-biased sample of 
$(\XP_{t,\ell})_{\ell=1}^\infty$
\cite[Proposition 2.8]{bertoin},
and thus
\cite[Corollary 2.4]{bertoin}
\begin{align}\label{bc2}
   \Ex [\XP_{t,1}^k]
=\Ex\Bigsqpar{ \suml\XP_{t,\ell}^{k+1}}.
\end{align}
Consequently, \eqref{bc2} and \eqref{bc1} 
together with the definition of $\Pi(t)_1$ in Section \ref{sec:paintbox}
yield
\begin{align}\label{bc3}
 \Ex[ \XP_{t,1}^k]&
= \Pr \bigpar{\Pi(t)_1\cap[k+1]=[k+1]}
\\\notag&
= \Pr \bigpar{2,3,\dots,k+1\in\Pi(t)_1}
\\\notag&
=\Pr \bigpar{\CTCS(k+1) \text{ has no branchpoint with height }\le t}.
\end{align}
 The latter event occurs if and only if 
for each $j\le k$, in the inductive construction
by the growth algorithm
of $\CTCS(j+1)$ from $\CTCS(j)$, there is no stop at height $\le t$.
Since the subclade at the current position then has size $j$ for all times
$\le t$,
the probability of this happening at step $j$, given that it has happened so
far, is $\exp(-t/j)$. Consequently,
\begin{align}\label{a2}
  \Ex [\cXt^k]&
=\prod_{j=1}^ke^{-t/j}
=\exp\Bigpar{-t\sum_{j=1}^k \frac{1}{j}}
=e^{-h_k t}.
\end{align}

On the other hand, 
let $\cY_t$ be the subordinator with L{\'e}vy measure
given by \eqref{muinf3}. Then, by definition, 
for any real $s\ge0$,
\begin{align}\label{b3}
  \Ex [e^{-s\cY_t}] = \exp\Bigpar{-t\intoo (1-e^{-sx})f_\infty(x)\dd x}
= \exp\Bigpar{-t\intoo \frac{1-e^{-sx}}{1-e^{-x}}e^{-x}\dd x}
.\end{align}
In particular, if $s=k$ is an integer,
\begin{align}\label{b2}
\Ex [\bigpar{e^{-\cY_t}}^k]&
=  \Ex [e^{-k\cY_t}] 
= \exp\Bigpar{-t\intoo \frac{e^{-x}-e^{-(k+1)x}}{1-e^{-x}}\dd x}
= \exp\Bigpar{-t\intoo \sum_{j=1}^k e^{-jx}\dd x}
\\\notag&
=\exp\Bigpar{-t\sum_{j=1}^k\frac{1}{j}}
=e^{-h_k t}
.\end{align}
Consequently, for any $t\ge0$,
\eqref{a2} and \eqref{b2} show that $\Ex[ \cXt^k] =\Ex[ \bigpar{e^{-\cY_t}}^k]$
for all $k\ge1$, and thus, by the method of moments
\footnote{The random variables on both sides are bounded, with values in $\oi$.}
$\cXt\eqd e^{-\cY_t}$.
Thus $Y_t=-\log \XP_{t,1}\eqd \cY_t$.

This calculation is for a fixed $t\ge0$, but we know that the process
$(Y_t)$ is a subordinator, and thus the distribution of the entire process
is determined by, say, the distribution of $Y_1$.

\medskip
\noindent
{\em Step 3. The moment formula.}
 For any complex $s$ with $\Re s>-1$, we have
\cite[5.9.16]{NIST} 
(as is easily verified by standard arguments)
\begin{align}\label{b4}
\intoo \frac{1-e^{-sx}}{1-e^{-x}}e^{-x}\dd x
=
\intoo \frac{e^{-x}-e^{-(s+1)x}}{1-e^{-x}}\dd x
=\psi(s+1)-\psi(1)
,
\end{align}
generalizing the formula for integer $s$ in \eqref{b2}.
Hence, \eqref{b3} yields, for $t\ge0$ and $\Re s>-1$, 
\begin{align}
\label{b55}
\Ex[ \cXt^s] 
=\Ex [e^{-sY_t} ] 
=\Ex [e^{-s\cY_t} ] 
=e^{-t(\psi(s+1)-\psi(1))}
,\end{align}
which shows \eqref{b5}
and completes the proof of the theorem.
\end{proof}

\begin{remark}
  \cite[Theorem 3.2]{bertoin} gives a general formula relating the
moment and Laplace transform in \eqref{b5} to the dislocation measure in
\eqref{ju8}.
This can be used to show \eqref{b5}, although we preferred above a
calculation using the growth algorithm; conversely,  
using  \cite[footnote on p.~135]{bertoin}, this formula can be used to show
\eqref{ju8} from \eqref{b5}.
\end{remark}

 \section{The occupation measure and the fringe process}
 \label{sec:OPfringe}

  \subsection{The (limit) fringe tree}
\label{sec:fringe}
To be consistent with the {\em cladogram} representation described below, we work here in the discrete time $\DTCS(n)$ setting:
the definition \eqref{def:ani} of $a(n,i)$ is of course unchanged in discrete time.

  The motivation for Theorem \ref{T:alimit} involves the (asymptotic) {\em fringe tree}
  for the random tree model $\DTCS(n)$, that is 
   the $n \to \infty$ local weak limit of the tree relative to a typical leaf.
  See \cite{me-fringe,fringe,SJ385} for general accounts
  of fringe trees, which for us\footnote{The general accounts take limits relative to a random {\em node}, but for our leaf-labelled trees it is more natural to use leaves.  In the terminology of \cite{me-fringe,fringe} these are {\em extended} fringe trees.}
  are random locally finite trees with a
  distinguished leaf.
  It will be straightforward to verify that  the fringe tree can be described in terms of the limits $(a(i), i \ge 1)$ as follows.
  
  \begin{Theorem}\label{Tqup}
    \begin{alphenumerate}
  \item\label{Tqup1}
 The sequence of clade sizes as one moves away from the distinguished leaf is the discrete time ``reverse HD" Markov chain started at state~$1$,  whose ``upward" 
transition probabilities $  q^\uparrow(i,j) $ are given by
\begin{align}\label{qup1}
\qup(i,j)=  \frac{a(j)}{a(i)}q^*(j,i), 
\end{align}
which, from the explicit formula \eqref{talimit} for $a(i)$, becomes
\begin{align}\label{qup2}
q^\uparrow (1,j) &= 6 \pi^{-2}     \sfrac{1}{(j-1)(j-1)}, \qquad j \ge 2 \\
q^\uparrow (i,j) &= \sfrac{i-1}{(j-1)(j-i)h_{i-1}}, \qquad 2 \le i < j . 
\label{qup3}
\end{align}
\item \label{Tqup2}
 At each such upward step $i \to j$, there is the  sibling clade of size $j-i$, and
 this clade is distributed as $\DTCS(j-i)$,
 independently for each step.
This sibling clade is randomly on the left or right side.     
\end{alphenumerate}
  \end{Theorem}
One can check that \eqref{qup3}  is a {\em probability} distribution by observing
\begin{align}
\sum_{j>i} \sfrac{1}{(j-1)(j-i)} 
= \sum_{j>i}\sfrac{1}{i-1} (\sfrac{1}{j-i} - \sfrac{1}{j-1})
= \sfrac{ h_{i-1}}{i-1} .
\end{align}

\begin{proof}
\ref{Tqup1}:
This is a simple exercise in reversing a Markov chain.
Let $1=k_1<k_2<\dots<k_\ell$ be a finite sequence of integers.
If $n\ge k_\ell$, then the probability that the HD chain started at $n$ 
ends with $k_\ell, k_{\ell-1},\dots,k_1=1$ is, by the definition of
$a(n,i)$ and recalling the transition probabilities $q^*(m,i)$ 
of the HD chain in \eqref{qstar}, 
\begin{align}
  a(n,k_\ell)\prod_{i=1}^{\ell-1} q^*(k_{\ell+1-i},k_{\ell-i})
=\prod_{j=1}^{\ell-1}\frac{a(n,k_{j+1})}{a(n,k_j)}q^*(k_{j+1},k_j).
\end{align}
So the reverse chain is a Markov chain with transition
probabilities 
\begin{align}\label{qup7}
\qup_n(i,j)=
  \frac{a(n,j)}{a(n,i)}q^*(j,i), 
\qquad i<j\le n.
\end{align}
Taking the limit as \ntoo, we obtain by \eqref{talimit} and \eqref{qstar}
the
transition probabilities \eqref{qup1}--\eqref{qup3}.

\ref{Tqup2}: Obvious.
\end{proof}
\begin{Remark}
  Using \eqref{qstar}, we can also write \eqref{qup7} as
\begin{align}\label{qup8}
i^{-1} a(n,i)   q^\uparrow(i,j) 
= 2j^{-1} a(n,j) q(j,i)
= j^{-1} a(n,j) (q(j,i) + q(j,j-i)) 
\end{align}
where the two sides both are
\begin{align}
 n^{-1} \Ex [\mbox{number of splits $j \to (i,j-i)$ or $(j-i,i)$ in $\DTCS(n)$} ],
\end{align}
calculated  in the two different directions.
\end{Remark}

\begin{figure}
\setlength{\unitlength}{0.087in}
\begin{picture}(80,50)(8,-30)
\multiput(1,0)(1,0){77}{\circle*{0.3}}
\put(77,0){\line(0,1){1}}
\put(76,0){\line(0,1){1}}
\put(76,1){\line(1,0){1}}
\put(75,0){\line(0,1){2}}
\put(75,2){\line(1,0){1.5}}
\put(76.5,1){\line(0,1){1}}

\put(74,0){\line(0,1){1}}
\put(73,0){\line(0,1){1}}
\put(73,1){\line(1,0){1}}
\put(72,0){\line(0,1){2}}
\put(72,2){\line(1,0){1.5}}
\put(73.5,1){\line(0,1){1}}

\put(75.5,2){\line(0,1){1}}
\put(72.5,2){\line(0,1){1}}
\put(72.5,3){\line(1,0){3}}
\put(73.5,3){\line(0,1){1}}
\put(71,0){\line(0,1){4}}
\put(71,4){\line(1,0){2.5}}

\put(70,0){\line(0,1){1}}
\put(69,0){\line(0,1){1}}
\put(69,1){\line(1,0){1}}

\put(68,0){\line(0,1){1}}
\put(67,0){\line(0,1){1}}
\put(67,1){\line(1,0){1}}

\put(69.5,1){\line(0,1){1}}
\put(67.5,1){\line(0,1){1}}
\put(67.5,2){\line(1,0){2}}

\put(66,0){\line(0,1){3}}
\put(68.5,2){\line(0,1){1}}
\put(66,3){\line(1,0){2.5}}
\put(65,0){\line(0,1){4}}
\put(68,3){\line(0,1){1}}
\put(65,4){\line(1,0){3}}

\put(66.5,4){\line(0,1){1}}
\put(71.75,4){\line(0,1){1}}
\put(66.5,5){\line(1,0){5.25}}
\put(68.5,5){\line(0,1){1}}

\put(64,0){\line(0,1){6}}
\put(64.0,6){\line(1,0){4.5}}
\put(65.5,6){\line(0,1){1}}
\put(63,0){\line(0,1){7}}
\put(63,7){\line(1,0){2.5}}
\put(64.25,7){\line(0,1){1}}

\put(62,0){\line(0,1){1}}
\put(61,0){\line(0,1){1}}
\put(61,1){\line(1,0){1}}

\put(61.5,1){\line(0,1){1}}
\put(60,0){\line(0,1){2}}
\put(60,2){\line(1,0){1.5}}
\put(60.75,2){\line(0,1){1}}
\put(59,0){\line(0,1){3}}
\put(59,3){\line(1,0){1.75}}
\put(59.75,3){\line(0,1){1}}

\put(58,0){\line(0,1){1}}
\put(57,0){\line(0,1){1}}
\put(57,1){\line(1,0){1}}

\put(57.5,1){\line(0,1){3}}
\put(57.5,4){\line(1,0){2.25}}
\put(58.25,4){\line(0,1){1}}

\put(54,0){\line(0,1){2}}
\put(55.5,1){\line(0,1){1}}
\put(54,2){\line(1,0){1.5}}
\put(53,0){\line(0,1){3}}
\put(54.5,2){\line(0,1){1}}
\put(53,3){\line(1,0){1.5}}
\put(53.75,3){\line(0,1){2}}
\put(53.75,5){\line(1,0){4.5}}
\put(55.75,5){\line(0,1){3}}
\put(55.75,8){\line(1,0){8.5}}

\put(56,0){\line(0,1){1}}
\put(55,0){\line(0,1){1}}
\put(55,1){\line(1,0){1}}

\put(52,0){\line(0,1){1}}
\put(51,0){\line(0,1){1}}
\put(51,1){\line(1,0){1}}
\put(51.5,1){\line(0,1){8}}
\put(59.5,8){\line(0,1){1}}
\put(51.5,9){\line(1,0){8}}

\put(50,0){\line(0,1){1}}
\put(49,0){\line(0,1){1}}
\put(49,1){\line(1,0){1}}

\put(48,0){\line(0,1){2}}
\put(49.5,1){\line(0,1){1}}
\put(48,2){\line(1,0){1.5}}

\put(47,0){\line(0,1){1}}
\put(46,0){\line(0,1){1}}
\put(46,1){\line(1,0){1}}
\put(46.5,1){\line(0,1){2}}
\put(48.5,2){\line(0,1){1}}
\put(46.5,3){\line(1,0){2}}

\put(45,0){\line(0,1){4}}
\put(47.25,3){\line(0,1){1}}
\put(45,4){\line(1,0){2.25}}

\put(44,0){\line(0,1){1}}
\put(43,0){\line(0,1){1}}
\put(43,1){\line(1,0){1}}

\put(43.5,1){\line(0,1){4}}
\put(46,4){\line(0,1){1}}

\put(43.5,5){\line(1,0){2.5}}
\put(44.5,5){\line(0,1){5}}
\put(54.5,9){\line(0,1){1}}
\put(44.5,10){\line(1,0){10}}

\put(42,0){\line(0,1){1}}
\put(41,0){\line(0,1){1}}
\put(41,1){\line(1,0){1}}

\put(41.5,1){\line(0,1){1}}
\put(40,0){\line(0,1){2}}
\put(40,2){\line(1,0){1.5}}
\put(40.75,2){\line(0,1){1}}
\put(39,0){\line(0,1){3}}
\put(39,3){\line(1,0){1.75}}
\put(39.75,3){\line(0,1){1}}
\put(38,0){\line(0,1){4}}

\put(38,4){\line(1,0){1.75}}
\put(38.75,4){\line(0,1){7}}
\put(48.5,10){\line(0,1){1}}
\put(38.75,11){\line(1,0){9.75}}

\put(37,0){\line(0,1){12}}
\put(43,11){\line(0,1){1}}
\put(37,12){\line(1,0){6}}

\put(36,0){\line(0,1){1}}
\put(35,0){\line(0,1){1}}
\put(35,1){\line(1,0){1}}
\put(35.5,1){\line(0,1){1}}
\put(34,0){\line(0,1){2}}
\put(34,2){\line(1,0){1.5}}
\put(33,0){\line(0,1){1}}
\put(32,0){\line(0,1){1}}
\put(32,1){\line(1,0){1}}
\put(32.5,1){\line(0,1){2}}
\put(34.5,2){\line(0,1){1}}
\put(32.5,3){\line(1,0){2}}

\put(31,0){\line(0,1){1}}
\put(30,0){\line(0,1){1}}
\put(30,1){\line(1,0){1}}
\put(30.5,1){\line(0,1){3}}
\put(33.5,3){\line(0,1){1}}
\put(30.5,4){\line(1,0){3}}

\put(29,0){\line(0,1){5}}
\put(32,4){\line(0,1){1}}
\put(29,5){\line(1,0){3}}

\put(28,0){\line(0,1){6}}
\put(30.5,5){\line(0,1){1}}
\put(28,6){\line(1,0){2.5}}

\put(28.75,6){\line(0,1){7}}
\put(39.5,12){\line(0,1){1}}
\put(28.75,13){\line(1,0){10.75}}

\put(27,0){\line(0,1){1}}
\put(26,0){\line(0,1){1}}
\put(26,1){\line(1,0){1}}
\put(26.5,1){\line(0,1){13}}
\put(33,13){\line(0,1){1}}
\put(26.5,14){\line(1,0){6.5}}

\put(25,0){\line(0,1){1}}
\put(24,0){\line(0,1){1}}
\put(24,1){\line(1,0){1}}
\put(24.5,1){\line(0,1){1}}
\put(23,0){\line(0,1){2}}
\put(23,2){\line(1,0){1.5}}

\put(23.5,2){\line(0,1){13}}
\put(29,14){\line(0,1){1}}
\put(23.5,15){\line(1,0){5.5}}

\put(22,0){\line(0,1){1}}
\put(21,0){\line(0,1){1}}
\put(21,1){\line(1,0){1}}

\put(21.5,1){\line(0,1){1}}
\put(20,0){\line(0,1){2}}
\put(20,2){\line(1,0){1.5}}
\put(20.75,2){\line(0,1){1}}
\put(19,0){\line(0,1){3}}
\put(19,3){\line(1,0){1.75}}
\put(19.75,3){\line(0,1){1}}

\put(18,0){\line(0,1){1}}
\put(17,0){\line(0,1){1}}
\put(17,1){\line(1,0){1}}

\put(17.5,1){\line(0,1){1}}
\put(16,0){\line(0,1){2}}
\put(16,2){\line(1,0){1.5}}
\put(16.75,2){\line(0,1){1}}
\put(15,0){\line(0,1){3}}
\put(15,3){\line(1,0){1.75}}
\put(15.75,3){\line(0,1){1}}

\put(15.75,4){\line(1,0){4}}
\put(17.75,4){\line(0,1){1}}

\put(14,0){\line(0,1){5}}
\put(14,5){\line(1,0){3.75}}
\put(15.5,5){\line(0,1){1}}

\put(13,0){\line(0,1){6}}
\put(13,6){\line(1,0){2.5}}
\put(14,6){\line(0,1){10}}
\put(26,15){\line(0,1){1}}
\put(14,16){\line(1,0){12}}

\put(12,0){\line(0,1){17}}
\put(19,16){\line(0,1){1}}
\put(12,17){\line(1,0){7}}

\put(11,0){\line(0,1){1}}
\put(10,0){\line(0,1){1}}
\put(10,1){\line(1,0){1}}

\put(9,0){\line(0,1){1}}
\put(8,0){\line(0,1){1}}
\put(8,1){\line(1,0){1}}

\put(10.5,1){\line(0,1){1}}
\put(8.5,1){\line(0,1){1}}
\put(8.5,2){\line(1,0){2}}

\put(9.5,2){\line(0,1){16}}
\put(15,17){\line(0,1){1}}
\put(9.5,18){\line(1,0){5.5}}

\put(7,0){\line(0,1){1}}
\put(6,0){\line(0,1){1}}
\put(6,1){\line(1,0){1}}

\put(6.5,1){\line(0,1){1}}
\put(5,0){\line(0,1){2}}
\put(5,2){\line(1,0){1.5}}
\put(5.75,2){\line(0,1){1}}
\put(4,0){\line(0,1){3}}
\put(4,3){\line(1,0){1.75}}
\put(4.75,3){\line(0,1){1}}
\put(3,0){\line(0,1){4}}

\put(3,4){\line(1,0){1.75}}
\put(3.75,4){\line(0,1){15}}
\put(12,18){\line(0,1){1}}
\put(3.75,19){\line(1,0){8.25}}

\put(2,0){\line(0,1){20}}
\put(7.5,19){\line(0,1){1}}
\put(2,20){\line(1,0){5.5}}

\put(1,0){\line(0,1){21}}
\put(4.5,20){\line(0,1){1}}
\put(1,21){\line(1,0){3.5}}
\put(2.5,21){\line(0,1){1}}
\end{picture}

\vspace*{-1.7in}

\includegraphics[width=5.1in]{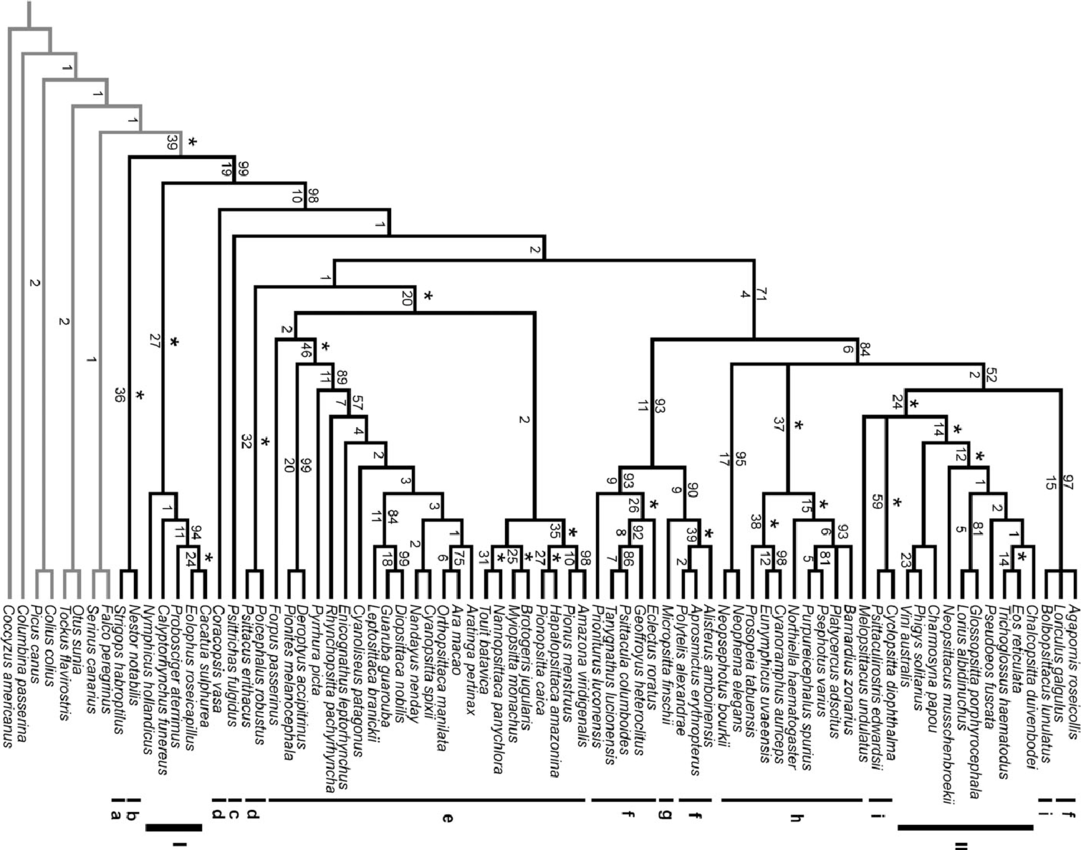}
 
 \caption{Bottom: cladogram showing phylogenetics of 77 parrot species, from \cite{parrot:phy}.  Top: simulation of $\DTCS(77)$, drawn as fringe distribution in the style of biological cladograms.}
\label{Fig:cladogram}
 \end{figure}

\subsection{Fringe trees and cladograms}
\label{sec:clad2}

As mentioned before, the model arose as a toy model for phylogenetic trees, designed to mimic the
uneven splits observed in real world examples.
The small-scale study \cite{me_yule} suggests
that in splits $m \to (i,m-i)$ in real-world phylogenetic trees, the median size of the smaller subtree scales roughly as $m^{1/2}$. 
That data is not consistent with more classical random tree models, where the median size would be 
$O(\log m)$  or $\Theta(m)$, 
but this $m^{1/2}$ median property does hold for our particular model. 
Figure  \ref{Fig:cladogram} compares a simulation of $\DTCS(77)$ with a real cladogram on 77 species; 
these appear visually quite similar.  
As shown in that figure, a cladogram is typically drawn upwards from the leaves, and we draw the fringe tree in the same way.
That is, one should visualize a fringe tree\footnote{There is no biological significance to the positioning of left/right branches, though a common convention is to position the larger subclade to the right.  
In our model, branches are randomly positioned left/right, but in  drawing Figure \ref{Fig:cladogram} (top) we followed the biological convention for visual comparison.}
as in Figure \ref{Fig:cladogram} (top), but with leaves labelled as
$\ldots, -2, -1, -0, 1, 2, \ldots$.

\subsection{Properties of the fringe tree}
\label{sec:propfringe}

For large $n$, a realization of $\DTCS(n)$ will contain many copies of small clades in its fringe.
In the asymptotic fringe tree, the probability that a given leaf is in {\em some} clade $\chi$ of size $i$ is just $a(i)$.
Because a clade of size $i$ has the $\DTCS(i)$ distribution, we can then calculate (numerically via recursion) the probability 
$p(\chi)$ that a leaf is in a {\em specific clade} $\chi$ of size $i$.
Some numerical results are shown in Figure \ref{small_clades}. 
In that figure we have grouped clades with the same {\em shape}, meaning that 
(as in the biology use) we do not distinguish left and right branches. 
Figure \ref{small_clades} compares these model predictions with the data from a small set of real cladograms\footnote{Dragonflies \cite{dragonfly:phy}, eagles \cite{eagles:phy}, elms \cite{elms:phy}, gamebirds \cite{gamebirds:phy}, 
ladybirds \cite{ladybirds:phy}, parrots\cite{parrot:phy}, primates \cite{primates:phy}, sharks \cite{sharks:phy}, snakes \cite{snakes:phy}, swallows \cite{swallows:phy}}
 -- 10 cladograms with a total of 995 species. 
 Further data will be given in \cite{beta2-arxiv}, and it is clear that the current model gives a better fit than other models such as those
 in \cite[Appendix A]{SJ385} and \cite{Ischebeck}.
Note that the
models treated in \cite{SJ385} are precisely the cases $\beta=\infty, 0, -3/2$
of the beta-splitting tree \cite{me_clad}.

But also one can use the fringe tree to study asymptotics of statistics of $\DTCS(n)$ or $\CTCS(n)$, for statistics which depend only on the structure of the tree near the leaves. 
In particular, the number $N_n(\chi)$ of copies of a size-$i$ clade $\chi$ in $\DTCS(n)$ will satisfy
$n^{-1} \Ex[N_n(\chi)] \to p(\chi)/i$.
By analogy with CLT results for other random tree models  \cite[section 14]{fringe}, and because occurrences of a given $\chi$ are only locally dependent,
we expect a CLT for $N_n(\chi)$, but have not attempted a proof.


\begin{figure}[ht]
\setlength{\unitlength}{0.06in}
\begin{picture}(60,80)(0,-64)
\multiput(0,0)(3,0){2}{\circle*{0.9}}
\put(0,0){\line(0,1){4}}
\put(3,0){\line(0,1){4}}
\put(0,4){\line(1,0){3}}
\put(1.5,4){\line(0,1){2}}
\put(-6,-4){$\sfrac{6}{\pi^2} \doteq 0.6079$}
\put(-1,-7){[0.573]}

\multiput(15,0)(3,0){3}{\circle*{0.9}}
\put(15,0){\line(0,1){8}}
\put(18,0){\line(0,1){4}}
\put(21,0){\line(0,1){4}}
\put(18,4){\line(1,0){3}}
\put(19.5,4){\line(0,1){4}}
\put(15,8){\line(1,0){4.5}}
\put(18,8){\line(0,1){2}}
\put(10,-4){$\sfrac{9}{2\pi^2} \doteq 0.4559$}
\put(16,-7){[0.491]}

\multiput(30,0)(3,0){4}{\circle*{0.9}}
\put(30,0){\line(0,1){12}}
\put(33,0){\line(0,1){8}}
\put(36,0){\line(0,1){4}}
\put(39,0){\line(0,1){4}}
\put(36,4){\line(1,0){3}}
\put(37.5,4){\line(0,1){4}}
\put(33,8){\line(1,0){4.5}}
\put(36,8){\line(0,1){4}}
\put(30,12){\line(1,0){6}}
\put(34.5,12){\line(0,1){2}}
\put(26.5,-4){$\sfrac{8}{3\pi^2} \doteq 0.2702$}
\put(32.3,-7){[0.285]}

\multiput(51,0)(3,0){4}{\circle*{0.9}}
\put(51,0){\line(0,1){4}}
\put(54,0){\line(0,1){4}}
\put(51,4){\line(1,0){3}}
\put(52.5,4){\line(0,1){4}}
\put(57,0){\line(0,1){4}}
\put(60,0){\line(0,1){4}}
\put(57,4){\line(1,0){3}}
\put(58.5,4){\line(0,1){4}}
\put(52.5,8){\line(1,0){6}}
\put(55.5,8){\line(0,1){2}}
\put(46.5,-4){$\sfrac{1}{\pi^2} \doteq 0.1013$}
\put(52.5,-7){[0.120]}

\multiput(69,0)(3,0){5}{\circle*{0.9}}
\put(69,0){\line(0,1){16}}
\put(72,0){\line(0,1){12}}
\put(75,0){\line(0,1){8}}
\put(78,0){\line(0,1){4}}
\put(81,0){\line(0,1){4}}
\put(78,4){\line(1,0){3}}
\put(79.5,4){\line(0,1){4}}
\put(75,8){\line(1,0){4.5}}
\put(78,8){\line(0,1){4}}
\put(72,12){\line(1,0){6}}
\put(76.5,12){\line(0,1){4}}
\put(69,16){\line(1,0){7.5}}
\put(75,16){\line(0,1){2}}
\put(66.5,-4){$\sfrac{15}{11\pi^2} \doteq 0.1382$}
\put(72.5,-7){[0.125]}


\multiput(0,-30)(3,0){5}{\circle*{0.9}}
\put(0,-30){\line(0,1){12}}
\put(0,-18){\line(1,0){7.5}}
\put(6,-18){\line(0,1){2}}
\put(3,-30){\line(0,1){4}}
\put(6,-30){\line(0,1){4}}
\put(9,-30){\line(0,1){4}}
\put(12,-30){\line(0,1){4}}
\put(3,-26){\line(1,0){3}}
\put(9,-26){\line(1,0){3}}
\put(4.5,-26){\line(0,1){4}}
\put(10.5,-26){\line(0,1){4}}
\put(4.5,-22){\line(1,0){6}}
\put(7.5,-22){\line(0,1){4}}
\put(-3,-34){$\sfrac{45}{88\pi^2} \doteq 0.0518$}
\put(3,-37){[0.055]}

\multiput(21,-30)(3,0){5}{\circle*{0.9}}
\put(21,-30){\line(0,1){4}}
\put(24,-30){\line(0,1){4}}
\put(21,-26){\line(1,0){3}}
\put(27,-30){\line(0,1){8}}
\put(30,-30){\line(0,1){4}}
\put(33,-30){\line(0,1){4}}
\put(30,-26){\line(1,0){3}}
\put(31.5,-26){\line(0,1){4}}
\put(27,-22){\line(1,0){4.5}}
\put(30,-22){\line(0,1){4}}
\put(22.5,-26){\line(0,1){8}}
\put(22.5,-18){\line(1,0){7.5}}
\put(27,-18){\line(0,1){2}}
\put(18,-34){$\sfrac{5}{4\pi^2} \doteq 0.1267$}
\put(24,-37){[0.145]}

\multiput(42,-30)(3,0){6}{\circle*{0.9}}
\put(42,-30){\line(0,1){20}}
\put(45,-30){\line(0,1){16}}
\put(48,-30){\line(0,1){12}}
\put(51,-30){\line(0,1){8}}
\put(54,-30){\line(0,1){4}}
\put(57,-30){\line(0,1){4}}
\put(54,-26){\line(1,0){3}}
\put(55.5,-26){\line(0,1){4}}
\put(51,-22){\line(1,0){4.5}}
\put(54,-22){\line(0,1){4}}
\put(48,-18){\line(1,0){6}}
\put(52.5,-18){\line(0,1){4}}
\put(45,-14){\line(1,0){7.5}}
\put(51,-14){\line(0,1){4}}
\put(42.5,-10){\line(1,0){9}}
\put(47,-10){\line(0,1){2}}
\put(40,-34){$\sfrac{864}{1375\pi^2} \doteq 0.0637$}
\put(46,-37){[0.030]}

\multiput(66,-30)(3,0){6}{\circle*{0.9}}
\put(66,-30){\line(0,1){16}}
\put(66,-14){\line(1,0){9}}
\put(70.5,-14){\line(0,1){2}}
\put(69,-30){\line(0,1){12}}
\put(69,-18){\line(1,0){7.5}}
\put(75,-18){\line(0,1){4}}
\put(72,-30){\line(0,1){4}}
\put(75,-30){\line(0,1){4}}
\put(78,-30){\line(0,1){4}}
\put(81,-30){\line(0,1){4}}
\put(72,-26){\line(1,0){3}}
\put(78,-26){\line(1,0){3}}
\put(73.5,-26){\line(0,1){4}}
\put(79.5,-26){\line(0,1){4}}
\put(73.5,-22){\line(1,0){6}}
\put(76.5,-22){\line(0,1){4}}
\put(65,-34){$\sfrac{324}{1375\pi^2} \doteq 0.0239$}
\put(71,-37){[0.024]}


\multiput(0,-56)(3,0){6}{\circle*{0.9}}
\put(0,-56){\line(0,1){16}}
\put(0,-40){\line(1,0){9}}
\put(4.5,-40){\line(0,1){2}}
\put(3,-56){\line(0,1){4}}
\put(6,-56){\line(0,1){4}}
\put(3,-52){\line(1,0){3}}
\put(9,-56){\line(0,1){8}}
\put(12,-56){\line(0,1){4}}
\put(15,-56){\line(0,1){4}}
\put(12,-52){\line(1,0){3}}
\put(13.5,-52){\line(0,1){4}}
\put(9,-48){\line(1,0){4.5}}
\put(12,-48){\line(0,1){4}}
\put(4.5,-52){\line(0,1){8}}
\put(4.5,-44){\line(1,0){7.5}}
\put(9,-44){\line(0,1){4}}
\put(-1.5,-60){$\sfrac{72}{125\pi^2} \doteq 0.0584$}
\put(4.5,-63){[0.042]}

\multiput(24,-56)(3,0){6}{\circle*{0.9}}
\put(24,-56){\line(0,1){4}}
\put(27,-56){\line(0,1){4}}
\put(24,-52){\line(1,0){3}}
\put(25.5,-52){\line(0,1){12}}
\put(25.5,-40){\line(1,0){9}}
\put(30,-40){\line(0,1){2}}
\put(30,-56){\line(0,1){12}}
\put(33,-56){\line(0,1){8}}
\put(36,-56){\line(0,1){4}}
\put(39,-56){\line(0,1){4}}
\put(36,-52){\line(1,0){3}}
\put(37.5,-52){\line(0,1){4}}
\put(33,-48){\line(1,0){4.5}}
\put(36,-48){\line(0,1){4}}
\put(30,-44){\line(1,0){6}}
\put(34.5,-44){\line(0,1){4}}
\put(22.5,-60){$\sfrac{36}{55\pi^2} \doteq 0.0663$}
\put(28.5,-63){[0.054]}

\multiput(48,-56)(3,0){6}{\circle*{0.9}}
\put(48,-56){\line(0,1){4}}
\put(51,-56){\line(0,1){4}}
\put(48,-52){\line(1,0){3}}
\put(49.5,-52){\line(0,1){8}}
\put(49.5,-44){\line(1,0){9}}
\put(54,-44){\line(0,1){2}}
\put(54,-56){\line(0,1){4}}
\put(57,-56){\line(0,1){4}}
\put(54,-52){\line(1,0){3}}
\put(55.5,-52){\line(0,1){4}}
\put(60,-56){\line(0,1){4}}
\put(63,-56){\line(0,1){4}}
\put(60,-52){\line(1,0){3}}
\put(61.5,-52){\line(0,1){4}}
\put(55.5,-48){\line(1,0){6}}
\put(58.5,-48){\line(0,1){4}}
\put(46.5,-60){$\sfrac{27}{110\pi^2} \doteq 0.0249$}
\put(52.5,-63){[0.030]}

\multiput(72,-56)(3,0){6}{\circle*{0.9}}
\put(72,-56){\line(0,1){8}}
\put(75,-56){\line(0,1){4}}
\put(78,-56){\line(0,1){4}}
\put(75,-52){\line(1,0){3}}
\put(76.5,-52){\line(0,1){4}}
\put(72,-48){\line(1,0){4.5}}
\put(75,-48){\line(0,1){4}}
\put(81,-56){\line(0,1){8}}
\put(84,-56){\line(0,1){4}}
\put(87,-56){\line(0,1){4}}
\put(84,-52){\line(1,0){3}}
\put(85.5,-52){\line(0,1){4}}
\put(81,-48){\line(1,0){4.5}}
\put(84,-48){\line(0,1){4}}
\put(75,-44){\line(1,0){9}}
\put(79.5,-44){\line(0,1){2}}
\put(70.5,-60){$\sfrac{2}{5\pi^2} \doteq 0.0405$}
\put(76.5,-63){[0.024]}

 \end{picture}
\caption{Proportions of leaves in clades of a given shape, for each shape with $2 - 6$ leaves in the fringe tree.
The top number is from our model, the bottom number $[ \cdots ]$ from our small data set.}
\label{small_clades}
\end{figure}
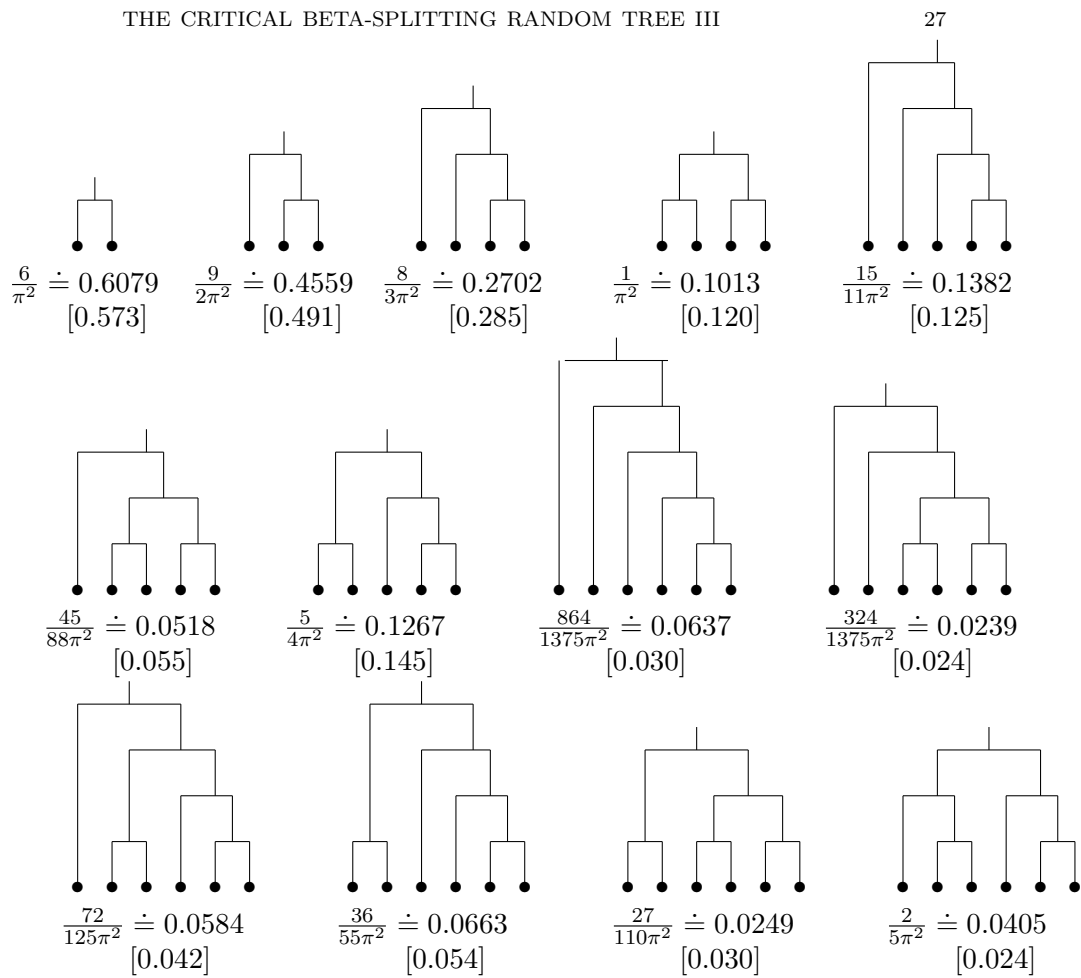

\subsection{Combinatorial questions}
Regarding  the number $N_n(\chi)$ of copies of a clade $\chi$ in $\DTCS(n)$,  
there are aspects which have not been studied (even within the usual random tree models).
For example one could study distributions of the following:
\begin{itemize}
 \item The number $K_n := \sum_\chi 1_{(N_n(\chi) \ge 1)}$ of different-shape clades within (a realization of) $\DTCS(n)$.
\item The largest clade that appears more than once within $\DTCS(n)$.
\item The smallest clade that does not appear within $\DTCS(n)$.
  \end{itemize}
  
  
The difficulty is that, although one can calculate each $p(\chi)$ numerically, we do not have a useful explicit description of the set of probabilities 
$(p(\chi): \ |\chi| = m)$ of size-$m$ clades.

\section{Proving Theorem \ref{T:alimit} via study of $\CTCS(\infty)$ and Mellin transforms}
\label{sec:surprise}
Having the exchangeable formalization of $\CTCS(\infty)$ 
leads to an alternate proof of the second
foundational result
(Theorem \ref{T:alimit}).
This is rather surprising, because convergence of $\CTCS(n)$ to  $\CTCS(\infty)$ seems a kind of ``global" convergence, whereas 
 the asymptotic fringe is a ``local" limit.
 It turns out that the Mellin transform method can be used to study
 many other aspects of the model, including leaf heights.
 Work in progress  will appear in \cite{beta4}, and the following is included here as an illustration of the methodology,
 
 Recall notation from Section \ref{sec:exch}.
The central idea of the proof is to define an infinite measure $\gU$ on
$\oio$ by 
\begin{align}\label{e7}
  \gU :=\intoo \cL(\XP_{t,1})\dd t.
\end{align}
 Formula \eqref{b5} immediately tells us the moments of the measure $\gU$: 
\begin{align}
\label{m1}
  \intoi x^{s-1}\dd\gU(x)
=\intoo \Ex [\XP_{t,1}^{s-1}]\dd t
= \frac{1}{\psi(s)-\psi(1)},
\qquad \Re s>1.
\end{align}
So this is the Mellin transform of $\gU$.
We do not know how to invert the transform to obtain an explicit formula for $\gU$, but what is relevant to the current proof is the behavior of $\gU$ near $0$, as follows.
\begin{Lemma}\label{LM}
  Let $\gU$ be the infinite
measure on $\oio$ having the Mellin transform
\eqref{m1}.
Then $\gU$ is absolutely continuous, with a continuous density $\gu(x)$ on $(0,1)$
that satisfies
\begin{align}\label{lmb}
\gu(x)
= \frac{6}{\pi^2x}+O\bigpar{x^{-s_1}+x^{-s_1}|\log x|\qw},
\end{align}
uniformly for $x\in(0,1)$, 
where $s_1\doteq -0.567$ is the largest negative root of $\psi(s)=\psi(1)$.
In particular, 
for $x\in(0,\frac12)$ say,
\begin{align}\label{lmob}
\gu(x)
= \frac{6}{\pi^2x}+O\bigpar{x^{-s_1}} \mbox{ as } x \downarrow 0
.\end{align}
\end{Lemma}
The (quite technical) proof of this ``inversion" lemma is given in
\refApp{invert}.

\subsection{Deriving Theorem \ref{T:alimit} }
\label{sec:pfMellin}
Here we show how to derive Theorem \ref{T:alimit}
 via  the exchangeable representation and Theorem \ref{T:exact} and Lemma \ref{LM}.

For $j\ge2$ let,
as in \refS{sec:OP},
$N_n(j)$ be the number of internal nodes of $\CTCS(n)$ that
have exactly $j$ leaves as descendants. Similarly, let
$N_n(j;t)$ be the number of clades of $\CTCS(n)$ at time $t$ that have size
exactly $j$.
Let
\begin{align}
\label{e1}
 e_n(j):=\Ex [N_n(j)],\qquad
e_n(j;t):=\Ex [N_n(j;t)].
\end{align}
The integral $\intoo N_n(j;t)\dd t$ equals the sum 
of the lifetimes of all clades of size $j$ that ever appear in $\CTCS(n)$. 
Because these lifetimes
have expectation $1/h_{j-1}$ (and are independent of the structure), we have
\begin{align}
\label{e2}
  \intoo e_n(j;t)\dd t = \Ex \Bigsqpar{ \intoo N_n(j;t)\dd t }
= \frac{1}{h_{j-1}}\Ex [N_n(j)]
= \frac{e_n(j)}{h_{j-1}}.
\end{align}
As noted previously
in \eqref{jup}, 
\begin{align}\label{e3}
  a(n,j) = \sfrac{j}{n}  e_n(j)
\end{align}
so to prove Theorem \ref{T:alimit} it will suffice to study the behavior of  $e_n(j)$.

To start to calculate $e_n(j)$, use the paintbox construction 
in Theorem \ref{T:paintbox}
to see that
\begin{quote}
conditioned on $(P_{t,\ell})_{\ell=1}^\infty$,
the probability that a given set of $j$ leaves form a clade at
time $t$ equals $\sum_\ell \XP_{t,\ell}^j(1-\XP_{t,\ell})^{n-j}$. 
\end{quote}
Thus,  by recalling
that $\XP_{t,1}$ can be regarded as a size-biased sample of
$(\XP_{t,\ell})_{\ell=1}^\infty$
\cite[Corollary 2.4]{bertoin}, we see
\begin{align}\label{e6}
  e_n(j;t) 
= \binom nj \Ex \Bigsqpar{\sum_\ell \XP_{t,\ell}^j(1-\XP_{t,\ell})^{n-j}}
= \binom nj \Ex \bigsqpar{ \XP_{t,1}^{j-1}(1-\XP_{t,1})^{n-j}}.
\end{align}
Recall that  $\gU$ is the infinite measure on $\oio$ given by
\begin{align}\label{e7b}
  \gU :=\intoo \cL(\XP_{t,1})\dd t.
\end{align}
Then \eqref{e6} yields
\begin{align}
\label{e8}
\intoo e_n(j;t)\dd t
= \binom nj \intoi x^{j-1}(1-x)^{n-j}\dd\gU(x)
\end{align}
and thus
\begin{align}
\frac{1}{n}\intoo e_n(j;t)\dd t
\sim \frac{n^{j-1}}{j!} \intoi x^{j-1}(1-x)^{n-j}\dd\gU(x)
\qquad\text{as \ntoo}.
\end{align}
Lemma \ref{LM} gives us the relevant information about the density $\gu(x)$ of $\gU$, and then we complete a proof of Theorem \ref{T:alimit} as follows.
  By \eqref{e2}, \eqref{e8} and Lemma \ref{LM},
  we have
\begin{align}\label{fr2}
\frac{1}{h_{j-1}} \
\frac{e_n(j)}{n}
=\frac{1}{n}\intoo e_n(j;t)\dd t 
= \frac{1}{n}\binom nj \intoi x^{j-1}(1-x)^{n-j}\gu(x)\dd x .
\end{align}
Substitute $\gu(x)$ from \eqref{lmb}. The main term becomes, using a standard beta integral,
\begin{align}\label{fr3}
  \frac{1}{n}\binom nj \intoi x^{j-1}(1-x)^{n-j}\frac{6}{\pi^2}x\qw\dd x&
=\frac{6}{\pi^2} \frac{1}{n}\binom nj \intoi x^{j-2}(1-x)^{n-j}\dd x
\\\notag&
=\frac{6}{\pi^2} \frac{1}{n}\binom nj \frac{\gG(j-1)\gG(n-j+1)}{\gG(n)}
\\\notag&
=\frac{6}{\pi^2} \frac{1}{n}\binom nj \frac{(j-2)!\,(n-j)!}{(n-1)!}
\\\notag&
=\frac{6}{\pi^2} \frac{1}{j(j-1)}.
\end{align}
The contribution from the error term in \eqref{lmb} 
has absolute value at most,
letting $C$ denote unimportant constants (not necessarily the same),
and using $-\log x > 1-x$,
\begin{align}\label{fr4}
&
 C \frac{1}{n}\binom nj 
\intoi x^{j-1}(1-x)^{n-j}x^{-s_1}\bigpar{1+|\log x|\qw}\dd x
\\\notag&\hskip8em
\le C n^{j-1} \intoi x^{j-s_1-1}(1-x)^{n-j-1}\dd x
\\\notag&\hskip8em
\le C n^{j-1} \intoo x^{j-s_1-1}e^{-(n-j-1)x}\dd x
\\\notag&\hskip8em
= C n^{j-1}(n-j-1)^{-j+s_1}
\\\notag&\hskip8em
= O\bigpar{n^{s_1-1}}.
\end{align}
This is $o(1)$ as \ntoo, and thus from 
\eqref{fr2} and \eqref{fr3}, 
for every fixed $j\ge2$,
  \begin{align}\label{tfr1}
 \frac{e_n(j)}{n}\to  
    \frac{6}{\pi^2}\frac{h_{j-1}}{j(j-1)}
.  \end{align}
Then by \eqref{e3} we get the assertion of Theorem \ref{T:alimit}: for $j \ge 2$,
\begin{align}\label{e14}
  a(n,j)\to  \frac{6}{\pi^2}\frac{h_{j-1}}{j-1} =: a(j) 
\end{align}
with the bound
\begin{equation}
\label{anjbd}
| a(n,j) - a(j)| = O\bigpar{n^{s_1-1}} \mbox{ as } n \to \infty .
\end{equation}
\qed

 \section{Final remarks}
 \subsection{The general beta-splitting model}
 \label{sec:motive}
 The mathematical theme of  \cite{me_clad} was to introduce the {\em beta-splitting model} with split probabilities
 \begin{equation}
 q(n,i) = \frac{1}{a_n(\beta)} \ \frac{\Gamma(\beta + i + 1) \Gamma(\beta +n - i + 1)} { \Gamma(i+1) \Gamma(n-i+1)} , \ 1 \le i \le n-1 
 \label{rule-1}
 \end{equation}
with a parameter $- 2 \le \beta \le \infty$ and normalizing constant $a_n(\beta)$.
 The qualitative behavior of the model is different for $\beta > -1$ than for $\beta < -1$;
in the former case the height (number of edges to the root) of a typical leaf grows as order $\log n$, 
and in the latter case as order $n^{-\beta - 1}$.
In this article we are studying the {\em critical}\footnote{Hence our terminology CS for {\em critical splitting}. But note that {\em critical} in our context is quite different from the usual {\em critical} in the context of branching processes or percolation.} case $\beta = -1$. 
The general beta-splitting model is often
mentioned in the mathematical biology literature 
on phylogenetics as one of several simple stochastic models.
  See \cite{lambert,steel} for recent overviews of that literature.

 \subsection{Analogies with and differences from the Brownian CRT}
 \label{sec:CRT}
As noted in the introduction, this article is part of a broad project investigating the random tree model.
 The document \cite{beta2-arxiv}, which will be periodically updated,  is intended to provide an overview: 
statements of results, proofs not given elsewhere, heuristics and open problems, data from phylogenetic trees,
 and general discussion. 
 Let us discuss below only one aspect of the project.
 
 The best known continuous limit of finite random tree models is the Brownian 
continuum random tree (CRT) \cite{crt2,crt3,evans,goldschmidt2}, 
which is a scaling limit of conditioned Galton-Watson trees and other ``uniform random tree" models.
 Our $\CTCS(\infty)$ model can also be regarded as a scaling limit\footnote{The ``scaling" in $\CTCS(n)$  arises from the initial splitting rate being $h_{n-1}$; 
in the finite uniform random tree model we scale edge lengths to be order $n^{-1/2}$.}  
 of $\CTCS(n)$. 
How do these compare?
 
 \smallskip
 \noindent ({\bf a})  The most convenient formalization of the Brownian CRT is as a random {\em measured metric space}, with the
Gromov-Hausdorff-Prokhorov topology \cite{ADH} on the set of all such spaces.  So one automatically has a notion of convergence in distribution.
Our formalization of $\CTCS(\infty)$ via exchangeable partitions is less amenable to rephrasing as a random element of some metric space. 
For instance it is easy to visualize Brownian motion  \cite{andriopoulos2024cover} on a realization of the CRT, but it seems harder to visualize a stochastic process on a realization of $\CTCS(\infty)$.

 \smallskip
 \noindent  ({\bf b}) Our consistency  result, that $\CTCS(n)$ is consistent as $n$ increases, and exchangeable over the random leaves, 
constitutes one general approach to the construction of continuum random trees (CRTs) \cite{crt3,evans}.

\smallskip
 \noindent ({\bf c}) Our explicit inductive construction (growth algorithm) is analogous to the line-breaking constructions of the Brownian CRT \cite{crt2} and stable trees \cite{linebreak}.

\smallskip
 \noindent ({\bf d})  Haas et al \cite{haas-pitman} and subsequent work such as \cite{haas-miermont}
have given a detailed general treatment of self-similar fragmentations via  exchangeable partitions, though the focus there
is on characterizations and on models like the 
$-2 < \beta < -1$ case of the beta-splitting model. 
On the range $-2 < \beta < -1$ , such models have limits which are qualitatively analogous to the Brownian continuum random tree, 
which is the case $\beta = -2$.

\smallskip
 \noindent ({\bf e}) It is implausible that $\CTCS(\infty)$ is as ``universal" a limit as the Brownian CRT has proved to be, but nevertheless one can ask
{\em Are there superficially different discrete models whose limit is the same $\CTCS(\infty)$?}
The key feature of our model seems to be the subordinator approximation \eqref{approx}: can this arise in some other model?

\section*{Acknowledgments}
Thanks to Boris Pittel for extensive interactions regarding this project.
Thanks to Serte Donderwinkel for pointing out a gap in an early version,
 and to Jim Pitman and David Clancy and Prabhanka Deka for helpful comments
 on  early versions. 
 Thanks especially to B\'{e}n\'{e}dicte Haas for his careful explanation of how our setting fits into the general theory of exchangeable random partitions,
 which is the basis of our Section \ref{sec:exch}, as noted in Remark \ref{R:Haas}.
 
For recent 
alternative proofs mentioned in the text we thank  Brett
Kolesnik,   Luca Pratelli and Pietro Rigo, 
and in particular Alexander Iksanov, whose observation of the connection
with  regenerative composition structures 
may lead to interesting further results.

\newpage
\appendix

\section{More on the consistency property and growth algorithm}
\label{sec:proofCP}

We give in this section our original proof  of the consistency property \refT{T:consistent}
and the growth algorithm \refT{T:growth}.  This proof uses only straightforward
(although rather long) calculations of (conditional) probabilities.

Recall the statement of Theorem  \ref{T:consistent}:
 {\em The operation ``delete and prune leaf $n+1$ from $\CTCS(n+1)$" gives a tree distributed as $\CTCS(n)$.}



Consider a pair of trees $(\bt_n, \bt_{n+1})$ 
with $n$ and $n+1$ leaves,
in which $\bt_n$ can be obtained from $\bt_{n+1}$ via the
``delete a leaf and prune" operation in Figure \ref{Fig:dpo}.
So $\bt_{n+1}$ arises by adding a new bud to $\bt_n$,
which can happen in one of three qualitative ways illustrated in Figure
\ref{Fig:3b}.

We will do explicit calculations for (CTCS(3), CTCS(4)) in the following section.
This enables one to guess the growth algorithm in \refT{T:growth}, which
we 
verify for general $n$ in Section \ref{sec:general_step}.
In the argument below, it will be convenient to use unlabelled leaves, 
cf.\ \refR{R:unlabelled}.

\subsection{A starting step}
The distribution of $\CTCS(n)$ is specified by the shape of the tree and the probability density of the edge-lengths.
For $n = 3$ there are only two possible shapes, as  $\bt$ in Figure \ref{Fig:density} and as its ``reflection" with the side-bud on the left instead of the right.
There are two  edge-lengths $(a,b)$. Clearly the density of CTCS(3) is
\begin{equation}
 f(\bt; a,b) =  \sfrac{1}{2}  h_2 e^{-h_2a} \cdot e^{-b}; \quad 0 < a, b < \infty   
 \label{fab}
 \end{equation}
and the probability of $\bt$ is $1/2$.
There are 7 shapes of CTCS(4) that are consistent with this $\bt$, shown as $\bt_1, \bt_2, \bt_3, \bt_4$ in Figure \ref{Fig:density}, together with the ``reflected"
forms of $\bt_1$ and $\bt_3$ and $\bt_4$ (the added side-bud involves the other side; drawn as $\hat{\mathbf{t}}_1, \hat{\mathbf{t}}_3, \hat{\mathbf{t}}_4$)
which will be accounted for as $q(\cdot,\cdot) + q(\cdot,\cdot) $ terms in
the calculation below.\footnote{That is, $f^+_1$ is the density of $\bt_1$
  plus the density of $\hat{\mathbf{t}}_1$.}
The densities of these shapes involve 3 edge-lengths $(a,b,c)$, calculated below as $f^+_i(a,b,c)$.
We also calculate the marginals $f_i(a,b) = \int f^+_i(a,b,c) \dd c$.

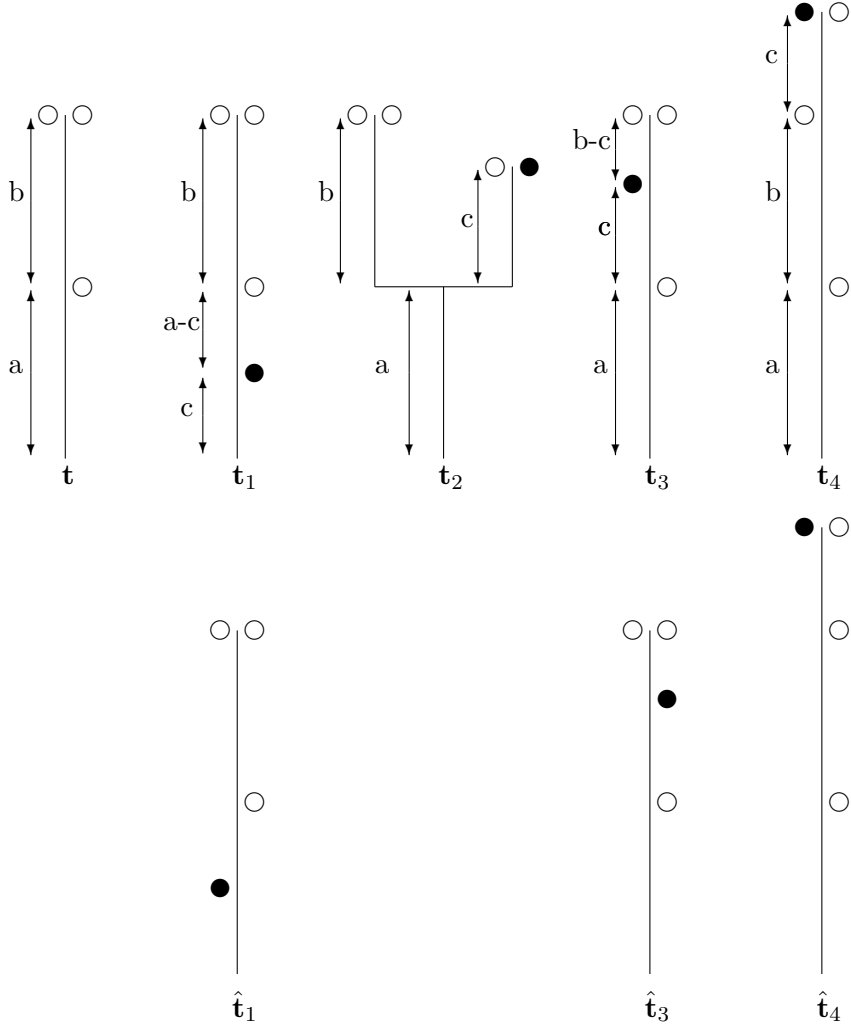
\begin{figure}[ht]
\setlength{\unitlength}{0.09in}
\begin{picture}(60,60)(-3,-33)

\put(0,0){\line(0,1){20}}
\put(1,10){\circle{1}}
\put(1,20){\circle{1}}
\put(-1,20){\circle{1}}
\put(-3.3,5){a}
\put(-2,5){\vector(0,1){4.8}} 
\put(-2,5){\vector(0,-1){4.8}} 
\put(-3.3,15){b}
\put(-2,15){\vector(0,1){4.8}} 
\put(-2,15){\vector(0,-1){4.8}} 
\put(-0.2,-1.5){\bf t}

\put(10,0){\line(0,1){20}}
\put(11,10){\circle{1}}
\put(11,20){\circle{1}}
\put(9,20){\circle{1}}
\put(6.7,15){b}
\put(8,15){\vector(0,1){4.8}} 
\put(8,15){\vector(0,-1){4.8}} 
\put(11,5){\circle*{1}}
\put(8,2.5){\vector(0,-1){2.2}} 
\put(8,2.5){\vector(0,1){2.2}} 
\put(6.7,2.5){c}
\put(8,7.5){\vector(0,-1){2.2}} 
\put(8,7.5){\vector(0,1){2.2}} 
\put(5.7,7.5){a-c}
\put(9.7,-1.5){$\mathbf{t}_1$}

\put(22,0){\line(0,1){10}}
\put(18,10){\line(1,0){8}}
\put(18,5){a}
\put(20,5){\vector(0,1){4.8}} 
\put(20,5){\vector(0,-1){4.8}} 
\put(18,10){\line(0,1){10}}
\put(26,10){\line(0,1){7}}
\put(17,20){\circle{1}}
\put(19,20){\circle{1}}
\put(25,17){\circle{1}}
\put(27,17){\circle*{1}}
\put(14.7,15){b}
\put(16,15){\vector(0,1){4.8}} 
\put(16,15){\vector(0,-1){4.8}} 
\put(23,13.5){c}
\put(24,13.5){\vector(0,1){3,3}} 
\put(24,13.5){\vector(0,-1){3,3}} 
\put(21.7,-1.5){$\mathbf{t}_2$}

\put(34,0){\line(0,1){20}}
\put(35,10){\circle{1}}
\put(35,20){\circle{1}}
\put(33,20){\circle{1}}
\put(30.7,5){a}
\put(32,5){\vector(0,1){4.8}} 
\put(32,5){\vector(0,-1){4.8}} 
\put(33,16){\circle*{1}}
\put(31,13){c}
\put(32,13){\vector(0,1){2.8}} 
\put(32,13){\vector(0,-1){2.8}} 
\put(31,13){c}
\put(32,18){\vector(0,1){1.8}} 
\put(32,18){\vector(0,-1){1.8}} 
\put(29.5,18){b-c}
\put(33.7,-1.5){$\mathbf{t}_3$}

\put(44,0){\line(0,1){26}}
\put(45,10){\circle{1}}
\put(43,20){\circle{1}}
\put(40.7,5){a}
\put(42,5){\vector(0,1){4.8}} 
\put(42,5){\vector(0,-1){4.8}} 
\put(40.7,15){b}
\put(42,15){\vector(0,1){4.8}} 
\put(42,15){\vector(0,-1){4.8}} 
\put(45,26){\circle{1}}
\put(43,26){\circle*{1}}
\put(40.7,23){c}
\put(42,23){\vector(0,1){2.8}} 
\put(42,23){\vector(0,-1){2.8}} 
\put(43.7,-1.5){$\mathbf{t}_4$}


\put(10,-30){\line(0,1){20}}
\put(11,-20){\circle{1}}
\put(11,-10){\circle{1}}
\put(9,-10){\circle{1}}
\put(9,-25){\circle*{1}}
\put(9.7,-32.5){$\hat{\mathbf{t}}_1$}

\put(34,-30){\line(0,1){20}}
\put(35,-20){\circle{1}}
\put(35,-10){\circle{1}}
\put(33,-10){\circle{1}}
\put(35,-14){\circle*{1}}
\put(33.7,-32.5){$\hat{\mathbf{t}}_3$}

\put(44,-30){\line(0,1){26}}
\put(45,-20){\circle{1}}
\put(45,-10){\circle{1}}
\put(45,-4){\circle{1}}
\put(43,-4){\circle*{1}}
\put(43.7,-32.5){$\hat{\mathbf{t}}_4$}

\end{picture}
\caption{The  possible transitions from $\bt$: the added bud is $\bullet$.}
\label{Fig:density}
\end{figure}

\noindent
The consistency assertion that we wish to verify is the assertion, for $f = f(\bt; \cdot, \cdot)$ as at \eqref{fab}, 
\begin{equation}
 f \overset{?}{=}  \sfrac{1}{4} f_1 + \sfrac{2}{4}f_2 +  \sfrac{1}{4} f_3 + \sfrac{2}{4}f_4 
 \label{fab2}
 \end{equation}
where the fractions denote the probability that deleting a random bud gives $\bt$ with the given edge-lengths $(a,b)$.
From the definition of CTCS(4) we can calculate
\begin{align}
 f^+_1(a,b,c) &= (q(4,1) + q(4,3)) \cdot h_3 e^{-h_3c} \cdot q(3,2) \cdot  h_2 e^{-h_2 (a-c)} \cdot e^{-b} \\
 f_1(a,b) &= (q(4,1) + q(4,3)) \cdot h_3 \cdot 3(1 - e^{-a/3})  \cdot q(3,2) \cdot h_2 e^{-h_2 a} \cdot e^{-b} \\\notag
  &= 3(e^{-h_2a} - e^{-h_3a})\cdot  e^{-b} \\
 f^+_2(a,b,c) &= q(4,2) \cdot h_3 e^{-h_3 a} \cdot e^{-b} e^{-c} \\
f_2(a,b) &= q(4,2) \cdot h_3 e^{-h_3 a} \cdot e^{-b}   \\\notag
&= \sfrac{1}{2} e^{-h_3a} \cdot e^{-b} \\
f^+_3(a,b,c) &= q(4,3) \cdot  h_3 e^{-h_3 a} \cdot (q(3,2) + q(3,1)) \cdot h_2 e^{-h_2c} \cdot e^{-(b-c)} \\
f_3(a,b) &= q(4,3) \cdot  h_3 e^{-h_3 a} \cdot (q(3,2) + q(3,1)) \cdot h_2 \cdot 2(1 - e^{-b/2}) \cdot e^{-b} \\\notag
&=  2 e^{-h_3a} (1 - e^{-b/2}) \cdot e^{-b} \\
f^+_4(a,b,c) &= q(4,3) \cdot  h_3 e^{-h_3 a} \cdot (q(3,2) + q(3,1)) \cdot h_2 e^{-h_2b} \cdot e^{-c} \\
f_4(a,b) &= q(4,3) \cdot  h_3 e^{-h_3 a} \cdot (q(3,2) + q(3,1)) \cdot h_2 e^{-h_2b} \\\notag
&= e^{-h_3a} \cdot e^{-h_2b} .
  \end{align}
  From this we can verify \eqref{fab2}.
  
This argument is not so illuminating, but we can immediately derive the conditional distribution of CTCS(4) given that CTCS(3) is $(\bt, a, b)$.
Writing $g_i(c|a,b)$ for the conditional density of shape $\bt_i$ or $\hat{\bt}_i$ and additional edge length $c$, 
and $p(\bt_i | a,b) = \int g_i(c|a,b) \ dc$ for the conditional probability of shape $\bt_i$ or  $\hat{\bt}_i$, 
we have 
\begin{align}
g_1(c|a,b) = \frac{ \frac{1}{4} f^+_1(a,b,c)}{f(a,b)} &= \sfrac{1}{3} e^{-c/3} ; \quad p(\mathbf{t}_1 | a,b) = 1 - e^{-a/3} \\
 g_2(c|a,b) = \frac{ \frac{1}{2} f^+_2(a,b,c)}{f(a,b)} &= \sfrac{1}{3} e^{-a/3} \cdot e^{-c} ; \quad p(\mathbf{t}_2 | a,b) = \sfrac{1}{3}  e^{-a/3}   \\
 g_3(c|a,b) = \frac{ \frac{1}{4} f^+_3(a,b,c)}{f(a,b)} &= \sfrac{1}{3} e^{-a/3} \cdot e^{-c/2} ; \quad p(\mathbf{t}_3 | a,b) = \sfrac{2}{3}  e^{-a/3}   (1 - e^{-b/2}) \\
g_4(c|a,b) = \frac{ \frac{1}{2} f^+_4(a,b,c)}{f(a,b)} &= \sfrac{2}{3} e^{-a/3} e^{-b/2} \cdot e^{-c} ; \quad p(\mathbf{t}_4 | a,b) = \sfrac{2}{3}  e^{-a/3}  \cdot  e^{-b/2}  .
\end{align}

One can now see that these are the conditional probabilities that arise from
the growth algorithm in 
\refT{T:growth}
which we for convenience repeat:
\smallskip

{\em
  Given a realization of\/ $\CTCS(n)$ for some $n \ge 1$ (above, $n=3$):
 \begin{enumerate}
 \item Pick a uniform random bud; move up the path from the root toward that
   bud.
 A \stopp{} event occurs at rate = $1$/(size of clade from current position).
 \item If\/ \stopp{} before reaching the target bud, make a side-bud at that point, random on left or right. (As in $\bt_1$ or $\bt_3$ above.)
 \item Otherwise, extend the target bud into a branch of\/ $\Expo(1)$
   length to make a bud-pair.   (As in $\bt_2$ or $\bt_4$ above.)
\end{enumerate}
}%
Figure \ref{Fig:density} indicated three of these possibilities $(\bt_1, \bt_3, \bt_4)$ when the chosen target bud was at the top right.  
The ``rate" is $1/3$ until the side-bud, and then $1/2$. 
Note that case $\bt_2$ arises as an ``extend the target bud" for a different target bud.

\subsection{The general step}
\label{sec:general_step}
To set up a calculation, we
consider the side-bud addition
case first, illustrated by the example in Figure \ref{Fig:4},
where the left diagram shows the relevant part of $\bt_n$ and the right diagram shows the side-bud addition making $\bt_{n+1}$.
The $\ell_i$ are edge-lengths and the $(n_i)$ are clade sizes.  
The side-bud is attached to some edge, in Figure \ref{Fig:4} 
an edge at edge-height $4$ with length $\ell_4$ and defining a clade of size $n_4 \ge 2$.
The new bud splits that edge into edges of length $\alpha$ and $\ell_4 - \alpha$.
The probability density function on a given tree is a product of terms for each edge.
Table \ref{table:1} shows the terms for the edges where the terms differ between the two trees -- 
these are only the edges on the path from the root to the added bud.
The first three lines in Table \ref{table:1} refer to the edges below the old edge into which the new bud is inserted, and the bottom line refers
to that old edge.

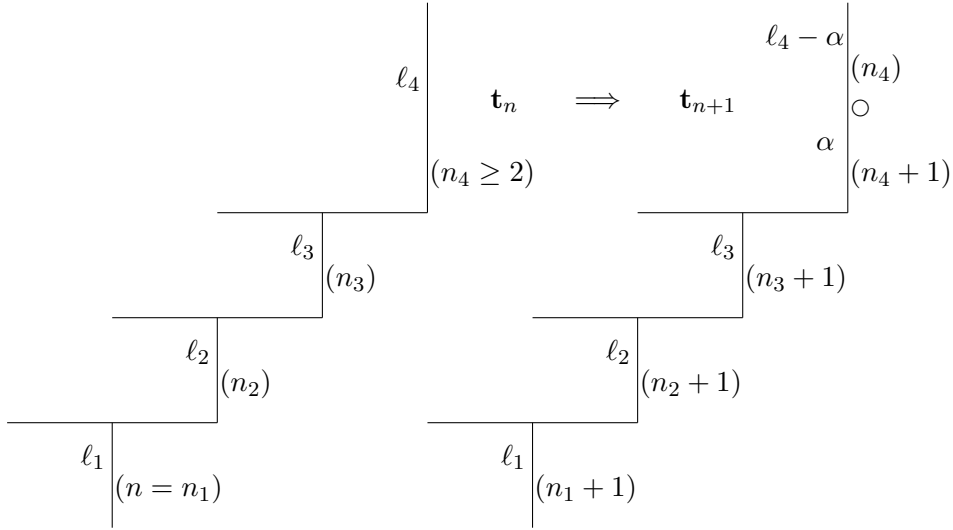
\begin{figure}[ht]
\setlength{\unitlength}{0.11in}
\begin{picture}(40,27)(0,0)

\put(0,0){\line(0,1){5}}
\put(0.1,1.5){($n = n_1$)}
\put(-1.5,3){$\ell_1$}
\put(5,5){\line(-1,0){10}}
\put(5,5){\line(0,1){5}}
\put(5.1,6.5){($n_2$)}
\put(3.5,8){$\ell_2$}
\put(10,10){\line(-1,0){10}}
\put(10,10){\line(0,1){5}}
\put(10.1,11.5){($n_3$)}
\put(8.5,13){$\ell_3$}
\put(15,15){\line(-1,0){10}}
\put(15,15){\line(0,1){10}}
\put(15.1,16.5){($n_4 \ge 2$)}
\put(13.5,21){$\ell_4$}

\put(20,0){\line(0,1){5}}
\put(20.1,1.5){($n_1 +1$)}
\put(18.5,3){$\ell_1$}
\put(25,5){\line(-1,0){10}}
\put(25,5){\line(0,1){5}}
\put(25.1,6.5){($n_2 +1$)}
\put(23.5,8){$\ell_2$}
\put(30,10){\line(-1,0){10}}
\put(30,10){\line(0,1){5}}
\put(30.1,11.5){($n_3 + 1$)}
\put(28.5,13){$\ell_3$}
\put(35,15){\line(-1,0){10}}
\put(35,15){\line(0,1){10}}
\put(35.1,16.5){($n_4 +1$)}
\put(35.1,21.5){($n_4$)}
\put(33.5,18){$\alpha$}
\put(31.1,23){$\ell_4 - \alpha$}
\put(35.6,20){\circle{0.8}}

\put(18,20){$\bt_n$}
\put(27,20){$\bt_{n+1}$}
\put(22,20){$\Longrightarrow$}

\end{picture}
\caption{Growing via a side-bud addition}
\label{Fig:4}
\end{figure}


\begin{table}[ht]
$
\begin{array}{c|c}
\mbox{left tree} & \mbox{right tree}\\
h_{n_1-1} \exp(- h_{n_1-1} \ell_1) \ddx \ell_1 \ \cdot \ q(n_1,n_2) 
& 
h_{n_1} \exp(- h_{n_1} \ell_1) \ddx \ell_1 \ \cdot \ q(n_1+1,n_2+1) \\
h_{n_2-1} \exp(- h_{n_2-1} \ell_2) \ddx \ell_2 \ \cdot \ q(n_2,n_3) 
& 
h_{n_2} \exp(- h_{n_2} \ell_2) \ddx \ell_2 \ \cdot \ q(n_2+1,n_3+1) \\
h_{n_3-1} \exp(- h_{n_3-1} \ell_3) \ddx \ell_3 \ \cdot \ q(n_3,n_4) 
& 
h_{n_3} \exp(- h_{n_3} \ell_3) \ddx \ell_3 \ \cdot \ q(n_3+1,n_4+1) \\
\hline 
h_{n_4-1} \exp(- h_{n_4-1} \ell_4) \ddx \ell_4 
& 
h_{n_4} \exp(- h_{n_4} \alpha) \ddx \alpha \ \cdot \ q(n_4+1,1)\\
&
\cdot h_{n_4 -1} \exp(- h_{n_4 -1} (\ell_4 - \alpha) ) \ddx \ell_4
\end{array}
$
\caption{Differing terms in density product (side-bud case)}
\label{table:1}
\end{table}

\begin{table}[ht!]
$
\begin{array}{l|r}
& \\
\hline 
h_{n_4-1} \exp(- h_{n_4-1} \ell_4) \ddx \ell_4 
& 
h_{n_4} \exp(- h_{n_4} \ell_4) \ddx \ell_4 \ \cdot \ q(n_4+1,1)\\
&
\cdot h_{1} \exp(- h_{1} \beta ) \ddx \beta 
\end{array}
$
\caption{Differing terms in density product (branch extension case)}
\label{table:2}
\end{table}

Because $h_{n-1}q(n,k) = \frac{n}{2k(n-k)}$ the ratios right/left of each of the first 3 lines in  Table \ref{table:1} equal
\begin{align} 
\frac{n_i+1}{n_i} \cdot \frac{n_{i+1}}{n_{i+1}+1} \cdot \exp(-\ell_i/n_i) , 
\qquad i = 1, 2, 3. 
\end{align}
The corresponding ratio for the final term equals
\begin{align}
\frac{n_4+1}{2n_4} \cdot \exp(-\alpha/n_4) \dd \alpha . 
\end{align}
Combining terms, the ratio of densities equals
\begin{align}
\frac{n+1}{2n} \cdot \exp(- \ell_1/n_1 - \ell_2/n_2 - \ell_3/n_3 - \alpha/n_4) \dd \alpha . 
\end{align}
In obtaining $\bt_n$ from $\bt_{n+1}$ we chose  one of $n+1$ buds  to delete, so finally the conditional density of CTCS(n+1) given $\CTCS(n)$ at $(\bt_{n+1} | \bt_n)$ 
equals
\begin{equation}
 \frac{1}{2n} \cdot \exp(- \ell_1/n_1 - \ell_2/n_2 - \ell_3/n_3 - \alpha/n_4) \dd \alpha . 
 \label{up1}
 \end{equation}
 We need to check that this agrees with the growth algorithm.
 According to the algorithm, the conditional density is a product of terms
 \begin{itemize}
 \item $n_4/n$: the chance that the target bud is in the relevant clade;
 \item $\exp(- \ell_1/n_1 - \ell_2/n_2 - \ell_3/n_3)$: the chance of not stopping before reaching the edge of length $\ell_4$;
 \item $\frac{1}{n_4} \exp(-\alpha/n_4) \dd \alpha$: the chance of stopping in $\ddx \alpha$;
 \item $1/2$: chance of placing side-bud on right side.
 \end{itemize}
 And this agrees with \eqref{up1}.
 
That was the side-bud addition case.
Now consider the  branch extension case,  illustrated in Figure \ref{Fig:5}.
In this case, $\bt_n$ has an edge terminating in two buds.
Then $\bt_{n+1}$ is obtained by extending the branch by an extra edge of some length $\beta$ to two terminal buds, leaving one bud as a side-bud.
Comparing the densities of $\bt_n$ and $\bt_{n+1}$ in this case,
the first 3 lines are the same as in Table \ref{table:1}, and the 4th is shown in Table \ref{table:2}.
Following the previous argument we derive the conditional density in a 
format similar to \eqref{up1}:
\begin{equation}
 \frac{1}{2n} \cdot \exp(- \ell_1/n_1 - \ell_2/n_2 - \ell_3/n_3 - \ell_4/n_4 - \beta ) \dd\beta . 
 \label{up2}
 \end{equation}
 Again this agrees with the growth algorithm.
 The third case, the side-bud extension, is similar.
 \qed

\begin{figure}
\setlength{\unitlength}{0.12in}
\begin{picture}(40,27)(0,0)

\put(0,0){\line(0,1){5}}
\put(0.1,1.5){($n = n_1$)}
\put(-1.5,3){$\ell_1$}
\put(5,5){\line(-1,0){10}}
\put(5,5){\line(0,1){5}}
\put(5.1,6.5){($n_2$)}
\put(3.5,8){$\ell_2$}
\put(10,10){\line(-1,0){10}}
\put(10,10){\line(0,1){5}}
\put(10.1,11.5){($n_3$)}
\put(8.5,13){$\ell_3$}
\put(15,15){\line(-1,0){10}}
\put(15,15){\line(0,1){5}}
\put(15.1,16.5){($n_4 = 2$)}
\put(13.5,18){$\ell_4$}
\put(15.6,20){\circle{0.8}}
\put(14.4,20){\circle{0.8}}

\put(20,0){\line(0,1){5}}
\put(20.1,1.5){($n_1 +1$)}
\put(18.5,3){$\ell_1$}
\put(25,5){\line(-1,0){10}}
\put(25,5){\line(0,1){5}}
\put(25.1,6.5){($n_2 +1$)}
\put(23.5,8){$\ell_2$}
\put(30,10){\line(-1,0){10}}
\put(30,10){\line(0,1){5}}
\put(30.1,11.5){($n_3 + 1$)}
\put(28.5,13){$\ell_3$}
\put(35,15){\line(-1,0){10}}
\put(35,15){\line(0,1){10}}
\put(35.1,16.5){($n_4 +1 = 3$)}
\put(35.1,21.5){($n_4 =  2$)}
\put(33.5,18){$\ell_4$}
\put(33.5,23){$\beta$}
\put(34.4,20){\circle{0.8}}
\put(35.6,25){\circle{0.8}}
\put(34.4,25){\circle{0.8}}

\put(18,20){$\bt_n$}
\put(27,20){$\bt_{n+1}$}
\put(22,20){$\Longrightarrow$}

\end{picture}
\caption{Growing via a  branch extension}
\label{Fig:5}
\end{figure}
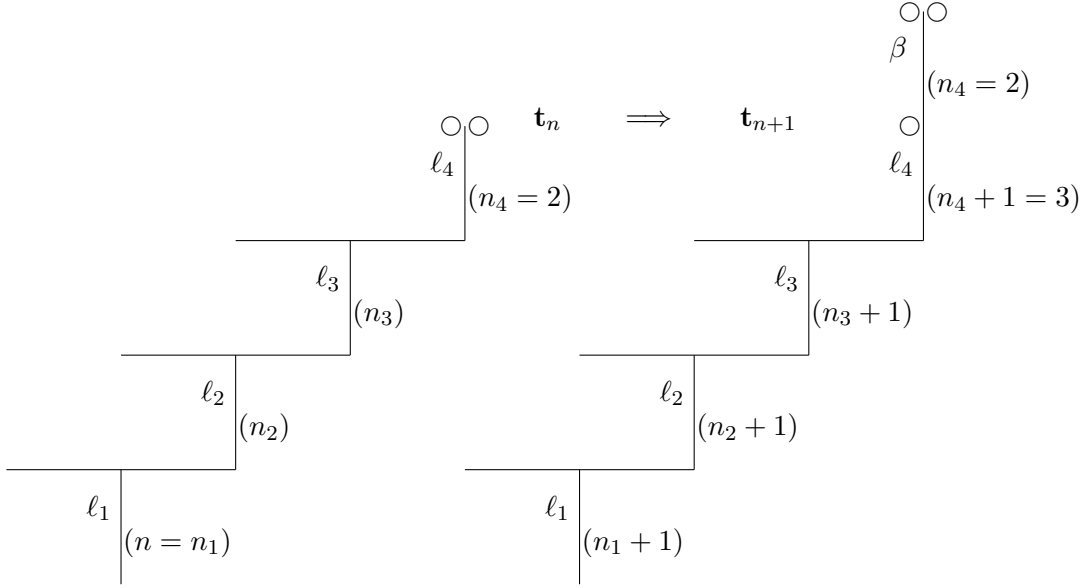

\section{Another proof of Proposition \ref{CP:split}}
 \label{sec:Exp1}
In $\CTCS(n)$ write 
$(X_n(i,t), i \ge 1)$ for the clade sizes at time $t$ and consider
\begin{align}
Q_n(t) = \sum_i X_n^2(i,t).
\end{align}
Note that, when a size-$m$ clade is split, the effect on sum-of-squares of clade sizes has expectation
\begin{equation}
\sum_{i=1}^{m-1}  (m^2 - i^2 - (m-i)^2) \ q(m,i) 
= \frac{m}{2h_{m-1}}  \sum_{i=1}^{m-1}  2 = \frac{m(m-1)}{h_{m-1}}  .
\label{id1}
\end{equation}
If we chose some arbitrary rates $r(m,n)$ for splitting a size-$m$ clade, then
\begin{align}
\Ex [Q_n(t) - Q_n(t+\ddx t) \vert \FF_t] = \sum_i r(X_n(i,t),n) \  \frac{X_n(i,t) (X_n(i,t)-1)}{h_{X_n(i,t)-1}}  \dd t .
\end{align}
So by choosing $r(m,n) = h_{m-1}$ we obtain
\begin{align} 
\Ex [Q_n(t) - Q_n(t+\ddx t) \vert \FF_t] = (Q_n(t) - n) \dd t. 
\end{align}
Because $Q_n(0) = n^2$ we obtain the exact formula
\begin{equation}
 \Ex [Q_n(t)] = n + (n^2  - n)e^{-t}, \ 0 \le t < \infty . 
 \label{EQn}
  \end{equation}
Now we are studying the height $B_n$ of the branchpoint between the paths to two uniform random distinct leaves of $\CTCS(n)$.
The conditional probability that both sampled leaves are in clade $i$ at time $t$ equals
$\sfrac{1}{n(n-1)} X_n(i,t) (X_n(i,t)-1) $.  So
\begin{align}
\Pr(B_n > t) &=
\sfrac{1}{n(n-1)} \Ex[\sum_i  X_n(i,t) (X_n(i,t)-1) ]\\\notag
&= \sfrac{1}{n(n-1)} \Ex[ Q_n(t) - n] \\\notag
&= e^{-t} \mbox{ by } \eqref{EQn} .
\end{align}

\section{Proof of  Lemma \ref{LM} by Mellin inversion}
\label{invert}

\begin{proof}[Proof of \refL{LM}]
 We begin by noting that the Mellin transform $1/\bigpar{\psi(s)-\psi(1)}$
in \eqref{m1}  extends to 
a meromorphic function in the entire
complex plane. The poles are the roots of
\begin{align}\label{m2}
  \psi(s)=\psi(1).
\end{align}
Obviously, $s_0:=1$ is a pole. Its residue is
\begin{align}\label{m3}
  \Res_{s=1}\frac{1}{\psi(s)-\psi(1)}
=\frac{1}{\psi'(1)}=\frac{6}{\pi^2},
\end{align}
using the well known formula $\psi'(1)=\pi^2/6$ 
\cite[5.4.12]{NIST}
(see also \eqref{psi'} below).
As shown in Lemma \ref{Lpsi} below
the other poles are real and negative, and thus can be ordered
$0>s_1>s_2>\dots$. In particular,
there are no other poles in the half-plane
$\Re s>s_1$, with $s_1\doteq -0.567$.

We cannot immediately use standard results on Mellin inversion\footnote{And we cannot use
\cite[Theorem 2(ii)]{FGD} since we do not know that $\gU$ has a density
that is locally of bounded variation.}
(as  in \cite[Theorem 2(i)]{FGD})
because the Mellin transform in \eqref{m1} decreases too slowly 
as $\Im s\to\pm\infty$ to be integrable on a vertical line $\Re s=c$.
In fact, Stirling's formula implies (see e.g.\
\cite[5.11.2]{NIST}) that 
\begin{align}\label{m4}
  \psi(s) = \log s + o(1) = \log|s| + O(1) = \log|\Im s|+O(1)
\end{align}
as $\Im s\to\infty$ with $s$ in, for example, any half-plane $\Re s\ge c$.

We overcome this problem by differentiating the Mellin transform, but we
first subtract the leading term corresponding to the pole at 1.
Since $\gU$ is an infinite measure, we first replace it by 
$\nu$ defined by 
$\dd\nu(x)=x\dd\gU(x)$; note that $\nu$ is also a measure on $\oio$, and 
taking $s=2$ in \eqref{m1} shows that $\nu$ is a finite measure.

Next, define $\nu_0$ as the measure $(6/\pi^2)\dd x$ on $\oio$, and let $\nux$
be the (finite) signed measure $\nu-\nu_0$. 
Then $\nux$ has the Mellin transform, by \eqref{m1},
\begin{align}\label{m5}
\tnux(s)
&:=
  \intoi x^{s-1}\dd\nux(x)
=  \intoi x^{s}\dd\gU(x)
- \frac{6}{\pi^2} \intoi x^{s-1}\dd x
\\\notag&\phantom:
=\frac{1}{\psi(s+1)-\psi(1)}-\frac{6}{\pi^2s},
\qquad \Re s>0.
\end{align}
We may here differentiate under the integral sign, which gives
\begin{align}\label{m6a}
\tnux'(s)
&:=
  \intoi(\log x) x^{s-1}\dd\nux(x)
\\ \label{m6b}&\phantom:
=
-\frac{\psi'(s+1)}{(\psi(s+1)-\psi(1))^2}
+\frac{6}{\pi^2s^2},
\qquad \Re s>0.
\end{align}
The Mellin transform $\tnux(s)$ extends to a meromorphic function in $\bbC$ with
(simple) poles $(s_i-1)_1^\infty$; note that there is no pole at $s_0-1=0$,
since the residues there of the two terms in \eqref{m5} cancel by \eqref{m3}.
Furthermore, the formula \eqref{m6b} for $\tnux'(s)$ holds
for all $s$ (although the integral in \eqref{m6a} diverges unless $\Re s>0$).

For any real $c$ we have, on the vertical line $\Re s=c$, as $\Im s\to\pm\infty$,
that $\psi(s)\sim \log|s|$ by \eqref{m4}, and also, by careful differentiation of
\eqref{m4} or by \cite[5.15.8]{NIST},
that $\psi'(s)\sim s\qw$.
It follows from \eqref{m6b} that $\tnux'(s)=O(|s|\qw\log\qww|s|)$ on the
line $\Re s=c$,  
for $|\Im s|\ge2$ say,
and thus
$\tnux'$ is integrable on this line
unless $c$ is one of the poles $s_i-1$.
In particular, taking $c=1$ and thus $s=1+u\ii$ ($u\in\bbR$),
we see that the function 
\begin{align}\label{mma}
\tnux'(1+\ii u)=\intoi x^{\ii u}\log (x) \dd\nux(x)  
\end{align}
is integrable. 
The change of variables $x=e^{-y}$ shows that the function
\eqref{mma} 
is the Fourier transform of the signed measure 
on $\bbR_+$ that corresponds to $\log(x)\dd\nux(x)$.
This measure 
on $\bbR_+$
is thus a finite signed measure with integrable Fourier transform,
which implies that it is absolutely continuous with a continuous density.
Reversing the change of variables, we thus see that the signed measure
$\log(x)\dd\nux(x)$ is absolutely continuous with a continuous density on
$(0,1)$. Moreover, denoting this density by $h(x)$, we obtain the standard
inversion formula for the Mellin transform
\cite[Theorem 2(i)]{FGD}, \cite[1.14.35]{NIST}:
\begin{align}\label{m8}
 h(x)=\frac{1}{2\pi\ii}\int_{c-\infty\ii}^{c+\infty\ii} x^{-s}\tnux'(s)\dd s,
\qquad x>0,
\end{align}
with $c=1$.
Furthermore, the integrand in \eqref{m8} is analytic in the half-plane $\Re
s>s_1-1$, and the estimates above of $\psi(s)$ and $\psi'(s)$ are
uniform for $\Re s$ in any compact interval and, say, $|\Im s|\ge2$. 
Consequently, we may shift the line of integration in \eqref{m8} to any
$c>s_1-1$. Taking absolute values in \eqref{m8}, and recalling that
$\tnux'(s)$ is integrable on the line, then yields
\begin{align}\label{m9}
  h(x) = O\bigpar{x^{-c}}
\end{align}
for any $c>s_1-1$.

Reversing the transformations above, we see that $\nux$ has the density
$(\log x)\qw h(x)$, and thus $\nu$ has the density
$(\log x)\qw h(x)+ 6/\pi^2$, and, finally, that $\gU$ has the density
\begin{align}\label{m10}
\gu(x):=\frac{\ddx\gU}{\ddx x}
=\frac{1}{x}\frac{\ddx\nu}{\ddx x}
= \frac{6}{\pi^2x}+\frac{1}{x\log x} h(x),
\qquad 0<x<1.
\end{align}
Furthermore, \eqref{m10} and \eqref{m9} have the form of the claimed
estimate \eqref{lmb},
although with the weaker error term $O(x^{-s_1-\eps}|\log x|\qw)$ 
for any $\eps>0$.

To obtain the (stronger) claimed error term, we note that the residue of $\nux(s)$
at $s_1-1$ is $r_1:=1/\psi'(s_1)$. Let $\nu_1$ be the measure
$r_1x^{1-s_1}\dd x$ on $\oio$; then $\nu_1$ has Mellin transform
\begin{align}\label{m11}
  \widetilde{\nu_1}(s)=r_1\intoi x^{s-1}x^{1-s_1}\dd x
=\frac{r_1}{s+1-s_1},
\qquad \Re s> s_1-1.
\end{align}
It follows from \eqref{m5} and \eqref{m11}
that the signed measure $\nu-\nu_0-\nu_1=\nux-\nu_1$ has the Mellin
transform 
\begin{align}\label{m12}
\frac{1}{\psi(s+1)-\psi(1)}-\frac{6}{\pi^2s}-\frac{r_1}{s+1-s_1},
\end{align}
which is an analytic function in the half plane $\Re s > s_2-1$.
Hence, the same argument as above yields the estimate
\begin{align}\label{lm1}
\gu(x)
= \frac{6}{\pi^2x}+
\frac{1}{\psi'(s_1)} x^{-s_1}
+O\bigpar{x^{-s_2-\eps}|\log x|\qw},
\qquad x\downarrow 0,
\end{align}
for any $\eps>0$, which in particular yields \eqref{lmb}.
\end{proof}

%

\subsection{A lemma on the digamma function}\label{App:digamma}

\begin{Lemma}\label{Lpsi}
  The roots of the equation $\psi(s)=\psi(1)$ 
are all real and can be enumerated in
  decreasing order as $s_0=1>s_1>s_2>\dots$, with $s_i\in(-i,-(i-1))$ for
  $i\ge1$. 
Numerically, $s_1\doteq-0.567$ and 
$s_2\doteq-1.628$.
\end{Lemma}

\begin{proof}
Recall that $\psi(s)$ is a meromorphic function of $s$, with poles at
$0,-1,-2,\dots$. For any other complex $s$ we have the standard formulas
\cite[5.7.6 and 5.15.1]{NIST}
\begin{align}\label{psi}
  \psi(s)&=-\gamma+\sumko\Bigpar{\frac{1}{k+1}-\frac{1}{k+s}}
,\\\label{psi'}
\psi'(s)&=\sumko\frac{1}{(k+s)^2}
.\end{align}

If $\Im s>0$, then $\Im(1/(k+s))<0$ for all $k$ and thus \eqref{psi}
  implies
$\Im\psi(s)>0$.
Similarly, if $\Im s<0$, then $\Im\psi(s)<0$.
Consequently, all roots of $\psi(s)=\psi(1)$ are real.

For real $s$, \eqref{psi'} shows that $\psi'(s)>0$. 
We can write $\bbR\setminus\set{\text{the poles}}=
\bigcup_{i=0}^\infty I_i$ with 
$I_0:=(0,\infty)$ and $I_i:=(-i,-(i-1))$ for $i\ge1$;
it then follows that $\psi(s)$ is strictly increasing in each interval
$I_i$.
Moreover, by \eqref{psi} (or general principles),
at the poles we have the limits $\psi(-i-0)=+\infty$ and
$\psi(-i+0)=-\infty$ ($i\ge0$). Consequently, $\psi(s)=\psi(1)$
has exactly one root $s_i$ in each $I_i$, and obviously the positive root is
$s_0=1$. (See also the graph of $\psi(s)$ in \cite[Figure 5.3.3]{NIST}.)
%
\end{proof}




\end{document}